\documentclass[a4paper,12pt]{article}
\usepackage{tgpagella,eulervm}
\usepackage{multicol} 
\usepackage{tablefootnote}
\usepackage[english]{babel}
\usepackage[normalem]{ulem}
\usepackage[utf8x]{inputenc}
\usepackage[T1]{fontenc}
\usepackage{amsmath}
\usepackage{amssymb}
\usepackage{mathrsfs}
\usepackage{float}
\usepackage{wrapfig}
\usepackage{amsthm}
\usepackage{tikz-cd}
\usepackage{tikz}
\usepackage{tikz-3dplot} 
\usetikzlibrary{shapes.geometric, graphs, arrows, arrows.meta, bending, positioning}

\usepackage{savesym}
\savesymbol{faCompress}
\usepackage{fontawesome}

\usepackage[colorlinks=true, allcolors=darkblue]{hyperref}
\usepackage{nomencl}
\usepackage{multicol}
\usepackage{makecell}
\usepackage{multirow}
\usepackage{caption}

\usepackage[many]{tcolorbox}
\usetikzlibrary{decorations.pathreplacing}

\newtcolorbox{leftbrace}{%
    enhanced jigsaw, 
    breakable, 
    frame hidden, 
    overlay={%
        \draw [
            decoration={brace,amplitude=0.5em},
            decorate,
            ultra thick,
        ]
        (frame.north west)--(frame.south west);
    },
    parbox=false,
}

\renewcommand{\nomgroup}[1]{%
\ifthenelse{\equal{#1}{B}}{\item[\textbf{Invariants}]}{%
\ifthenelse{\equal{#1}{C}}{\item[\textbf{Functions}]}{%
\ifthenelse{\equal{#1}{D}}{\item[\textbf{Other}]}{%
\ifthenelse{\equal{#1}{A}}{\item[\textbf{Collections of subposets}]}{%
}}}}}
\makenomenclature

\usepackage{thm-restate}
\usepackage{enumerate}
\usepackage[inline]{enumitem}
\usepackage{xcolor}
\usepackage{colortbl} 
\usepackage{bbold}
\usepackage{pdflscape} 
\usepackage[margin=0.95in]{geometry}
\usepackage{authblk}
\usepackage{framed}

\numberwithin{equation}{section}

\newcommand\freefootnote[1]{%
  \let\thefootnote\relax%
  \footnotetext{#1}%
  \let\thefootnote\svthefootnote%
}

\usepackage{amsmath}
\usepackage{graphicx}
\usepackage{url}
\usepackage[colorinlistoftodos]{todonotes}

\definecolor{darkblue}{rgb}{0.0, 0.0, 0.8}
\definecolor{darkred}{rgb}{0.8, 0.0, 0.0}
\definecolor{darkgreen}{rgb}{0.0, 0.65, 0.0}

\DeclareMathAlphabet{\mathcal}{OMS}{cmsy}{m}{n}

\newcommand{\ok}[1]{{#1}}

\newcommand{\white}[1] {{\textcolor{white}{#1}}}

\newcommand{\contributionn}{}
\newcommand{\pfin}{pointwisely finite}
\newcommand{\invertible}{M\"obius invertible{}}
\newcommand{\convolvable}{convolvable}

\newcommand{\lin}{\mathrm{Span}}
\newcommand{\abs}[1]{\lvert{#1}\rvert}
\newcommand{\pmo}{$\catP$-module}

\newcommand{\catQ}{\mathrm{Q}}
\newcommand{\R}{\mathbb{R}}
\newcommand{\N}{\mathbb{N}}
\newcommand{\Z}{\mathbb{Z}}
\newcommand{\kf}{k}
\newcommand{\im}{\mathrm{im}}

\newcommand{\dE}{d_\mathrm{E}}

\newcommand{\cvec}{\mathbf{vec}}
\newcommand{\dI}{d_\mathrm{I}}
\newcommand{\inv}{^{-1}}
\newcommand{\id}{\mathrm{id}}
\newcommand{\dgm}{\mathrm{\mathrm{dgm}}}
\newcommand{\rk}{\mathrm{rk}}
\newcommand{\eps}{\epsilon}
\newcommand{\catP}{\mathrm{P}}
\newcommand{\catS}{\mathrm{S}}
\newcommand{\rank}{\mathrm{rk}}
\newcommand{\mob}{\zeta}
\newcommand{\Int}{\mathrm{Int}}
\newcommand{\fint}{\mathrm{int}}
\newcommand{\con}{\mathrm{con}}
\newcommand{\Con}{\mathrm{Con}}
\newcommand{\Jcal}{\mathcal{J}}
\newcommand{\Ical}{\mathcal{I}}
\newcommand{\Scal}{\mathcal{S}}
\newcommand{\Tcal}{\mathcal{T}}
\newcommand{\Rcal}{\mathcal{R}}
\newcommand{\Pcal}{\mathcal{P}}
\newcommand{\Lfrak}{\mathfrak{L}}
\newcommand{\barc}{\mathrm{barc}}

\newcommand{\ZZ}{\mathbb{ZZ}}
\newcommand{\SZZ}{\mathrm{S}\mathbb{ZZ}}
\newcommand{\overmax}{\max_{\ZZ}}
\newcommand{\overmin}{\min_{\ZZ}}
\newcommand{\bv}{\mathbf{v}}
\newcommand{\bw}{\mathbf{w}}

\newcommand{\seg}{\mathrm{Seg}}
\newcommand{\bareps}{\vec{\epsilon}}
\newcommand{\Q}{\mathbb{Q}}

\newcommand{\mult}{\mathrm{mult}}
\newcommand{\Zplus}{\Z_{\geq 0}}

\newcommand{\drm}{\mathrm{d}}

\newcommand{\mmod}{\mathrm{mod}\ }
\newcommand{\Hom}{\mathrm{Hom}}

\tikzcdset{scale cd/.style={every label/.append style={scale=#1},
    cells={nodes={scale=#1}}}}

\newtheorem{theorem}{Theorem}

\newtheorem{proposition}{Proposition}[section]
\newtheorem{oldtheorem}[proposition]{Theorem}
\newtheorem{lemma}[proposition]{Lemma}
\newtheorem{corollary}[proposition]{Corollary}

\theoremstyle{definition}
\newtheorem{definition}[proposition]{Definition}
\newtheorem{example}[proposition]{Example}
\newtheorem{convention}[proposition]{Convention}

\theoremstyle{remark}
\newtheorem{remark}[proposition]{Remark}

\newtheorem{claim}{Claim}

\title{The Generalized Rank Invariant:\\ M\"obius invertibility, Discriminating Power, and Connection to Other Invariants \freefootnote{\emph{2020 Mathematics Subject Classification}:\ \ 55N31, 55U99} \freefootnote{\emph{Keywords}: M\"obius inversion, persistence diagrams, multi-parameter persistent homology, topological data analysis}}

\date{}

\author[1]{Nathaniel Clause} 

\author[2]{Woojin Kim} 

\author[1]{Facundo M\'{e}moli}

\affil[1]{The Ohio State University, Columbus OH \thanks{\texttt{facundo.memoli@gmail.com}} \thanks{\texttt{clause.15@osu.edu}}}

\affil[2]{KAIST, Daejeon, South Korea 	\thanks{\texttt{woojin.kim.math@gmail.com}}}

\begin{document}
\maketitle

\begin{abstract}

\ok{
In addition to inherent computational challenges, the absence of a canonical method for quantifying `persistence' in multi-parameter persistent homology remains a hurdle in its application.} 

One of the best known 
 quantifications of persistence for multi-parameter persistent homology is the rank invariant, which has recently evolved into the generalized rank invariant (GRI) by naturally extending its domain. This extension enables us to quantify persistence across a broader range of regions in the indexing poset compared to the rank invariant.  However, the size of the domain of the GRI is generally formidable, making it desirable to restrict its domain to a more manageable subset for computational purposes. The foremost questions regarding such a restriction of the domain are:
 \textbf{(1)} 
How to restrict, if possible, the domain of the GRI without any loss of information? 
 \textbf{(2)} When can we more compactly encode the GRI 
 as a `persistence diagram'?
 \textbf{(3)} What is the trade-off  between
computational efficiency and the discriminating power of the GRI as the amount of the restriction on the domain varies? 
\ok{\textbf{(4)} What proxies exist for persistence diagrams in the multi-parameter setting that can be derived from the GRI?}
To address \ok{the first three} questions, we 
generalize and axiomatize the classic fundamental lemma of persistent homology via the notion of \emph{M\"obius invertibility} of the GRI which we propose. This extension also contextualizes known results regarding the (generalized) rank invariant within the classical theory of M\"obius inversion. We conduct a comprehensive comparison between M\"obius invertibility and other existing concepts related to the structural simplicity of persistence modules, such as Miller's notion of tameness. During this investigation, we identify an example of a persistence module whose GRI is not M\"obius invertible. 

\ok{We address the fourth question through the notion of motivic invariants. We demonstrate that many invariants from the literature can be both derived from the GRI and recast as motivic invariants.}

\end{abstract}

\setlength{\nomitemsep}{-\parsep}

\nomenclature[A, 1]{\(\catP\)}{A poset \(\catP=(\catP,\leq)\), regarded as a category.}

\nomenclature[A, 1.5]{\([n]\)}{The totally ordered set $\{1<\ldots<n\}$.}

\nomenclature[A, 2]{\(\seg(\catP)\)}{Poset of segments of the poset \(\catP\), ordered by containment (Equation \ref{eq:segment}).}

\nomenclature[A, 3]{\(\Int(\catP)\)}{Poset of intervals of the poset \(\catP\), ordered by containment (Definition \ref{def:intervals}).}

\nomenclature[A, 4]{$\fint(\catP)$}{Poset of finite intervals of the poset $\catP$, ordered by containment (pg. \pageref{ex:mn subposet}).} 

\nomenclature[A, 4.5]{$\Int_{m,n}(\catP)$}{Poset of intervals of $\catP$ with at most $m$ minimal points and $n$ maximal points (pg. \pageref{eq:intmn}).}

\nomenclature[A, 1]{\(\Con(\catP)\)}{Poset of connected subposets of \(\catP\), ordered by containment (pg. \pageref{nom:Con(P)}).}

\nomenclature[A, 6]{$\widehat{\Ical}$}{The limit completion of $\Ical\subset \Int(\catP)$ (Equation (\ref{eq:limit intervals})).}

\nomenclature[B, 1]{RI}{Rank invariant (Definition \ref{def:generalized rank invariant}).}

\nomenclature[B, 2]{GRI}{Generalized rank invariant (Definition {\ref{def:generalized rank invariant}}).}

\nomenclature[B, 3]{$\rk(M)$}{The (generalized) rank of $M$ (pg. \pageref{nom:generalized rank}).}

\nomenclature[B, 4]{$\rk_M^\Ical$}{The generalized rank invariant of $M$ over $\Ical\subset\Con(\catP)$ (Definition \ref{def:generalized rank invariant}).}

\nomenclature[B, 5]{$\rk_M^{\Int}$}{The generalized rank invariant of $M$ over $\Int(\catP)$ (Definition \ref{def:generalized rank invariant}).
}

\nomenclature[B, 6]{$\dgm_M^\Ical$}{The generalized persistence diagram of $M$ over $\Ical$ (Definition \ref{def:GPD over I}).}

\nomenclature[B, 7]{ZIB}{The zigzag-path-indexed-barcode (Definition \ref{def:zigzag-path-indexed barcode general poset}).}

\nomenclature[B, 8]{$M_{S\Z\Z}$}{The Zigzag-path-indexed-barcode over simple paths of $M:\catP\to \cvec$ (Definition \ref{def:zigzag-path-indexed barcode}).}

\nomenclature[C, 1]{$I(\catQ,\kf)$}{The incidence algebra of $\catQ$ over $\kf$ (pg. \pageref{nom:incidence algebra}).}

\nomenclature[C, 2]{$\delta_\catQ$}{The Dirac delta function (Equation (\ref{eq:delta function})).}

\nomenclature[C, 3]{$\zeta_\catQ$}{The zeta function (Equation (\ref{eq:zeta function})).}

\nomenclature[C, 4]{$\mu_\catQ$}{The M\"obius function (Equation (\ref{eq:mobius})).}

\nomenclature[D]{$\barc(M)$}{The barcode of $M$ (pg. \pageref{nom:barcode}).}

\nomenclature[D]{$\kf_I$}{Interval module for the interval $I$ (Equation (\ref{eq:interval module})).}

\nomenclature[D]{$\Gamma$}{A path in $\Z^2$ (pg. \pageref{nom: path}).}

\nomenclature[D]{$\kf^\catQ$}{The vector space of functions $f:\catQ\to \kf$ (pg. \pageref{nom: kq}).}

\nomenclature[D]{$p^\downarrow$}{The set of elements in $\catP$ less than or equal to $p$ (pg. \pageref{nom: down p}).}

\nomenclature[D]{$p^\uparrow$}{The set of elements in $\catP$ greater than or equal to $p$ (pg. \pageref{nom: up p}).}

\tableofcontents

\section{Introduction}\label{section:introduction}
One-parameter persistent homology is a central concept in topological data analysis (TDA) with a wide range of applications \cite{bleher2021topological,cang2015topological,carlsson2021topological,chan2013topology,chowdhury2018importance,dabaghian2012topological,PhysRevC.106.064912,lee2018high,stolz2023relational}. Under fairly general assumptions, a one-parameter persistence module $M:\R\to \cvec$ is completely characterized by its so-called \emph{persistence diagram}, which provides a \ok{canonical}, compact and interpretable summary of $M$. 

However, one-parameter persistent homology has well-known struggles when dealing with noise and outliers, motivating a surge in research in recent years on the more robust \emph{multi-parameter} persistent homology, i.e. the setting in which we contend with \ok{\emph{multi-parameter} persistence modules $M:\R^d\to \cvec$. Unfortunately, in contrast with the case of one-parameter persistence modules (where persistence diagrams are full invariants thereof),  there is no complete and simultaneously discrete invariant for multi-parameter persistence modules \cite{carlsson2009theory}.}
Accordingly, many discrete and necessarily incomplete invariants have been studied, e.g. \cite{asashiba2023approximation,asashiba2019approximation,blanchette2021homological,botnan2021signed,chacholski2022effective,harrington2019stratifying,landi2018rank,lesnick2015interactive,miller2020homological,oudot2024stability}. The central hurdle preventing widespread use of \ok{many of} these invariants in practice is their computational complexity.

In this work, we focus on a particular class of invariants for multi-parameter 
persistence modules, or more generally for \emph{$\catP$-modules} $M:\catP\rightarrow \cvec$, where $\catP$ is an arbitrary poset. For these invariants, we study the tension that exists  between their  computational complexity and their inherent discriminating power.

One of these invariants that is intimately tied to this work is the classical \emph{rank invariant} (RI), which encodes the rank of all linear maps present in a given $\catP$-module; see  \cite{carlsson2009theory,puuska2020erosion}. 
For one-parameter persistence modules, the RI contains equivalent information to the classical persistence diagram \cite{abeasis1981geometry,carlsson2009theory}, making it a natural notion to import into the multi-parameter setting.
The RI allows us to quantify persistence over \emph{segments} in the indexing poset (a segment is the set of points lying between a given pair of comparable points), and it is a lossless invariant when applied to so-called rectangle-decomposable modules \cite{botnan2022rectangle}.
As a downside, the RI \emph{only} allows us to quantify persistence over segments, and in the multi-parameter setting there \ok{typically are many}  natural non-segment like subsets of the indexing poset $\catP$ over which we desire to quantify persistence.

Motivated by this, the RI has recently evolved into the \emph{generalized rank invariant} (GRI)  by extending the domain of the RI from the set $\seg(\catP)$ of segments of the indexing poset $\catP$ to the set $\Int(\catP)$ of \emph{intervals} of $\catP$ or to the even larger set $\Con(\catP)$ of \emph{connected subposets} of $\catP$ \cite{kim2021generalized}.
By recording the rank of the canonical limit-to-colimit map of the diagram over each element of $\Int(\catP)$ (or of $\Con(\catP)$), the GRI  quantifies persistence across a broader range of regions in the indexing poset compared to the RI and, therefore, the GRI acquires more discriminating power than the RI.

However, the size of $\Int(\catP)$, let alone that of the larger $\Con(\catP)$, are generally formidable, causing a bottleneck for computing the GRI over the entire $\Int(\catP)$ or $\Con(\catP)$. 
For example, if $\catP$ is the 2-dimensional 10-by-10 grid (which serves as a much simplified setup for 2-parameter persistent homology), then $\Int(\catP)$ comprises 1,497,925,315 elements \cite[Theorem 31]{asashiba2022interval}. 
Therefore, \textbf{restricting the domain of the GRI} to a more manageable subset is desirable, \ok{if not downright necessary}, 
for computational purposes. 

\medskip
Questions regarding the magnitude and consequent effect of this restriction can be grouped into two paradigms: 
\begin{itemize}
\item[\faDiamond] In the \emph{lossless paradigm}, \ok{we ask} when we can, and if so how to, restrict the GRI's domain without losing information \ok{and efficiently represent the information.}
\item[\faCompress] In the \emph{lossy paradigm},\footnote{I.e. when we `compress'.} \ok{we ask} how much information 
 a restricted GRI retains. 
 \end{itemize}
Within these two categories, the foremost questions regarding the restriction of the GRI's domain are:
    
\begin{enumerate}
\item[(\faDiamond)] \hypertarget{q:1}{\textbf{Question 1.}} 
How to restrict, if possible, the domain of the GRI
without any loss of information? 

\item[(\faDiamond)] \hypertarget{q:2}{\textbf{Question 2.}} 
Under what conditions can we more compactly encode the GRI as a 'persistence diagram,' even when the indexing poset $\catP$ is not discrete?

\item[(\faCompress)]\hypertarget{q:3}{\textbf{Question 3.}} In the lossy regime, what is the trade-off  between
computational efficiency and the discriminating power of the GRI as the amount of the restriction varies? 
\end{enumerate}
\ok{Finally, we move beyond considering the GRI itself and inquire as to what other invariants exist for quantifying persistence which are downstream from the GRI: }
\begin{enumerate}
\item[(\faCompress)] \hypertarget{q:4}{\textbf{Question 4.}}   What proxies exist for persistence diagrams in the multi-parameter setting that can be derived from the GRI?
\end{enumerate}

Our work embarks on addressing these questions. 

\subsection*{Contributions.} 

In order to tackle \hyperlink{q:1}{\textbf{Question 1}} and \hyperlink{q:2}{\textbf{Question 2}}, we introduce the concept of \emph{M\"obius invertibility} of the GRI which is obtained by axiomatizing the classical fundamental lemma of persistent homology \cite[Section VII]{edelsbrunner2008computational}. For persistence modules over finite posets,
M\"obius invertibility of the GRI is always guaranteed, and its M\"obius inverse is known as the \emph{generalized persistence diagram} \cite{kim2021generalized} (see also \cite{asashiba2019approximation}) --- a compact encoding of the GRI.\footnote{In fact, the generalized persistence diagram is well-defined for persistence modules over a slightly more general class of indexing posets; see \cite[Section 3]{kim2021generalized}.}
 
In addition, M\"obius invertibility is useful to contextualize known theorems regarding the RI or GRI within the framework of classical M\"obius inversion theory, which also sometimes enables us to simplify existing proofs or strengthen statements of existing theorems.
Furthermore, M\"obius invertibility is closely linked to several invariants of multi-parameter persistence \cite{blanchette2021homological,botnan2021signed, gulen2022galois,lesnick2015interactive}, as we elucidate in this work. \ok{We also observe that M\"obius invertibility of the GRI of a $\catP$-module $M$ is connected to the structural simplicity of $M$ in a sense akin to Miller’s notion of tameness \cite{miller2020homological}.  In relation to this, we conduct a thorough comparison between tameness and M\"obius invertibility.}

To tackle \hyperlink{q:3}{\textbf{Question 3}}, we study the \emph{completeness} properties of the GRI.
Namely, fixing any collection of intervals  $\Ical$ in the indexing poset $ \catP$,
we characterize the collection of $\catP$-modules on which the GRI restricted to $\Ical$ is a complete invariant.
Further, if the GRI restricted to $\Ical$ is not a complete invariant on an arbitrary fixed collection of modules, we suitably quantify this lack of completeness. 

We tackle \hyperlink{q:4}{\textbf{Question 4}} by considering the concept of \emph{motivic invariants}, which is based on the idea of `probing' the indexing poset via a simpler poset called a \emph{motif}.
We demonstrate that many invariants from the literature can be derived from the GRI, and showcase how they fit into the framework of motivic invariants: see Table \ref{table:comparison}, Examples \ref{ex:gri is motivic}, \ref{ex:table 1 motivic}, and Remark \ref{rem:zib as motivic}.
One specific motivic invariant we focus on is the Zigzag-path-Indexed Barcode (ZIB, Definition \ref{def:zigzag-path-indexed barcode general poset}), which consists of the restrictions of a given persistence module along all zigzag paths in its indexing poset.

A recent finding that the GRI of a $\Z^2$-module $M$ is determined by the collection of all zigzag persistence modules that arise from  restricting  $M$ to certain paths in $\Z^2$ \cite{dey2022computing}
indicates that 
identifying the \emph{minimal} set of paths over which  zigzag barcodes should be computed  to determine the (restricted) GRI 
is important in order to be able to: 
\begin{enumerate}
\item[(i)]  exploit existing efficient zigzag persistence algorithms \cite{carlsson2009zigzag,dey2022fast,milosavljevic2011zigzag} for computing the (restricted) GRI of $\Z^2$-modules as well as for computing other related invariants \cite{amiot2024invariants,asashiba2023approximation,blanchette2021homological,botnan2021signed,chacholski2022effective}. 
\item[(ii)]  
examine how much is gained in terms of efficiency when computing the  restricted GRI of a \(\mathbb{Z}^2\)-module as opposed  to computing the entire GRI.
 Also, from a different perspective, by investigating  
 the extent to which the ZIB can be recovered from the (restricted) GRI, we can ascertain the discriminating power of the (restricted) GRI.
 \end{enumerate}
 Therefore, clarifying the connection between the GRI and the ZIB for $\Z^2$-modules also aids in addressing \hyperlink{q:3}{\textbf{Question 3}} for $\Z^2$-modules.

A priori, there is the possibility that the class of zigzag paths in $\Z^2$ composing the the domain of the ZIB can be further restricted compared to the class considered in \cite{dey2022computing}, while still being able to recover the entire GRI.
One natural candidate for such a class is the collection of \emph{simple} paths, i.e. paths with no repeated vertices (Definition \ref{def:zigzag-path-indexed barcode}).
Motivated by this, we elucidate the precise relationship between the ZIB over simple paths and the GRI over $\Int(\Z^2)$, the set of all intervals of $\Z^2$.

\paragraph{Details about Questions 1, 2 and 3.}

\begin{enumerate}[label=\textbf{(\Roman*)},wide, labelindent=0pt]
    \item Our results related to \hyperlink{q:1}{\textbf{Question 1}} and \hyperlink{q:2}{\textbf{Question 2}} are the following.\label{item:contribution1}
    
    \begin{enumerate}[leftmargin=*,label=(\roman*)]   
    \item 
   The notion of M\"obius invertibility of the GRI encompasses \emph{generalized persistence diagrams \cite{kim2021generalized}} and \emph{rank decompositions} \cite{botnan2021signed}, and properly adapts \emph{M\"obius inversion of a constructible function} \cite{gulen2022galois}: see \textbf{Section \ref{sec:relative to a dictionary}}.
     \item We show that if {a GRI admits a rank decomposition, then even without assuming that $\Int(\catP)$ is locally finite, we have that the minimal rank decomposition of the GRI is obtained from M\"obius inversion over a specific locally finite $\Ical\subset \Int(\catP)$:} 
    see \textbf{Theorem \ref{thm:rk decomposition is GPD}}. \label{item:contribution2a}
    \item We clarify the relationship among the concepts pertaining to structural simplicity of persistence modules, including M\"obius invertibility of the GRI, tameness \cite{miller2020homological}, constructibility \cite{gulen2022galois,patel2018generalized}, interval decomposability \cite{botnan2018algebraic}, finite presentability: see \textbf{Figure \ref{fig:tameness + MI visual}}.
    \item We identify a 2-parameter persistence module whose GRI over intervals is not M\"obius invertible (which implies that the GRI admits no rank decomposition). To the best of our knowledge, this is the first example of its kind in the existing literature: see \textbf{Theorem \ref{thm:tame does not imply Int-GRI invertible}} and \textbf{Remark \ref{rem:implication of the counterexample}.}

      \item 
    We establish a number of sufficient conditions for M\"obius invertibility of the GRI. For instance, our results imply that the GRI \ok{over intervals} of any finitely presentable $\R^d$-module is M\"obius invertible.
    Furthermore, in this scenario, we construct a finite poset of intervals to which the GRI \ok{over intervals} can be restricted \emph{without any loss of information}. 
    That is, the M\"obius inversion of the GRI over this \emph{finite} poset encodes the entire GRI over the uncountable set of intervals in $\R^d$: \ok{see \textbf{Theorem \ref{thm:sufficient conditions for invertibility} and \textbf{Proposition \ref{prop:finitely presentable invertiver over intmn}}.} \ok{In another instance,  we establish the M\"obius invertibility of $\catP$-modules satisfying certain assumptions akin to \emph{tameness}  \cite{miller2020homological}: see \textbf{Theorem \ref{thm:pull-back GPD}}}.}
    
    All these results shed light on computational aspects of the GRI and also on its compact encoding as a ‘persistence diagram’, which provides a concrete answer to \hyperlink{q:1}{\textbf{Question 1}}.     
    \end{enumerate}

    \item \label{item:contribution2} We address \hyperlink{q:3}{\textbf{Question 3}} as follows.
\begin{enumerate}[leftmargin=*,label=(\roman*)]
\item  
Let $\Ical\subset \Int(\catP)$ be any collection of intervals.
We use the M\"obius inversion formula to prove that the GRI over $\Ical$ is a complete invariant for $\catP$-modules whose indecomposable summands are interval modules supported on elements in $\Ical$: see \textbf{Theorem \ref{thm:mainthm}} (note: the statement itself was already known; see Remark \ref{rem:relationship with known result}). \label{item:contribution1a}

\item  We show that Theorem \ref{thm:mainthm} is, in a certain precise sense, optimal: see \textbf{Theorem \ref{thm:tightness}} and \textbf{Corollary \ref{cor:Ical and Jcal}}. \label{item:contribution1b}
 \item We go one step further and describe the equivalence classes of $\catP$-modules that have the same GRI over any $\Ical\subset \Int(\catP)$.
 Interestingly, we show this via exploiting the M\"obius inversion of functions \ok{which do not arise as GRIs of $\catP$-modules}: see \textbf{Theorem \ref{thm:coincidence of rank invariants}} and \textbf{Corollary \ref{cor:minimal pairs}}. \label{item:contribution1c}
 \end{enumerate}

In what follows, we give a special attention to the discriminating power and computational cost of the (restricted) GRI of 2-parameter persistence modules. Let $\fint(\Z^2)$ be the collection of all \emph{finite} intervals of $\Z^2$. 

\begin{enumerate}[resume,label=(\roman*),leftmargin=*]
\item  
  We show that the ZIB over simple paths and the GRI over $\fint(\Z^2)$ do not determine each other. As a corollary to this result and \cite[Theorem 24]{dey2022computing}, it follows that the ZIB over \emph{all} paths is a \emph{strictly} finer invariant than both the ZIB over \emph{simple} paths and the GRI over $\fint(\Z^2)$: see \textbf{Examples \ref{ex:Int does not recover simple ZZ}, \ref{ex:intvszz}} and \textbf{Figure \ref{fig:intro}}.  \label{item:contribution3a}
 
\item  Using M\"obius inversion, we show the ZIB over simple paths and the GRI over $\fint(\Z^2)$  
  estimate each other: see \textbf{Remark \ref{rem:approximating Int rank invariant via simple zz}} and \textbf{Proposition \ref{prop:approximating zz}}.  \label{item:contribution3b}

\item We show that the ZIB (or equivalently GRI) over zigzag paths of the form $\bullet \rightarrow \bullet \leftarrow \bullet$ and $\bullet \leftarrow \bullet \rightarrow \bullet$ is a strictly stronger invariant than the bigraded Betti numbers: see \textbf{Proposition \ref{prop:bigraded Bettis do not determine GRI over length-3 zigzags}}.
\end{enumerate}

\item A few subsidiary contributions follow:
We establish a stability theorem for GRIs and their restrictions 
-- a property that was not addressed in \cite{kim2021generalized}: see \textbf{Theorem \ref{thm:gristability}.}
A suitable reinterpretation of this theorem implies stability of ZIBs: 
see \textbf{Theorem \ref{thm:ZIB stability}}.  \label{item:contribution3}
Also, we analyze the trade-off between computational complexity of the erosion distance (Definition \ref{def:erosion distance Z2}) between restricted GRIs and the discriminating power of restricted GRIs as the domain of the restriction grows; \textbf{Remark \ref{rem:tension remark}}.

\end{enumerate}

\begin{figure}[h]
\captionsetup{singlelinecheck=off}
\caption{Implications and non-implications among the concepts pertaining to the structural simplicity of a persistence module $M$ over $\R^d$, as detailed in Sections \ref{sec:On structural simplicity of persistence modules} and \ref{sec:sufficient conditions}:
}
\begin{multicols}{3}
\begin{itemize}[topsep=0pt]
  \setlength\itemsep{0em}
\item[(1)] Proposition \ref{prop:constructible implies finitely presentable},
\item[(2)] Pf. of Thm. \ref{thm:sufficient conditions for invertibility} \ref{item:finitely presentable invertible},
\item[(3)] Theorem \ref{thm:sufficient conditions for invertibility} \ref{item:interval-decomposable},
\item[(4)] Theorem \ref{thm:sufficient conditions for invertibility} \ref{item:finitely encodable},
\item[(5)] By definition,
\item[(6)] Proposition \ref{prop:Mobius invertibility does not imply tameness},
\item[(7)] Theorem \ref{thm:tame does not imply Int-GRI invertible},
\item[(8)] Remark \ref{rem:monotonicity of mobius invertibility} \ref{item:monotonicity of mobius invertibility},
\item[(9)] Remark \ref{rem:implication of the counterexample} \ref{item:implication of the counterexample1}.
\end{itemize}
\end{multicols}

\begin{center}
    \includegraphics[width=0.9\linewidth]{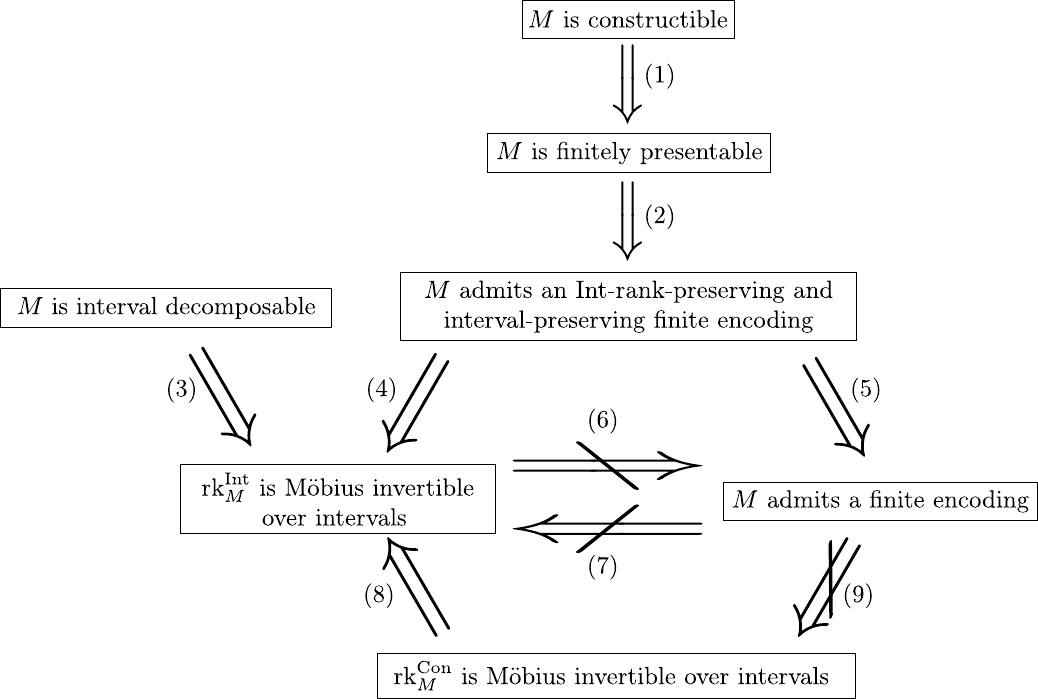}
\label{fig:tameness + MI visual}
\end{center}
\end{figure}

\newpage

\begin{landscape}

\begin{table}[h]
\caption{Comparison between the GRI 
 of a $\catP$-module $M$, denoted by $\rk_M$, and other invariants of $M$. Different choices of the indexing poset $\catP$, and domain \ok{$\Jcal\subset \Int(P)$} of $\rk_M$ are made. $\partial \rk_M^{\Jcal}$ denotes the M\"obius inversion of either $\rk_M^{\Jcal}$ or an appropriate restriction of $\rk_M^{\Jcal}$, which is essentially unique (Proposition \ref{prop:justifying the term mobius invertibility}). 
Section \ref{sec:comparison} is dedicated to providing detailed explanations of this table.
}\label{table:comparison}
\begin{center}
 \renewcommand\theadfont{}
\begin{tabular}{|l|l|l|l|l|l|}
\hline
\cellcolor[HTML]{EFEFEF} \# & \cellcolor[HTML]{EFEFEF}$\catP$ & \cellcolor[HTML]{EFEFEF}Domain $\Jcal$ of $\rk^\Jcal_M$                                                                  & \cellcolor[HTML]{EFEFEF}Name or Comparison &\cellcolor[HTML]{EFEFEF} $ \partial \rk_M^{\Jcal}$    & \cellcolor[HTML]{EFEFEF} \thead{$ \partial \rk_M^{\Jcal}$ \\always exists?} \\ \hline

1& Any &  $\{\{p\}:p\in \catP\}$                                                                 & $\Leftrightarrow$ Dimension function                                               & Dimension function & Yes   \\ \hline

2& Any & $\seg(\catP)$                                                                             & $\Leftrightarrow$ Rank Invariant              
& Signed barcode \cite{botnan2021signed} & ?   \\ \hline
3& $\R^d$ & $\displaystyle \bigsqcup \{\seg(\ell): \mbox{$\ell$ is a monotone line in $\R^d$}\}$                                                                               & $\Leftrightarrow$ Rank invariant       & Fibered barcode \cite{lesnick2015interactive} & Yes            
\\ \hline
4 & Any                         & 
  $\Int(\catP)$&          $\Int$-GRI   & \thead{GPD over $\Int(\catP)$\\ Rank decomposition} & No   \\ \hline

5 & Any                          & $\Con(\catP)$                        & $\Con$-GRI       & GPD over $\Con(\catP)$ & No \\ \hline 
6 & $[m]\times [n]$ &  $\Int(\catP)$                                                                 & \thead{tot-Compressed multiplicity\\ ($\Leftrightarrow \Int$-GRI)}                                               & $\delta^\mathrm{tot}_M$ & Yes   \\ \hline 
7 & $[m]\times [n]$ &  $\{\min(I)\cup\max(I):I\in \Int(\catP)\}$                                                              & ss-Compressed multiplicity                                               & $\delta^\mathrm{ss}_M$ & Yes   \\ \hline

8 & $\Z^2$                        & \thead{\{\mbox{Zigzag paths of the forms} \\ \mbox{$\bullet \rightarrow \bullet \leftarrow \bullet$ and $\bullet \leftarrow \bullet \rightarrow \bullet$ in $\Z^2$}\}}                                                                                 & $\Rightarrow$ Bigraded Betti numbers          & \thead{Barcodes over \\ zigzag paths of length 3}                                             & Yes \\ \hline
9 & $\Z^2$                        & $\{\mbox{Zigzag paths in $\Z^2$}\}$                                                                                  & $\Rightarrow$ $\Int$-GRI      &  \thead{Barcodes over \\ zigzag paths}                                   & Yes
  \\ \hline
10&  $\Z^2$                        & $\{\mbox{Simple zigzag paths in $\Z^2$}\}$                                                                                  & $\not\Leftrightarrow$ $\Int$-GRI         & \thead{Barcodes over\\ simple zigzag paths}                                             & Yes
  
  \\ \hline

\end{tabular}
\end{center}
\end{table}
\end{landscape}

\subsubsection*{Related work.} 
Patel first noted that in the one-parameter setting the persistent diagram can be defined as the M\"{o}bius inversion of the RI and thereby introduced the generalized persistence diagram \cite{patel2018generalized} of a constructible $\R$-indexed functor whose target can be different from the category of vector spaces. Patel's work became a motivation for the work by Kim and M\'emoli \cite{kim2021generalized} and the work by McCleary and Patel  \cite{mccleary2020bottleneck,mccleary2022edit}.
In particular, in \cite{kim2021generalized}, the generalized persistence diagram of a $\catP$-module is defined as the M\"{o}bius inversion of the GRI over $\Con(\catP)$, which can be viewed as a multiset of signed elements of $\Con(\catP)$ {(assuming that $\Con(\catP)$ is locally finite)}. 

\medskip
The following are also related to our work.
\begin{itemize}
\item Asashiba et al. also invoke M\"{o}bius inversion to devise methods for approximating a persistence module $M$ over a finite 2d-grid by an interval-decomposable module \cite{asashiba2019approximation}. One of their approximation methods yields an invariant that possesses the same amount of information as 
(the M\"obius inversion of) the GRI of $M$ over the intervals in the grid \cite[Remark 2.19]{kim2021bettis}. These two equivalent invariants naturally encode the bigraded Betti numbers of $M$ \cite{kim2021bettis}. \ok{See also \cite{asashiba2024interval,hiraoka2023refinement} for further developments of these ideas.}

\item \ok{Amiot et al. take a systematic approach to the invariant obtained by restricting the indexing poset of a given persistence module $M$ to a subposet $X\subset\catP$ of a finite representation type \cite{amiot2024invariants}.
Their work is related to ours in the sense that invariants from both works consider restrictions of the indexing poset of a given module.
Indeed, as the GRI is a complete invariant for interval-decomposable modules, whenever a choice of $X$ ensures that $M|_X$ is interval-decomposable, their work can be seen as studying the restriction of the GRI to $X$.}

\item \ok{Recently, connections between relative homological algebra and the RI/GRI have been explored} \cite{ asashiba2023approximation,blanchette2021homological, botnan2021signed, botnan2022bottleneck}.

\item In recent years, M\"obius inversion has been utilized alongside other invariants of persistence modules, including the graded rank function \cite{betthauser2022graded}, the birth-death function \cite{gulen2022galois,mccleary2022edit,morozov2021output}, the meta-rank \cite{clause2023meta}, the Hilbert function \cite{oudot2024stability},
the multi-rank \cite{thomas2019invariants}, the persistent cup length and persistent cup product \cite{memoli2022persistent}, and the Grassmanian persistence diagram \cite{gulen2023orthogonal}.  
\item Very recently, through an extension of the main result from \cite{dey2022computing}, in  \cite{dey2024computing} Dey et al. also propose to utilize  zigzag persistence for computing the rank of the limit-to-colimit map \cite{kim2021generalized} for persistence modules over posets more general than $\Z^2$. 
\end{itemize}

\subsubsection*{Organization.} In Section \ref{section:prelim} we review the requisite concepts about persistence modules, their decompositions, and introduce the concept of a motivic invariant, relating it to several existing invariants.
In Sections \ref{sec:mobius invertible Rk}, \ref{sec:main}, and \ref{sec:stability}, we establish the results outlined in Contributions \ref{item:contribution1}, \ref{item:contribution2}, and \ref{item:contribution3} respectively. 
In Section \ref{section:discussion}, we discuss open questions. 

\medskip
A table with the main nomenclature used in this paper is given below.

\subsubsection*{Acknowledgements.} This research was partially supported by the NSF through grants 
DMS-1723003, DMS-2301359, CCF-2310412, IIS-1901360, CCF-1740761, DMS-1547357 and by the BSF under grant  2020124. The authors thank anonymous reviewers for sharing their insight. \ok{We thank Profs. Hideto Asashiba and Emerson Escolar for bringing a gap in the proof of Theorem \ref{thm:gristabilityold} to our attention.} 

\printnomenclature[1in]

\section{Preliminaries}\label{section:prelim}

This section is organized as follows: in Section \ref{sec:persistence modules and their decomposition}, we review the concepts of persistence modules and their structures. In Section \ref{sec:mobius inversion}, we review the notion of the incidence algebra as well as the M\"obius inversion formula. In Section \ref{sec:GRI}, we recall the notions of the rank invariant, the generalized rank invariant, and their properties. In Section \ref{sec:GPD}, we recall the notion of the generalized persistence diagram. In Section \ref{sec:GPD and Rk decomposition}, we recall the notion of the rank decomposition and its connection with the notion of the generalized persistence diagram. 
In Section \ref{sec:motivic invariants}, we introduce the notion of a motivic invariant, provide historical context and motivation for the definition, and exemplify how certain invariants we discuss in this paper fit into this framework.
In Section \ref{sec:ZIB general}, we introduce a new motivic invariant called the zigzag-path-indexed barcode, which is based on slicing a $\catP$-module over zigzag paths in $\catP$.

\subsection{Persistence modules}\label{sec:persistence modules and their decomposition}

Throughout this paper, $\catP=(\catP,\leq)$ is a poset, regarded as the category  with objects the elements $p\in \catP$, and a unique morphism $p\to q$ if and only if $p\leq q\in \catP$.
All vector spaces in this paper are over a fixed field $\kf$. 
Let $\cvec$ denote the category of finite-dimensional vector spaces and linear maps over $\kf$.
\paragraph{Persistence modules and their decompositions.} 
 
A persistence module over $\catP$ is a functor $M:\catP\to \cvec$.\footnote{In the literature, $M$ is often referred to as a pointwisely finite-dimensional persistence module.} We refer to $M$ simply as a \textbf{$\catP$-module}. 
For any $p\in \catP$, we denote the vector space $M_p:=M(p)$, and for any $p\leq q\in \catP$, we denote the linear map $\varphi_M(p,q):=M(p\leq q)$. 
Given any $\catP$-modules $M$ and $N$, the \text{direct sum} $M\oplus N$ is defined pointwisely at each $p\in \catP$.
We say that a nontrivial $\catP$-module $M$ is \textbf{decomposable} if $M$ is isomorphic to $N_1\oplus N_2$ for some non-trivial $\catP$-modules $N_1$ and $N_2$, which we denote by $M\cong N_1\oplus N_2$. 
Otherwise, we say that $M$ is \textbf{indecomposable}. 

By the Azumaya-Krull-Remak-Schmidt Theorem \cite[Theorem 1]{azumaya1950corrections}, every $\catP$-module is isomorphic to a direct sum of indecomposable $\catP$-modules; see also \ok{\cite[Theorem 1.1]{botnan2020decomposition}}.  
This direct sum decomposition is unique up to isomorphism and permutations of summands. The multiset of isomorphism classes of indecomposable summands of $M$ is called the \textbf{barcode} of $M$, denoted by $\barc(M)$.

\begin{definition}[\cite{botnan2018algebraic}] \label{def:intervals} An \textbf{interval} of a poset $\catP$ is any non-empty subset $I\subset \catP$ such that 
\begin{enumerate}[label=(\roman*),leftmargin=*]
    \item  (\textbf{convexity}) If $p,r\in I$ and $q\in \catP$ with $p\leq q\leq r$, then $q\in I$, \label{item:convexity}
    \item (\textbf{connectivity}) For any $p,q\in I$, there is a sequence $p=r_0,r_1,\ldots,r_n=q$ of elements of $I$, where $r_i$ and $r_{i+1}$ are comparable for $0\leq i\leq n-1$.
    \label{item:connectivity}
\end{enumerate}
\end{definition}

Given an interval $I$ of $\catP$, the \textbf{interval module} $\kf_I$ is the $\catP$-module, with
\begin{equation}\label{eq:interval module} (\kf_I)_p := \begin{cases} \kf & \mathrm{if \ } p\in I\\
0 & \mathrm{otherwise}
\end{cases},
\hspace{10mm}
\varphi_{\kf_I}(p,q) := \begin{cases} \id_\kf & \mathrm{if \ } p\leq q\in I\\
0 & \mathrm{otherwise}\end{cases}\end{equation}

Every interval module is indecomposable \cite[Proposition 2.2]{botnan2018algebraic}. 
\ok{\label{nom:barcode} 
A $\catP$-module $M$ is \textbf{interval-decomposable} if it is isomorphic to a direct sum of interval modules.}
If \ok{the equivalence class of the interval module $\kf_I$ belongs to} $\barc(M)$, then \ok{we simply say that $I$ belongs to $\barc(M)$ and write $I\in \barc(M)$.}  
A \textbf{zigzag poset} of $n$ points is 
    $\bullet_{1}\leftrightarrow \bullet_{2} \leftrightarrow \ldots \bullet_{n-1} \leftrightarrow \bullet_n$
where $\leftrightarrow$ stands for either $\leq$ or $\geq$. A functor from a zigzag poset (of $n$ points) to $\cvec$ is called a \textbf{zigzag module} \cite{carlsson2010zigzag}. 

\begin{oldtheorem}[\cite{carlsson2010zigzag,gabriel1972unzerlegbare}]\label{thm:zigzag modules are interval-decomposable}
Zigzag modules are interval-decomposable.
\end{oldtheorem}

\ok{Let $\mathfrak{L}$ be any collection of indecomposable $\catP$-modules. A $\catP$-module $M$ is called \textbf{$\mathfrak{L}$-decomposable}, if every indecomposable summand of $M$ is isomorphic to an element in $\mathfrak{L}$.\footnote{This terminology was used in \cite{asashiba2022interval}.} If $\mathfrak{L}$ consists solely of interval modules supported on intervals in a collection $\Ical$ of intervals of $\catP$, we simply say that $M$ is \textbf{$\Ical$-decomposable}.}

\paragraph{Tame, Finitely presentable, and Constructible $\catP$-modules.} Aside from interval decomposability, there are other notions of structural simplicity for persistence modules.  

For any element $p$ of a poset $\catP$, let $p^\uparrow:=\{q\in \catP:p\leq q\}$\label{nom: up p}, which is an interval of $\catP$.
\begin{definition}\label{def:finitely presentable}
 A $\catP$-module $M$ is said to be \textbf{finitely presentable} if $M$ is isomorphic to the cokernel of a morphism $F_1\rightarrow F_0$, where $F_0$ and $F_1$ are direct sums of the form $F_0=\bigoplus_{a\in A} \kf_{a^\uparrow}$ and $F_1=\bigoplus_{b\in B} \kf_{b^\uparrow}$, with $A$ and $B$ finite multisets of elements of $\catP$.
\end{definition}

\begin{definition}[\cite{miller2020homological}] \label{def:finite encoding}An \textbf{encoding} of a $\catP$-module $M$ by a poset $\catQ$ is an order-preserving map $\pi : \catP \rightarrow \catQ$ together with a $\catQ$-module $N$ such that $M$ is the \textbf{pullback} of $N$ along $\pi$, i.e. for all $p,p'\in \catP$ with $p\leq p'$, we have that
\[M_p=N_{\pi(p)}\ \mbox{and } \varphi_M(p, p')=\varphi_N(\pi(p), \pi(p')).\]
In this case, we write $M=\pi^\ast N$. The encoding is called \textbf{finite} if the poset $\catQ$ is finite.\footnote{In \cite{miller2020homological}, it is also assumed that for all $q \in \catQ$, the vector space $N_q$ has finite dimension. However, in this paper, that must readily be the case since $M_p$ is assumed to be finite-dimensional for all $p\in \catP$.} A $\catP$-module $M$ is \textbf{tame} if $M$ admits a finite encoding \cite[Theorem 6.12]{miller2020homological}.
\end{definition}

There are other notions of structural simplicity of persistence modules. One of them is \emph{$q$-tameness}. An $\R$-module $M$ is considered $q$-tame if, for any $p < q$ in $\R$, the rank of the map $\varphi_M(p, q)$ is finite \cite{chazal2009proximity, chazal2016structure}. This definition can be directly generalized to $\catP$-modules for arbitrary posets $\catP$. Since this work only deals with pointwise finite-dimensional persistence modules, we readily assumed $q$-tameness throughout. Another is \emph{constructibility} \cite{de2016categorified,gulen2022galois,patel2018generalized}:

\begin{definition}\label{def:co-closure and constructible}
An order-preserving map $c:\catP\rightarrow \catP$ such that $c(p) \leq p$ for all $p \in \catP$, $c \circ c = c$ is called a \textbf{co-closure}.
A $\catP$-module $M$ (resp. any function $f$ from $\catP$) is \textbf{constructible} if there is a co-closure $c:\catP\rightarrow \catP$ with finite image, and  $M\circ c=M$ (resp. $f\circ c= f$). 
Specifically, if the image of $c$ is $\catS\subset \catP$, then $M$ (resp. $f$) is called \textbf{$\catS$-constructible}. 
\end{definition}

\subsection{The M\"obius inversion formula}\label{sec:mobius inversion}
We review the notions of incidence algebra and M\"obius inversion \cite{rota1964foundations, stanley2011enumerative}.
Throughout this section, let $\catQ$ denote a  \textbf{locally finite} poset, i.e. for all $p,q\in \catQ$ with $p\leq q$, the \textbf{segment} \[[p,q]:=\{x\in \catQ:p\leq x\leq q\}\]
is finite. Let 
\begin{equation}\label{eq:segment}
    \seg(\catQ):=\{[p,q]:p\leq q\ \mbox{in $\catQ$}\}.
\end{equation} Given any function $\alpha:\seg(\catQ)\rightarrow \kf$, we write $\alpha(p,q)$ for $\alpha([p,q])$.
\label{nom:incidence algebra} The \textbf{incidence algebra} $I(\catQ,\kf)$ of $\catQ$ over $\kf$ is the $\kf$-algebra of all functions $\seg(\catQ)\rightarrow \kf$ with the usual structure of a vector space over $\kf$, where multiplication is given by convolution:
\begin{equation}\label{eq:convolution}
(\alpha\, \beta)(p,r):=\sum_{q\in [p,r]}\alpha(p,q)\cdot\beta(q,r).
\end{equation}
Since $\catQ$ is locally finite, the above sum is finite and hence $\alpha\beta$ is well-defined. This multiplication is associative and thus $I(\catQ,k)$ is an associative algebra. The \textbf{Dirac delta function} $\delta_\catQ\in I(\catQ,\kf)$ is given by 
\begin{equation}\label{eq:delta function}
\delta_\catQ(p,q):=\begin{cases}1,&p=q\\0,&\mbox{else,}\end{cases}
\end{equation}
which is the two-sided multiplicative identity. 
\begin{remark}[\cite{stanley2011enumerative}]\label{rem:invertibility} An element $\alpha\in I(\catQ,\kf)$ admits a multiplicative inverse if and only if $\alpha(q,q)\neq 0$ for all $q\in \catQ$. 
\end{remark}

Another important element of $I(\catQ,\kf)$ is the \textbf{zeta function}:
\begin{equation}\label{eq:zeta function}
    \zeta_\catQ(p,q):=\begin{cases}1,&p\leq q\\0,&\mbox{else.}\end{cases}
\end{equation}
By Remark \ref{rem:invertibility}, the zeta function $\zeta_\catQ$ admits a multiplicative inverse, which is called the \textbf{M\"obius function} $\mu_{\catQ}\in I(\catQ,\kf)$. 
The M\"obius function can be computed recursively as 
\begin{equation}\label{eq:mobius}
    \mu_{\catQ}(p,q)=\begin{cases}1,& \mbox{$p=q$,}\\ -\sum\limits_{p\leq r< q}\mu_{\catQ}(p,r),& \mbox{$p<q$,} \\ 0,& \mbox{otherwise.} \end{cases}
\end{equation}

\label{nom: kq} \ok{Let $\kf^\catQ$ denote the
vector space of all functions $\catQ\rightarrow \kf$. 
\label{nom: down p}Also, for $q\in \catQ$, let \[q^\downarrow:=\{p\in \catQ: p\leq q\},\] called a \textbf{principal ideal}.}  Assuming that $q^{\downarrow}$ 
is finite for each $q\in \catQ$, every element in $I(\catQ,\kf)$ acts on  $\kf^\catQ$ by right multiplication: for any $f\in k^\catQ$ and for any $\alpha\in I(\catQ,k)$, we have

\begin{equation}\label{eq:function times matrix}(f\ast\alpha)(q):=\sum_{p\leq q}f(p)\,\alpha(p,q).
\end{equation}

In fact, even when not every principal ideal in a poset $\catQ$ is finite, Equation (\ref{eq:function times matrix}) still specifies a well-defined multiplication between $f$ and $\alpha$ under the weaker assumption that
\begin{equation}\label{eq:convolvable}
\mbox{for every $q\in \catQ$, 
$f(r)=0$ for all but finitely many $r\in q^\downarrow$. 
} 
\end{equation}
\begin{definition}\label{def:convolvable}
Given a locally finite poset $\catQ$,
we call a function $f:\catQ\rightarrow \kf$ \textbf{\convolvable} (over $\catQ$) if $f$ satisfies Condition (\ref{eq:convolvable}).
\end{definition}

For a function $f\in \kf^\catQ$, the \textbf{support} of $f$ is the set $\{q\in \catQ \ : \ f(q) \neq 0\}$.

\begin{remark}[About convolvability]\label{rem:convolvability}
Let $\catQ$ be a locally finite poset.
\begin{enumerate}[label=(\roman*),leftmargin=*]
\item  Any function $\catQ\rightarrow \kf$ with finite support is convolvable. \ok{Hence, if $\catQ$ is finite, \emph{any} function $\catQ\rightarrow \kf$ is convolvable.} \label{item:convolvability1}
\item If every principal ideal in $\catQ$ is finite, then every $f\in \kf^\catQ$ is \convolvable.\label{item:convolvability2}
\item The collection of all convolvable functions $\catQ\rightarrow \kf$ is a linear subspace of $\kf^\catQ$.\label{item:convolvability3}
\item If $f:\catQ\rightarrow \kf$ is \convolvable\ over $\catQ$, then for any $\catP\subset \catQ$, the restriction $f|_{\catP}$ is  \convolvable.\label{item:convolvability4}
\item Let $P_1,P_2\subset \catQ$. Assume that, for each $i=1,2$, $f_i:\catP_i\rightarrow k$ is convolvable over $\catP_i$, and that $f_1=f_2$ on $\catP_1\cap \catP_2$. Then $f_1\cup f_2:\catP_1\cup \catP_2\rightarrow k$ is convolvable over $\catP_1\cup \catP_2$.
\label{item:convolvability5}
\end{enumerate}
\end{remark}

\begin{remark}\label{rem:right multiplication} 
Let $\catQ$ be a locally finite poset. Let $\kf_c^\catQ$ be the space of convolvable functions $\catQ\rightarrow \kf$, which is a subspace of $\kf^\catQ$. It follows that:

\begin{enumerate}[label=(\roman*),leftmargin=*]
    \item \ok{By Remark \ref{rem:convolvability} \ref{item:convolvability1}, if $\catQ$ is finite, then $\kf_c^\catQ=\kf^\catQ$.} \label{item:right multiplication-Q is finite}
    \item Let $\alpha\in I(\catQ,\kf)$. The  \textbf{right multiplication map} $\ast \alpha:\kf_c^\catQ\rightarrow \kf_c^\catQ$ given by $f\mapsto f\ast \alpha$ is an automorphism if and only if $\alpha$ is invertible.\label{item:right multiplication}
    \item By  Remark \ref{rem:invertibility} and the previous item, the right multiplication map $\ast\zeta_\catQ$ by the zeta function is an automorphism on $\kf_c^\catQ$ with inverse $\ast\mu_{\catQ}.$\label{item:zeta defines an automorphism}
\end{enumerate}  
\end{remark}

The M\"obius inversion formula is a powerful tool that has found widespread applications in combinatorics despite the fact that it has its origins in number theory. It will be a central tool for establishing our main results.

\begin{oldtheorem}[M\"obius Inversion formula]\label{thm:mobius}
Let $\catQ$ be a locally finite poset.
For any pair of \convolvable\  functions $f,g:\catQ\rightarrow k$,   \begin{equation}\label{eq:zeta}\displaystyle g(q)=\sum_{r\leq q} f(r)\ \mbox{for all $q\in \catQ$}
\end{equation} if and only if 
\begin{equation}\label{eq:zeta inverse}\displaystyle f(q)=\sum_{r\leq q} g(r)\cdot \mu_{\catQ}(r,q) \ \mbox{for all } q\in \catQ.
\end{equation}
\end{oldtheorem}

\begin{proof}
    Equation (\ref{eq:zeta}) can be represented as $g=f\ast\zeta_\catQ$. By multiplying both sides by $\zeta_\catQ^{-1}=\mu_\catQ$ on the right, we have $g\ast\mu_\catQ=f$, which is precisely Equation (\ref{eq:zeta inverse}).
\end{proof}

\begin{definition}\label{def:mobius inversion}The function $f=g\ast\mu_Q$ is referred to as the \textbf{M\"obius inversion} of $g$ (over $\catQ$).
\end{definition}

\ok{M\"obius inversion serves as a discrete analogue of the concept of a derivative in calculus, as illustrated by the following example.}

\begin{example}Let $q\in \catQ$ and define the two functions $\mathbf{1}_q,\ \mathbf{1}_{\geq q}:\catQ\rightarrow \kf$ to be 
\[\mathbf{1}_q(p):=\begin{cases}1,&p=q\\0,&\mbox{otherwise.}\end{cases} \ \ \ \ \ \ \mathbf{1}_{\geq q}(p):=\begin{cases}1,&p \geq q\\0,&\mbox{otherwise.}\end{cases}
\]
Then, both functions are convolvable. Indeed, Remark \ref{rem:convolvability} \ref{item:convolvability1} implies that $\mathbf{1}_q$ is convolvable and the local finiteness of $\catQ$ guarantees that $\mathbf{1}_{\geq q}$ is convolvable. Notice that the M\"obius inversion of $\mathbf{1}_{\geq q}$ is equal to $\mathbf{1}_q$
and that 
$\mathbf{1}_q$ captures where 
 the function values of $\mathbf{1}_{q\geq}$ change. 
\end{example}

The following proposition will be useful for some proofs in the sequel. 

\begin{proposition}\label{prop:mobius inversion of non-decreasing map is convolvable} Let $g:\catQ\rightarrow \R$ be non-decreasing and convolvable over $\catQ$. Then, the M\"obius inversion of $g$ is  convolvable over $\catQ$.
\end{proposition}

\begin{proof} Let $f$ be the M\"obius inversion of $g$. Since $g$ is non-decreasing, if $g(r)=0$, then $g(r')=0$ for all $r'\leq r$  and thus $f(r)=\sum_{r'\leq r}g(r')\cdot\mu(r',r)=0$. Furthermore, since $g$ is convolvable, $g(r)=0$  for all but finitely many $r\in q^\downarrow$, for every $q\in \catQ$. Hence, we have $f(r)=0$ for all but finitely many $r\in q^\downarrow$, for every $q\in \catQ$, as desired. 
\end{proof}

A constructible function (Definition \ref{def:co-closure and constructible}) admits M\"obius inversion after its proper restriction:
\begin{proposition}[{\cite[Proposition 3.4]{gulen2022galois}}]\label{prop:Mobius inversion of a contructible function}\label{prop:inversion after restriction}Let $\catS$ be a nonempty finite subset of a given poset $\catP$. Let $m:\catP\rightarrow \Z$ be an $\catS$-constructible function. Then,  for all $q\in \catP$,
\[m(q)=\sum_{\substack{p\leq q \\ p\in S}} (m|_S\ast \mu_S)(p).\]
\end{proposition}

\paragraph{A matrix algebra perspective on the incidence algebra} Let $m\in \N$ and let $\catQ$ be a poset with $m$ elements. 
By the order-extension principle, we can extend the order on $\catQ$ to a total order.
Thus, we fix $\catQ:= \{q_1,\ldots,q_m\}$, where $q_i<q_j$ implies $i<j$.
Then, each element $\alpha$ in the incidence algebra $I(\catQ,\kf)$ is canonically identified with the $(m\times m)$-upper-triangular matrix $(\alpha_{ij})$ whose $(i,j)$-entry is
\[ \alpha_{ij}:=\begin{cases}\alpha(q_i,q_j)&\mbox{if $q_i<q_j$}
\\0,&\mbox{else.}\end{cases}
\]
Then, for $\beta\in I(\catQ,\kf)$, the product $\alpha\, \beta$ in Equation (\ref{eq:convolution}) can be expressed as the multiplication of the upper-triangular matrices $(\alpha_{ij})$ and $(\beta_{ij})$, where \[(\alpha\, \beta)_{ij}= \mbox{(the $i$-th row of $\alpha$)$\cdot$(the $j$-th column of $\beta$$)^t$}.\]
Now let us identify each $f\in k^\catQ$ with the $m$-dimensional row vector $(f_i)$ where $f_i=f(q_i)$ for $i=1,\ldots,m$. 
Then, the multiplication $f\ast\alpha$ in Equation (\ref{eq:function times matrix}) can be expressed as the multiplication of the $(1\times m)$-matrix $(f_i)$ and the $(m\times m)$-matrix $(\alpha_{ij})$.

\begin{remark} Recall that an upper-triangular matrix is invertible if and only if all of its diagonal entries are nonzero. Remark \ref{rem:invertibility} can be seen as a straightforward from this fact. 
\end{remark}

\subsection{Generalized rank invariant}\label{sec:GRI}
In this section, we recall the definitions of the rank invariant \cite{carlsson2009theory,puuska2020erosion} and the generalized rank invariant \cite{kim2021generalized}. 

Let $M$ be any $\catP$-module. 
\label{nom:rank invariant}
Then, $M$ admits both a \emph{limit} and a \emph{colimit} \cite[Chapter V]{mac2013categories}: 
A limit of $M$, denoted by $\varprojlim M$, consists of a vector space $L$ together with a collection of linear maps $\{\pi_p:L\to M_p\}_{p\in \catP}$ such that 
\begin{equation}\label{eq:limit}
    \varphi_M(p, q)\circ \pi_p =\pi_q\  \mbox{for every $p\leq q$ in $\catP$}.
\end{equation}
A colimit of $M$, denoted by $\varinjlim M$, consists of a vector space $C$ together with a collection of linear maps $\{i_p:M_p\to C\}_{p\in \catP}$ such that 
\begin{equation}\label{eq:colimit}
    i_q\circ \varphi_M(p, q)=i_p\  \mbox{for every $p\leq q$ in $\catP$}.
\end{equation}
Both $\varprojlim M$ and $\varinjlim M$ satisfy certain universal properties, making them unique up to isomorphism.

Let us assume that $\catP$ is connected (Definition \ref{def:intervals} \ref{item:connectivity}). The connectedness of $\catP$ alongside the equalities given in Equations (\ref{eq:limit}) and (\ref{eq:colimit}) imply that $i_p \circ \pi_p=i_q \circ \pi_q:L\rightarrow C$ for any $p,q\in \catP$. 
This fact ensures that the \textbf{canonical limit-to-colimit map} \[\psi_M:\varprojlim M\longrightarrow \varinjlim M\] given by $i_p\circ \pi_p$ for any $p\in \catP$ is well-defined. 
\label{nom:generalized rank} The \textbf{(generalized) rank} of $M$ is defined to be:\footnote{{This construction was considered in the study of quiver representations \cite{kinser2008rank}.}} 
\begin{equation}\label{eq:generalized rank of M}
    \rank(M):=\rank(\psi_M).
\end{equation}

\begin{remark}\label{rem:generalized rank is finite for pfd modules}
    $\rank(M)$ is finite as $\rank(M)=\rank(i_p\circ \pi_p)\leq \dim(M_p)<\infty$ for any $p\in P$.
\end{remark}

The rank of $M$ is a count of the `persistent features' in $M$ that span the entire indexing poset $\catP$. 
\ok{Such `persistent features' appear as summands of $M$ in the form of $\kf_\catP$:}

\begin{oldtheorem}[{\cite[Lemma 3.1]{chambers2018persistent}}]
\label{thm:rkequalsintervals}
Let $\catP$ be a connected poset. Assume that a $\catP$-module $M$ is isomorphic to a direct sum $\bigoplus_{a\in A} M_a$ for some indexing set $A$ where each $M_a$ is indecomposable. Then, the rank of $M$ is equal to the cardinality of the set $\{a\in A: M_a\cong \kf_P\}.$ 
\end{oldtheorem}

We refine the rank of a $\catP$-module, which is a single integer, into an integer-valued function. Possible domain options for the function include the following sets, in addition to $\seg(\catP)$ as shown in Equation (\ref{eq:segment}):
\label{nom:Con(P)}
\begin{align*}
\Con(\catP)&:=\{I\subset\catP:\mbox{$I$ is connected}\}, \mbox{}
\\
\Int(\catP)&:=\{I\subset\catP:\mbox{$I$ is an interval}\}.
\end{align*}
Note the inclusions $\seg(\catP)\subset \Int(\catP)\subset \Con(\catP).$

\begin{definition}[\cite{asashiba2019approximation,carlsson2009theory,dey2022computing,kim2021generalized}]\label{def:generalized rank invariant}   
    The \textbf{generalized rank invariant (GRI) of} a $\catP$-module $M$ is the map 
    \begin{align*}
    \rank_M:\Con(\catP)&\to \Z_{\geq 0}\\
    I&\mapsto \rank(M\vert_I)
    \end{align*}
where $M\vert_I$ is the restriction of $M$ to $I$.  The restriction of $\rk_M$ to $\Ical\subset \Con(\catP)$ is denoted by $\rk_M^\Ical$ and is called the \textbf{GRI of $M$ over} $\Ical$.  
When $\mathcal{I}=\Int(\catP)$, the GRI over $\Ical$ is simply called the \textbf{Int-GRI} and denoted by $\rk_M^{\Int}$. When $\Ical=\seg(\catP)$, the GRI of $M$ over $\Ical$ is the \textbf{rank invariant (RI)} of $M$, introduced in \cite{carlsson2009theory}. 
\end{definition}

\begin{remark} The following are useful properties of the GRI.\label{rem:properties of rank invariant}
\begin{enumerate}[label=(\roman*),leftmargin=*]
\item (Monotonicity)\label{item:monotonicity} If $I\subset J$ in $\Con(\catP)$, then $\rank_M(I)\geq \rank_M(J)$. This is because the canonical limit-to-colimit map over $I$ is a factor of the canonical
limit-to-colimit map over $J$ \cite[Proposition 3.8]{kim2021generalized}. 

\item (Additivity) If $M\cong \bigoplus_{i=1}^n M_i$, and $I\in \Con(\catP)$, then $\rank_M(I) = \sum_{i=1}^n \rank_{M_i}(I)$. \label{item:additivity}

\item (The GRI of an interval module) Let $J\in \Int(\catP)$. For the interval module $\kf_J$ and any $I\in \Con(\catP)$, we have
\[\rank_{\kf_J}(I)=\begin{cases}1,& J\supset I\\0,&\mbox{else.}\end{cases}\]
\label{item:the GRI of an interval module}
\end{enumerate}
\end{remark}

\begin{remark}[Domain of the GRI]
\begin{enumerate}[label=(\roman*)]
\item For 2-parameter persistence modules \cite{kim2021bettis}, the Con-GRI was proved to be strictly stronger than the Int-GRI.

\item The Int-GRI is strictly stronger than the bigraded Betti numbers, which follows from \cite{kim2021bettis} together with Proposition \ref{prop:bigraded Bettis do not determine GRI over length-3 zigzags}.

\item Some invariants of multi-parameter persistence modules determine and/or are determined by the GRI over a collection that includes non-intervals (e.g. fibered barcode, ss-compressed multiplicity \cite{asashiba2019approximation}, and \emph{zigzag-path-indexed barcode} (Definition \ref{def:zigzag-path-indexed barcode})).
\end{enumerate}
\end{remark}

We discuss the GRI's properties and discriminating power in Section \ref{sec:main}. 

\smallskip 
We close this section by describing a parameterized family of subsets of $\Int(\catP)$ that will be useful later. Given a poset $\catP$, let $a\in \catP$ and let $B\subset \catP$ be an antichain. We write $a\leq B$ (resp. $a\geq B$) if there exists $b\in B$ such that $a\leq b$ (resp. $a\geq b$). When two antichains $A,B\subset \catP$ are given, we write $A\leq B$ if $\forall a\in A$, $\forall b\in B$, $a\leq B$ and $A\leq b$. This defines a partial order on the set of antichains in $\catP$ \cite[Section 2.1]{blanchette2021homological}.
When $A\leq B$, define 
 \[[A,B):=\{x\in \catP: \exists a\in A, \exists b\in B, a\leq x <b\},\]
 \[[A,B]:=\{x\in \R^2: \exists a\in A, \exists b\in B, a\leq x \leq b\}.\]
  $[A,B)$ is either empty (when $A=B$) or an interval of $\catP$ (when $A<B$), whereas $[A,B]$ is always an interval of $\catP$.
    For $m,n\in \N$, let
    \begin{equation}\label{eq:intmn}
    \Int_{m,n}(\catP):=\{[A,B): A,B\subset \catP\ \mbox{are antichains s.t.  $A<B$, $\abs{A}\leq m$, $\abs{B} \leq n$ \}.}
    \end{equation}
    See for an illustration corresponding to the case $\catP=\R^2$.
    \begin{figure}[h]
    \centering
    \includegraphics[width=0.6\textwidth]{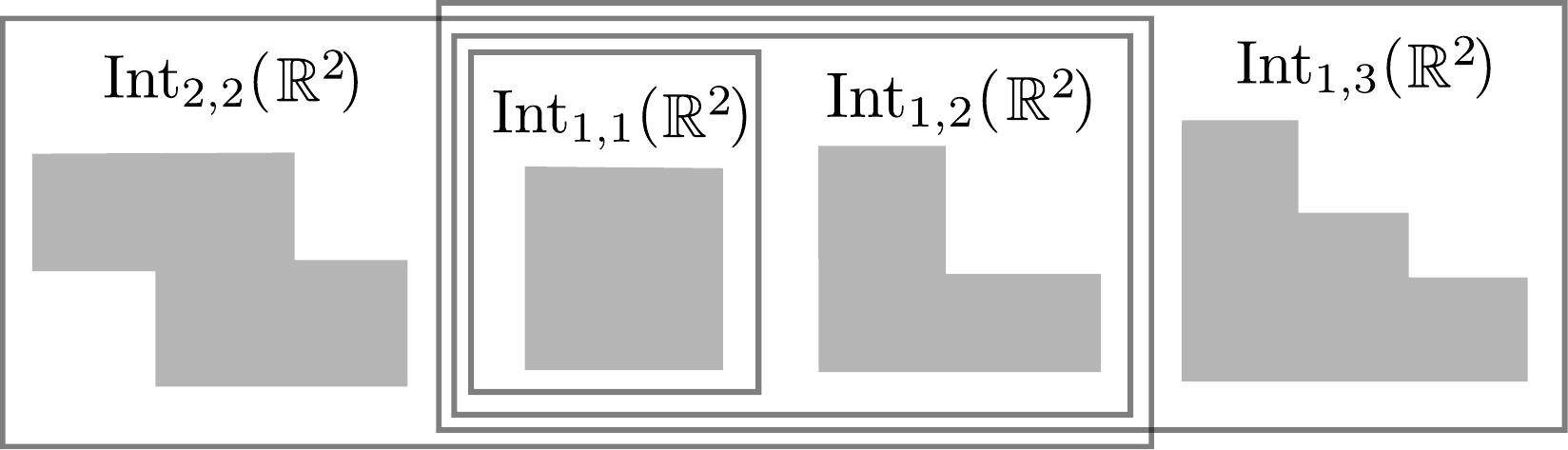}
\end{figure}

Similarly, define 
\begin{equation}\label{eq:intmn cc}\Int_{m,n}^\mathrm{cc}(\catP):=\{[A,B]:A,B\subset \catP\ \mbox{are antichains s.t.  $A\leq B$, $\abs{A}\leq m$, $\abs{B} \leq n$}\}.
\end{equation}

\subsection{Generalized persistence diagrams}\label{sec:GPD}

\ok{In this section we review the notion of the generalized persistence diagram (GPD) introduced in \cite{kim2021generalized}, but we modify the original formulation by adding more flexibility in choosing the domain of the GPD.} 
\ok{This additional flexibility is useful} when clarifying the relationship between the GRI/GPD and other invariants of persistence modules such as graded Betti numbers \cite{kim2021bettis} and rank decompositions (Section \ref{sec:GPD and Rk decomposition}).

 Let $\Con(\catP)$ and any subcollection of $\Con(\catP)$ be ordered by containment $\supset$. Let $M$ be a $\catP$-module.
If $\Ical\subset \Con(\catP)$ is locally finite and $\rank_M^\Ical$ is convolvable over $\Ical$, we say that \textbf{the GRI of $M$ is convolvable over $\Ical$}, cf. Definition \ref{def:convolvable}.

\begin{definition}\label{def:GPD over I}
    Let $M$ be a $\catP$-module and let $\Ical\subset \Con(\catP)$.
    Assume that the GRI of $M$ is \convolvable\ over $\Ical$. 
    The \textbf{generalized persistence diagram (GPD) of $M$ over $\Ical$} is the M\"obius inversion of $\rank_M^\Ical$ over the poset $(\Ical,\supset)$, i.e. the function $\dgm_M^{\Ical}:\Ical\rightarrow \Z$ given by:\begin{equation}\label{eq:dgm}\dgm_M^{\Ical}:=\rank_M^{\Ical}\ast\mu_{\Ical}.\end{equation}
   When $\Ical=\Int(\catP)$, the GPD of $M$ over $\Ical$ will be referred to as simply the \textbf{\ok{Int}-GPD} of $M$, and will be denoted by $\dgm_M$.  
   \ok{As will be shown later, $\Ical=\Int(\catP)$ is the minimal subset of $\Con(\catP)$ which allows the generalized rank invariant over $\Ical$ to be at least as strong an invariant as the barcode of interval-decomposable $\catP$-modules 
   (Theorem \ref{thm:tightness})}.
\end{definition}

\begin{remark}\label{rem:additivity of GPD} When the GRIs of $\catP$-modules $M$ and $N$ are both convolvable over $\Ical\subset \Con(\catP)$, by additivity of the GRI (cf. Remark \ref{rem:properties of rank invariant} \ref{item:monotonicity}), we have $\dgm_{M\oplus N}^\Ical=\dgm_M^\Ical+\dgm_N^\Ical$. 
\end{remark}

We establish a few other desirable properties of the GPD.

\begin{proposition}\label{prop:properties of GPD} \contributionn If the GRI of a $\catP$-module $M$ is convolvable over a locally finite $\Ical\subset \Con(\catP)$, then:

\begin{enumerate}[leftmargin=*, label=(\roman*)]
\item The support of $\dgm_M^{\Ical}$ is contained in the support of $\rank_M^{\Ical}$. In other words, for any $I\in \Ical$, if $\rank_M^{\Ical}(I)=0$, then $\dgm_M^{\Ical}(I)=0$. \label{item:GPD is more parsimonious}
\item $\dgm_M^{\Ical}$ is convolvable over $\Ical$.\label{item:convolvability of GPD}
\item The unique function $\drm:\Ical\rightarrow \R$ satisfying the following equality is $\dgm_M^{\Ical}$:
\[\rank_M(I)=\sum_{\substack{J\supset I\\ J\in \Ical}}\drm(J) \ \ \mbox{for all $I\in \Ical$}.\]
In other words, Equation (\ref{eq:dgm}) implies $\rk_M^\Ical = \dgm_M^\Ical \ast \zeta_\Ical$.
\label{item:universal property of GPD}
\end{enumerate} 
\end{proposition}

\begin{proof}Item \ref{item:GPD is more parsimonious} follows from Remark \ref{rem:properties of rank invariant} \ref{item:monotonicity} and  Proposition \ref{prop:mobius inversion of non-decreasing map is convolvable}.
Item \ref{item:convolvability of GPD} is an immediate consequence of item \ref{item:GPD is more parsimonious}.
Item \ref{item:universal property of GPD} follows from Theorem \ref{thm:mobius} and item \ref{item:convolvability of GPD}.
\end{proof}
\ok{Item \ref{item:GPD is more parsimonious} together with the fact that $\dgm_M^\Ical$ determines $\rk_M^\Ical$ suggests that $\dgm_M^\Ical$ is a concise encoding of $\rk_M^\Ical$.} 

\subsection{Rank decompositions}\label{sec:GPD and Rk decomposition}

In this section, we provide a review of the notion of \emph{rank decomposition} \cite{botnan2021signed}, with a strong focus on its connection with the M\"obius inversion formula. 

For any multiset $\Rcal$ of intervals in a poset $\catP$, let $\mult_{\Rcal}:\Int(\catP)\rightarrow \Zplus$ be defined by $R\mapsto$ (the multiplicity of $R$ in $\Rcal$). The multiset $\Rcal$ is said to be \textbf{\pfin}\ if for all $p\in \catP$, the sum $\sum_{R\ni p}\mult_{\Rcal}(R)$ is finite.

\begin{remark}\label{rem:pointwisely finite multiset}
    It is not difficult to check that the following are equivalent: 
    \begin{enumerate}[leftmargin=*,label=(\roman*)]
        \item A multiset $\Rcal$ of intervals in $\catP$ is \pfin{}.\label{item:pointwisely finite multiset1}
        \item The direct sum $\kf_{\Rcal}:=\bigoplus_{R\in \Rcal}k_R$ is pointwisely finite-dimensional, i.e. $\dim((\kf_{\Rcal})_p)<\infty$ for all $p\in \catP$. \label{item:pointwisely finite multiset2}
        \item For all $I\in \Int(\catP)$, $\rank_{\kf_{\Rcal}}(I)<\infty$  (cf. Remark \ref{rem:properties of rank invariant} \ref{item:monotonicity}).  \label{item:pointwisely finite multiset3}
        \item Every principal ideal of the subposet $\{R\in \Int(\catP): \mult_\Rcal(R)\neq 0\}\subset\Int(\catP)$ is finite. \label{item:pointwisely finite multiset4}
    \end{enumerate}
\end{remark}

\begin{definition}\label{def:rank decomposition}
     Let $M$ be a $\catP$-module. 
     Whenever they exist, 
     any pair $(\Rcal,\Scal)$ of \pfin\  multisets of elements in $\Int(\catP)$  such that
    \[\rank_M^{\Int}=\rank_{\kf_{\Rcal}}^{\Int}-\rank_{\kf_{\Scal}}^{\Int}\]
    is called a \textbf{rank decomposition} of $\rank_M^\Int$.

    If $\Rcal$ and $\Scal$ are disjoint, then the rank decomposition is called \textbf{minimal} and we write $\Rcal_M$ and $\Scal_M$ (the uniqueness of the pair $(\Rcal_M,\Scal_M)$ follows from the next proposition).
\end{definition}

\begin{proposition}[{\cite[Corollary 2.12]{botnan2021signed}}]\label{prop:if rk decomposes then the minimal one exists}
  Let $M$ be a $\catP$-module such that $\rank_M^\Int$ admits a rank decomposition $(\Rcal,\Scal)$. Then, the unique minimal rank decomposition of $\rank_M^\Int$ is given by 
  $(\Rcal_M,\Scal_M)$, where $\Rcal_M:=\Rcal-(\Rcal\cap\Scal)$ and $\Scal_M:=\Scal-(\Rcal\cap\Scal)$.
\end{proposition}

\begin{proposition}[{\cite[Proposition 3.3]{botnan2021signed}}]\label{prop:rk decomposition is GPD} 
Assume that $\Int(\catP)$ is locally finite.
If the $\Int$-GRI of a $\catP$-module $M$ is convolvable over $\Int(\catP)$, then the minimal rank decomposition of $\rk^\Int_M$ can be obtained as
\[\Rcal_M=\{d_I\cdot I:\dgm_M(I)>0\}\ \ \ \mbox{and}\ \ \  \Scal_M=\{d_I\cdot I:\dgm_M(I)<0\},\] 
where $d_I:=\abs{\dgm_M(I)}$ and $d_I\cdot I$ stands for $d_I$ copies of $I$.  
\end{proposition}
We include the proof from \cite{botnan2021signed}. 
\begin{proof}
Let $\mathbf{1}_{J\supset}:\Int(\catP)\rightarrow \{0,1\}$ be defined by $\mathbf{1}_{J\supset}(I)=1$ if $J\supset I$ and $\mathbf{1}_{J\supset}(I)=0$, otherwise. For every $I\in \Int(\catP)$, we have
     \begin{align*}
        \rk_M^\Int(I)&=\sum_{\substack{J\supset I\\ J\in\Int(\catP)}}\dgm_M(J)\\&=\sum_{\substack{J\supset I\\ \dgm_M(J)>0}}\dgm_M(J)+\sum_{\substack{J\supset I\\ \dgm_M(J)<0}}\dgm_M(J)\\&=\sum_{\substack{\dgm_M(J)>0}}\dgm_M(J)\cdot\mathbf{1}_{J\supset}(I)+\sum_{\substack{\dgm_M(J)<0}}\dgm_M(J)\cdot\mathbf{1}_{J\supset}(I).
    \end{align*}    
By Remark \ref{rem:properties of rank invariant} \ref{item:additivity} and \ref{item:the GRI of an interval module}, the right-hand side of the equation above equals $\rank_{\kf_{\Rcal_M}}^\Int(I)-\rank_{\kf_{\Scal_M}}^\Int(I)$ with $\Rcal_M\cap \Scal_M=\emptyset$, as desired.
\end{proof}

    Under the assumption that $\Int(\catP)$ is locally finite and $\rank_M$ is convolvable over $\Int(\catP)$, Proposition \ref{prop:rk decomposition is GPD} implies that:
    \begin{enumerate}[label=(\roman*)]
        \item  the existence of a rank decomposition is implied by the forward direction of Theorem \ref{thm:mobius}.
        \item the uniqueness of a minimal rank decomposition is implied by the backward direction of Theorem \ref{thm:mobius}.  
    \end{enumerate}

 In Definition \ref{def:rank decomposition} and Propositions \ref{prop:if rk decomposes then the minimal one exists} and \ref{prop:rk decomposition is GPD}, the collection $\Int(\catP)$ and the function $\rank_M^\Int:\Int(\catP)\rightarrow \Z$ can be replaced by any $\Ical\subset \Int(\catP)$ and any map  $r:\Ical\rightarrow \R$, respectively:

 \begin{oldtheorem}[{Restatement of \cite[Theorem 2.5]{botnan2021signed}}]\label{thm:BOO's setup is more general} Let $\Ical\subset \Int(\catP)$ be locally finite and let $r:\Ical\rightarrow \Z$ be convolvable. 
 Then, there exists $a_I\in \Z$ for each $I\in \Ical$ such that for $M:=\bigoplus\limits_{\substack{I\in \Ical\\a_I>0}} (\kf_I)^{a_I}$ and $N:=\bigoplus\limits_{\substack{I\in \Ical\\a_I<0}} (\kf_I)^{-a_I}$, we have  
 $r= \rk_M - \rk_N.$
 \footnote{\ok{Equivalently, 
 \begin{equation}\label{eq:r is a Z-linear comb.}
     r(J)=\sum_{I\in \Ical}a_I\rk_{\kf_I}^{\Ical}(J)
 \end{equation}
 for every $J\in \Ical$ 
where the RHS contains only finitely many nonzero terms.}}
  \end{oldtheorem}

 \ok{This theorem is proved by simply replacing $\Int(\catP)$,  $\rank_M^\Int$, $\dgm_M$ in the proof of Proposition \ref{prop:rk decomposition is GPD} with $\Ical$, $r$, and $r\ast\mu_{\Ical}$, respectively.}
 \ok{When $\Ical$ is finite, the theorem above can also be shown via elementary linear algebra as detailed below. This approach is distinct from that of \cite{botnan2021signed}.}
\begin{proof}[Another proof of Theorem \ref{thm:BOO's setup is more general} when $\Ical$ is finite] \ok{In this proof, we view the GPD and GRI as rational number valued functions in order to utilize the fact that the set  $\Q$ of rational numbers is a field.} For the interval module $\kf_I$, note that $\dgm_{\kf_I}^\Ical:\Ical\rightarrow \Q$ is identical to the indicator function $\mathbf{1}_I^\Ical:\Ical\rightarrow \Q$ with support $\{I\}$. Therefore, the canonical basis $\{\mathbf{1}_I^\Ical:I\in \Ical\}$ of the vector space $\Q^\Ical$ coincides with $\{\dgm_{\kf_I}^\Ical:I\in \Ical\}$. By Remark \ref{rem:convolvability} \ref{item:convolvability1},  $\mathbf{1}_I^\Ical$ is convolvable over $\Ical$. Now, as $\Q^\Ical$ is finite-dimensional, the image of $\{\dgm_{\kf_I}^\Ical:I\in \Ical\}$ under the automorphism $\ast \zeta_{\Ical}$ on $\Q^\Ical$ forms another basis for $\Q^\Ical$ (cf. Remark \ref{rem:right multiplication} \ref{item:right multiplication-Q is finite} and \ref{item:zeta defines an automorphism}). 
This implies that, since
\[\dgm_{\kf_I}^\Ical\ast\zeta_{\Ical}=\rank_{\kf_{I}}^\Ical,\] any function $r:\Ical\rightarrow \Q$ can be uniquely expressed as a linear combination $r=\sum_{I\in \Ical}a_I\cdot \rk_{k_I}$, $a_I\in \Q$. 
It remains to show that $a_I$ is an integer for every $I\in \Ical$. On one hand, since $r$ is a $\Z$-valued function, its M\"obius inverse $r\ast\mu_{\Ical}$ over $\Ical$ takes values in $\Z$. On the other hand, 
\[r\ast \mu_{\Ical}= \left(\sum_{I\in \Ical}a_I\cdot \rk_{k_I}\right)\ast\mu_\Ical=\sum_{I\in \Ical}a_I\cdot(\rk_{k_I}\ast\mu_\Ical)=\sum_{I\in \Ical}a_I\cdot \dgm_{k_I}=\sum_{I\in \Ical}a_I\cdot \mathbf{1}_I.
\]
Since $a_I=r(I)\in \Z$ for every $I\in \Ical$, we are done. (It is also noteworthy that for each $I\in \Ical$, the coefficient $a_I$ of $\rk_{\kf_I}$ is equal to the M\"obius inversion of $r$ evaluated at $I$, i.e. $r\ast \mu_{\Ical}$).
\end{proof}

\subsection{Motivic invariants}\label{sec:motivic invariants}

Throughout this section,  $\catP$ and $\catQ$ will denote posets.

\medskip
In this section we introduce a \emph{parametric} notion of invariant of persistence modules over posets that helps conceptualize several existing invariants at the same time that it permits designing new ones (as we do in Section \ref{sec:ZIB general}). We call these \emph{motivic} invariants since they are, roughly speaking, specified by the choice of a given fixed poset (or a fixed collection thereof), the \emph{motif}, which one uses to \emph{probe} a given $Q$-module.

\begin{figure}
\centering
\includegraphics[width=0.8\linewidth]{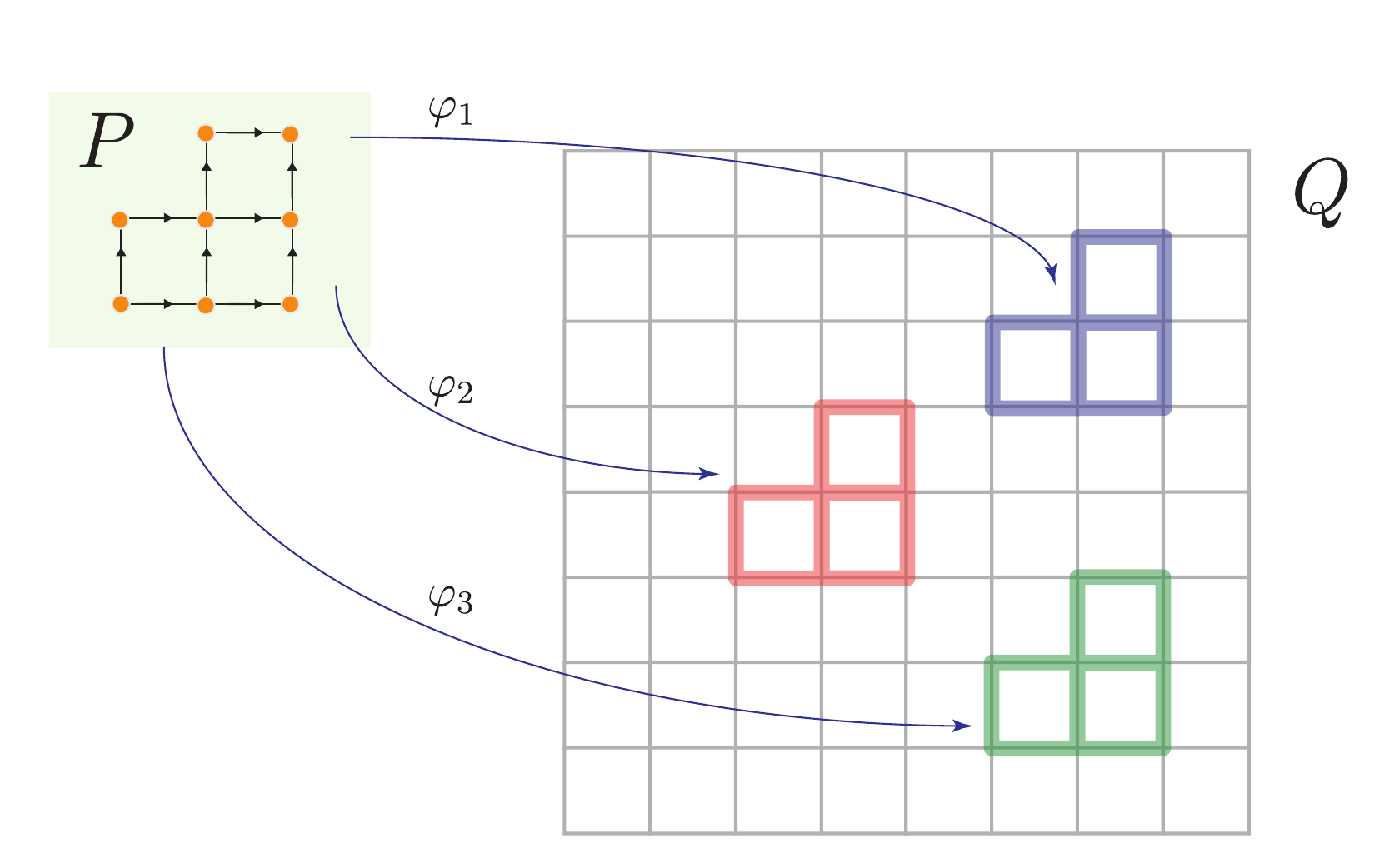}
\caption{A motivic invariant of $Q$-modules is based on finding manifestations, for example via embeddings or more general morphisms $\varphi:P\to Q$ in some allowed class $\Phi$, of a given motif $P$ (or of a collection thereof) inside a given poset $Q$ (here depicted as $\Z^2$, for simplicity). Now, for a given $Q$-module $M$, for each such  $\varphi$, one considers the pullback of $M$ via $\varphi$ and applies a given invariant (functor) $F$ to this pullback module. See Definition \ref{def:motivic invariant} for the actual definition. }\label{fig:motivic}
\end{figure}

The underlying idea of defining invariants through the process of studying all substructures of a given type (where this type is specified through the choice of motif) present in an object is manifested in different fields:
\begin{itemize}
\item In category theory through the pervasive notion of \emph{representable functor} \cite{mac2013categories}.
\item In metric geometry,  motivated by the notion of \emph{curvature sets} considered by Gromov in \cite[Chapter 3]{gromov1999metric}. For a given compact metric space $X$, Gromov considers $K_n(X)$, the $n$-th curvature set of $X$, as the set containing all $n\times n$ matrices produced by $n$-tuples of points in $X$. The collection of all curvature sets completely characterize compact metric spaces up to isometry. 

\item In Lovasz's study of graphs and graphons \cite[Chapter 5]{lovasz2012large} whereby the general notion of \emph{homomorphism number} $\mathsf{hom}(F,G)$ is considered as a count of the number of times a given graph $F$ appears as substructure (subgraph) of a given larger graph $G$. The numbers $\mathsf{hom}(F,G)$ (or closely related concepts) are then used to characterize graphs and graphons and to define distances between them.   

\end{itemize}

Invariants inspired by these notions have been considered in contexts closely related to persistence.  In \cite{gomez2023thesis,gomez2021curvature}, a persistence-like invariant of metric spaces was constructed via curvature sets. In \cite{carlsson2013classifying} the authors consider (hierarchical) clustering methods that are induced by a given motif or by a collection of motives (giving rise to \emph{representable} clustering functors). This line was further explored in \cite{carlsson2017representable,carlsson2021robust,memoli2020motivic} in the context of clustering directed graphs and networks and, in \cite{lyu2023sampling}, by directly exploiting homomorphism densities,  for the analysis of network data \cite{motif-sampling-github}. 

\medskip

We now provide a formal definition of motivic invariants, and then discuss how the invariants mentioned in Table \ref{table:comparison} can be seen as particular instantiations of this definition.

\begin{definition}\label{def:order embedding}

    Denote the set of all order-preserving maps from $\catP$ to $\catQ$ as $\Hom(\catP,\catQ)$.
    For $\Pcal$ a set of posets, let $\Hom(\Pcal,\catQ)$ denote the set 
    \[\Hom(\Pcal,\catQ):= \bigcup_{\catP\in \Pcal} \ \Hom(\catP,\catQ).\]
\end{definition}

We denote the collection of all $\catP$-modules as $\mmod\catP$.
For a $\catQ$-module $M$ and $\varphi\in \Hom(\catP,\catQ)$, recall the definition of the pullback $\varphi^\ast M$ (Definition \ref{def:finite encoding}).
\begin{definition}\label{def:motivic invariant}
    Let $\Pcal$ be a set of posets, for which we call $P\in \Pcal$ a \textbf{motif}.
    Let $\Phi\subset \Hom(\Pcal,\catQ)$, and let $F$ be an invariant for $\catP$-modules with codomain a category $\mathcal{D}$.
    A \textbf{motivic invariant of a $\catQ$-module} $M$ defined by the triple $(\Pcal,\Phi,F)$ is the map $\Psi_F(M):\Phi\to \mathcal{D}$, defined by: 
    \[\Psi_F(M)(\varphi):= F\left(\varphi^\ast M\right).\]
    
    The motivic invariant defined by the triple $(\mathcal{P},\Phi,F)$ is the assignment  $M\mapsto \Psi_F(M)$.
 
   We also say that an invariant $F$ for $\catQ$-modules is a motivic invariant if there exists a parametrization $(\Pcal,\Phi,F)$ such that for all $\catQ$-modules $M$, $F(M)$ can be computed from $\Psi_F(M)$ and vice versa.   See Figure \ref{fig:motivic} for an illustration.
\end{definition}

\begin{remark}\label{rem:justification for motivic invariant definition}
    In fact, all invariants for $\catP$-modules are motivic invariants via a trivial parametrization, i.e. a parametrization $(\Pcal,\Phi,F)$ where $|\Pcal|=|\Phi|=1$.
    Namely, if $F$ is an invariant and $M$ is a $\catQ$-module, then by letting $\Pcal = \{\catQ\}$ and $\Phi = \{\id_\catQ\}$, the map $\Psi_F$ in Definition \ref{def:motivic invariant} parametrized by $(\Pcal,\Phi,F)$ is precisely $\Psi_F(M)(\id_Q) = F(M)$.

    \ok{Nonetheless, as we see in the coming examples, numerous invariants from the literature are motivic invariants with nontrivial parametrizations.}
\end{remark}

\begin{example}\label{ex:gri is motivic}
Let $\Pcal = \Con(\catQ)$, $\Psi\subset \Hom(\Pcal,\catQ)$ consist of the canonical embeddings $\iota_I:I\hookrightarrow \catQ$ of connected subsets $I\in \Con(\catQ)$ into $\catQ$, and $F=\rk$, the generalized rank (Equation (\ref{eq:generalized rank of M})).
This parametrization yields a motivic invariant $\Psi_F$, such that for a $\catQ$-module $M$ we have $\Psi_F(M):\Phi\to \Z$ is given by
\[\Psi_\rk(M)(\iota_I) = \rk(M\vert_I).\]
This demonstrates that the GRI of a $\catQ$-module $M$ is a motivic invariant of $M$. 
    In a similar fashion, a restriction of the GRI to any $\Ical\subset \Con(\catP)$ is a motivic invariant.
\end{example}

\begin{example}[Connection to Table \ref{table:comparison}]\label{ex:table 1 motivic}
    The invariants $\rk^\Ical$ for each $\Ical$ in column 3 of Table \ref{table:comparison} are motivic invariants.
    For one example, letting $\Pcal= \{\{\ast\}\}$, $\Phi = \Hom(\{\ast\},\catQ)$, and $F=\rk$ defines a parametrization for the motivic invariant $\Psi_F$, which is equivalent to the dimension function in row 1 of Table \ref{table:comparison}.
\end{example}
Note that a motivic invariant does not necessarily have a unique parametrization.
For instance, if we set $\Pcal:= \{\{p\}: p\in \catP\}$, then the motivic invariant parametrized by $\Pcal$, $\Phi:=\{\iota_p:\{p\}\hookrightarrow \catQ\}$, $F=\rk$ is again the dimension function.

\begin{remark}\label{rem:motivic connection to Amiot}
    Amiot et al. \cite{amiot2024invariants} consider a notion of invariants for $\catP$-modules based on embeddings of posets, resulting in invariants directly related to the notion of motivic invariant we introduce.
    The invariants they study require all posets $\catP\in \Pcal$ to be of finite representation type, and they require that for all $\varphi\in \Hom(\catP,\catQ)$ and $p,p'\in \catP$, $p\leq p'\iff \psi(p)\leq \psi(p')$.
    These added conditions allow them to prove strong results about the invariant they introduce.
    This motivates a study of motivic invariants by focusing on the properties of a motivic invariant ensured by its parametrization(s).
\end{remark}

\subsection{Zigzag-path-indexed barcodes}\label{sec:ZIB general}

We introduce a motivic invariant of $\catP$-modules based on the idea of considering embeddings of zigzag posets.
We begin with preliminary terminology.

Given a poset $\catP$, a \textbf{path} in $\catP$ is a nonempty finite sequence $\Gamma:p_1,p_2,\ldots,p_n$ in $\catP$ such that $p_i\leq p_{i+1}$ or $p_{i+1} \leq p_{i}$ for each $i=1,\ldots  n-1$.
\label{nom: path} 
By inheriting the order on $\catP$, a path $\Gamma$ can be viewed as a zigzag poset $p_1\leftrightarrow p_2 \leftrightarrow \cdots \leftrightarrow p_n$, where $\leftrightarrow$ stands for either $\leq$ or $\geq$. 
For a path $\Gamma$ in $\catP$, with canonical inclusion $\iota_L:\Gamma\hookrightarrow \catP$, and a $\catP$-module $M$, we denote by $M_\Gamma$ the zigzag module $M_\Gamma:=M\circ \iota_L:\Gamma\to \cvec$, i.e. $M_\Gamma = \iota_L^\ast M$.

\begin{definition}\label{def:zigzag-path-indexed barcode general poset}
Given a $\catP$-module $M$, we define the \textbf{zigzag-path-indexed barcode (ZIB)} of $M$ as the map sending each path $\Gamma:p_1,p_2,\ldots,p_n$ in $\catP$ to $\barc(M_\Gamma)$, a multiset of intervals in the zigzag poset $p_1\leftrightarrow p_2 \leftrightarrow \cdots \leftrightarrow p_n$.
\end{definition}

\begin{remark}\label{rem:zib as motivic} 
Let $\Pcal$ be the set of all finite zigzag posets given by paths $\Gamma$ in $\catP$, and let $\Phi$ consist of the canonical embeddings $\iota_\Gamma:\Gamma\hookrightarrow \catP$.
The ZIB of a $\catP$-module $M$ is a motivic invariant with parametrization $(\Pcal, \Phi, \barc)$.

Similarly, the fibered barcode introduced by Cerri et al. \cite{cerri2013betti} and further studied by Lesnick and Wright \cite{lesnick2015interactive}, as well as the pathwise persistence barcode introduce by Neumann et al.  \cite{neumann2022murit} are all motivic invariants.

A zigzag poset is of finite representation type, but the ZIB does not fit into the framework of Amiot et al. \cite{amiot2024invariants}, as the canonical inclusion maps $\varphi_\Gamma:\Gamma\hookrightarrow \catP$ do not satisfy $p\leq p'\iff \varphi_\Gamma(p)\leq \varphi_\Gamma(p')$. Hence, the ZIB example justifies the level of generality in our definition of motivic invariant.
\end{remark}

In Section \ref{sec:ZIB over Z^2}, we will restrict our focus to the ZIB for $\Z^2$-modules, and compare its discriminating power with that of the GRI.

\section{M\"obius invertibility of the generalized rank invariant}
\label{sec:mobius invertible Rk}

In Section \ref{sec:relative to a dictionary}, we eliminate redundancy in the existing assumptions in the literature for defining the GPD and thereby further generalize the notion of a GPD. This yields the notion of \emph{M\"obius invertibility of the GRI}, which is useful to compare various invariants of multi-parameter persistence modules in a unified viewpoint (cf. Table \ref{table:comparison}). In Section \ref{sec:Mobius invertibility and rank decompositions}, we demonstrate that any rank decomposition of the GRI can always be achieved through M\"obius inversion, regardless of the local finiteness of the domain of the GRI. In Section \ref{sec:On structural simplicity of persistence modules}, we compare M\"obius invertibility of the GRI with other concepts regarding the structural simplicity of persistence modules.
In Section \ref{sec:sufficient conditions}, we identify some sufficient conditions that guarantee the M\"obius invertibility of the GRI. As a consequence, the Int-GRI of finitely presentable multi-parameter persistence modules is proved to be M\"obius invertible.

\subsection{Axiomatizing the fundamental lemma of persistent homology}\label{sec:relative to a dictionary}

In this section, we generalize the notion of GPD  (Definition \ref{def:GPD over I}) by (i) dropping the local finiteness assumption on the domain of the GRI (thus also the convolvability assumption on the GRI), and (ii) dissociating the domain of the GRI from that of the GPD.

\begin{definition}\label{def:mobius invertible} Let $M$ be a $\catP$-module and let $\Jcal\subset \Con(\catP)$. Then 
$\rk_M^{\Jcal}$ is said to be \textbf{\invertible} (over $\Ical$) if there exist 
$\Ical\subset \Jcal$ and a map $d_M:\Ical\rightarrow \Z$ such that 
\begin{equation}\label{eq:mobius invertible over I_v2}
\rank^\Jcal_M(I)=\sum_{\substack{J\supset I\\ J\in \Ical}}d_M^{\Ical}(J) \ \ \mbox{for every $I\in \Jcal$,}
\end{equation}
where the RHS contains only finitely many nonzero summands. In particular, if the collection $\Ical$ can be taken as a subset of $\Int(\catP)$, then $\rk_M^\Jcal$ is said to be M\"obius invertible \textbf{over intervals}. 
 \end{definition}

When $\catP$ is a finite totally ordered set, $\Ical=\Jcal=\Int(\catP)$, and $M$ arises from applying the homology functor to a filtration over $\catP$ of a given simplicial complex, $\drm_M^\Ical$ given in Equation (\ref{eq:mobius invertible over I}) coincides with the classical persistence diagram. In particular, in that setting, Equation (\ref{eq:mobius invertible over I}) is known as the \emph{fundamental lemma of persistent homology} \cite[Chapter VII]{edelsbrunner2008computational}. Hence, Definition \ref{def:mobius invertible} axiomatizes the fundamental lemma of persistent homology without any assumptions on $\catP$, $\Con(\catP)$ (or $\Int(\catP)$), $M$, or $\rk_M^{\Jcal}$. In this viewpoint, M\"obius invertibility of $\rk_M^{\Jcal}$ stands for availability of (a generalized version of) the fundamental lemma of persistent homology for $\rk_M^{\Jcal}$.

Next we show that $d_M^{\Ical}$ is acquired via M\"obius inversion, providing justification for the term M\"obius invertibility.
\begin{proposition}\label{prop:justifying the term mobius invertibility} If $\rk_M^\Jcal$ is M\"obius invertible, then there exists a unique minimal locally finite $\Ical\subset \Jcal$ over which $\rk_M^\Ical$ is convolvable and 
\begin{equation}\label{eq:mobius invertible over I}
\rank^\Jcal_M(I)=\sum_{\substack{J\supset I\\ J\in \Ical}}\dgm_M^{\Ical}(J) \ \ \mbox{for all $I\in \Jcal$.}
\end{equation}
\end{proposition}
When Equation (\ref{eq:mobius invertible over I}) holds, we call $\dgm_M^\Ical$ the \textbf{M\"obius inversion of $\rk_M^{\Jcal}$ (over 
$\Ical$).}  From another viewpoint, M\"obius invertibility of the GRI is a relaxation of the interval decomposability of persistence modules, as evidenced by Theorem \ref{thm:sufficient conditions for invertibility} \ref{item:interval-decomposable}. Note also that, when assuming $\Ical=\Jcal$ and $\rk_M^\Jcal$ is convolvable, Definition \ref{def:mobius invertible} reduces to Definition \ref{def:GPD over I}.

\begin{proof}[Proof of Proposition \ref{prop:justifying the term mobius invertibility}]  
      Let $\Ical_1,\Ical_2\subset \Jcal$ be locally finite, and let  $\drm_M^1:\Ical_1\rightarrow \Z$ and $\drm_M^2:\Ical_2\rightarrow \Z$ be any two functions such that 
      \[\forall I\in \Jcal,\ \ \  
 \rk_M^\Jcal(I)=\sum_{\substack{J\supset I\\ J\in \Ical_1}}d_M^1(J)=\sum_{\substack{J\supset I\\ J\in \Ical_2}}d_M^2(J),\]
      and such that $\rk_M^{\Ical_1}$ and  $\rk_M^{\Ical_1}$ are convolvable. Clearly, $\Ical_1\cup \Ical_2$ is also locally finite, and by Remark \ref{rem:convolvability} \ref{item:convolvability5}, $\rk_M^{\Ical_1\cup \Ical_2}$ is  convolvable.
        Now, we regard both $\drm_M^1$ and $\drm_M^2$ as functions $\Jcal\rightarrow \Z$ that vanish outside $\Ical_1$ and $\Ical_2$, respectively. The above equation gives:
        \[\forall I\in \Jcal,\ \ \  0=\sum_{\substack{J\supset I \\ J\in \Ical_1\cup\Ical_2}}(\drm_M^1-\drm_M^2)(J),\]
     which implies that 
     \[\forall I\in \Ical_1\cup \Ical_2,\ \ \  0=\sum_{\substack{J\supset I \\ J\in \Ical_1\cup\Ical_2}}(\drm_M^1-\drm_M^2)(J).\]
     This means that the map  $d_M^1-d_M^2$ is the M\"obius inversion of the zero function on $\Ical_1\cup \Ical_2$. Since the zero function itself is the unique M\"obius inversion of the zero function, we have that $d_M^1-d_M^2$ is the zero function on $\Ical_1\cup \Ical_2$. Also, as both $d_M^1$ and $d_M^2$ vanish outside of $\Ical_1\cup \Ical_2$, the map $d_M^1-d_M^2$ is the zero function on $\Jcal$. Now, the unique minimal $\Ical\subset \Jcal$ in the statement is the support of $d_M^1(=d_M^2)$, completing the proof.
\end{proof}

\begin{remark}[{Comparing Definition \ref{def:mobius invertible} with ideas from \cite[Section 8]{botnan2021signed}}]\label{rem:comparison with BOO21}
\white{.} 
\phantomsection

\begin{enumerate}[label=(\roman*)]
\item When the sets $\Ical$ and $\Jcal$ in Definition \ref{def:mobius invertible} consists solely of intervals of $\catP$ and one contains the other, then saying that $\rk_M^\Jcal$ is M\"obius invertible over $\Ical$ is equivalent to saying that the GRI of $M$ with \emph{test set} $\Jcal$ admits a rank decomposition with \emph{dictionary} $\Ical$ \cite[Section 8]{botnan2021signed}, i.e. there exist $\Ical$-decomposable $\catP$-modules $N_1$ and $N_2$ such that
\[\rk_M^{\Jcal}=\rk_{N_1}^{\Jcal}-\rk_{N_2}^\Jcal.\]
This can be proved in a similar way as Proposition \ref{prop:rk decomposition is GPD}. 
However, Definition \ref{def:mobius invertible} generalizes ideas in \cite[Section 8]{botnan2021signed} in that neither the test set nor the dictionary set needs to be a set of intervals.
This level of generality is useful in building a connection between invariants of persistence modules or clarifying the discriminating power of the GRI: Remarks \ref{rem:rectangle decomposition}, \ref{rem:implications of gen by zz}, Rows 3, 9, 10 of Table \ref{table:comparison}.
\label{item:GPD is more general than rank decomposition}

\item 
In contrast to \cite[Section 8]{botnan2021signed}, we consider the case where $\Ical\subsetneq \Jcal$ to be important both in theory and in practice: 
 Since $\dgm_M^\Ical$ and $\rank_M^{\Ical}$ determine each other (cf. Proposition \ref{prop:properties of GPD} \ref{item:universal property of GPD}), the M\"obius invertibility of \ok{$\rk_M^\Jcal$} over $\Ical$ implies that 
\ok{$\rk_M^\Jcal$} can be restricted to the smaller domain $\Ical$ \emph{without losing any information.} In particular, significant discrepancy between $\Ical$ and $\Jcal$ indicates $\rk^\Jcal_M$ is mostly constant.  
Identifying such an $\mathcal{I}$, ideally the minimal one, is crucial to minimize the computational burden as well as facilitating vectorization of the information for usage in machine learning pipeline \cite{loiseaux2024stable,xin2023gril}.  Consider, for instance, a finitely presentable $\R^d$-module $M$. While the poset $\Jcal:=(\Int(\R^d),\supset)$ is not locally finite, $
\rk_M^\Jcal$ is completely encoded into $\dgm_M^{\Ical}$ 
(Theorem \ref{thm:sufficient conditions for invertibility} \ref{item:finitely presentable invertible}). 
\end{enumerate}
\end{remark}

\begin{remark}\phantomsection\label{rem:monotonicity of mobius invertibility}
Further remarks on Definition \ref{def:mobius invertible} follow.

\begin{enumerate}[label=(\roman*)]

\item (Monotonicity) Let $\Jcal'\subset \Jcal$ and $\Ical\subset \Ical'$.  If $\rk_M^\Jcal$ is M\"obius invertible over $\Ical$, then $\rk_M^{\Jcal'}$ is  M\"obius invertible over $\Ical'$. \label{item:monotonicity of mobius invertibility}

\item 
Assume $\Ical\subset\Jcal$ and $\Ical$ is finite. If $\rk_M^{\Jcal}$ is $\Ical$-constructible, then by Proposition \ref{prop:Mobius inversion of a contructible function}$, \rk_M^{\Jcal}$ is M\"obius invertible over $\Ical$. The converse does not hold. 
\end{enumerate}
\end{remark}

\subsection{Rank decomposition can always be obtained via M\"obius inversion}
\label{sec:Mobius invertibility and rank decompositions}

In this section, we demonstrate that any rank decomposition of the GRI can always be achieved through M\"obius inversion, regardless of the local finiteness of the domain of the GRI. This stands in contrast to the method presented in \cite{botnan2021signed}, where the authors establish the uniqueness of the minimal rank decomposition (if it exists) through two different means, one of which is contingent upon the local finiteness of the domain of the GRI.

We assume $\Jcal=\Int(\catP)$ for ease of notation, bearing Remark \ref{rem:monotonicity of mobius invertibility} \ref{item:monotonicity of mobius invertibility} in mind.

\begin{theorem}\label{thm:rk decomposition is GPD} \contributionn Given any $\catP$-module $M$, regardless of the locally finiteness of $\Int(\catP)$, the following are equivalent.
    \begin{enumerate}[label=(\roman*),leftmargin=*]
        \item $\rank_M^\Int$ is M\"obius invertible. 
        \label{item:1}
         \item $\rank_M^\Int$ admits a rank decomposition.  \label{item:2}     
         \item $\rank_M^\Int$ admits a unique minimal rank decomposition. \label{item:3}
        \item There exists a unique minimal $\Ical\subset \Int(\catP)$ over which $\rank_M^\Int$ is M\"obius invertible.\label{item:4}
        
        \item There exists a unique function $\drm_M:\Int(\catP)\rightarrow \Z$ such that for all $I\in\Int(\catP)$, the set $\{J\supset I: J\in \Int(\catP),\ \ \drm_M(J)\neq 0\}$ is finite and 
        \begin{equation}\label{eq:mobius invertible}
        \rank_M^\Int(I)=\sum_{\substack{J\supset I\\ J\in \Int(\catP)}}\drm_M(J).
        \end{equation}
        \label{item:5} 
    \end{enumerate}
\end{theorem}

\begin{proof}[Proof of Theorem \ref{thm:rk decomposition is GPD}] 
 The implications \ref{item:1} $\Leftrightarrow$  \ref{item:4} $\Leftrightarrow$ \ref{item:5} are direct from Proposition \ref{prop:justifying the term mobius invertibility}.
    \ref{item:1} $\Rightarrow$ \ref{item:2} follows from Proposition \ref{prop:rk decomposition is GPD} and Theorem \ref{thm:BOO's setup is more general}.    
      \noindent\ref{item:2} $\Rightarrow$ \ref{item:1}: Let $(\Rcal,\Scal)$ be a rank decomposition of the Int-GRI of $M$. Then, by Remark \ref{rem:pointwisely finite multiset} \ref{item:pointwisely finite multiset4}, every principal ideal of the poset $\Ical:=\{I\in \Int(\catP): I \mbox{ belongs to either $\mathcal{R}$ or $\mathcal{S}$}\}$ is finite.
      Define $\drm_M^{\Ical}:\Ical\rightarrow \Z$ by $I\mapsto \mult_\Rcal(I)-\mult_\Scal(I)$. Then, for all $I\in \Int(\catP)$:
    \[\rank_M(I)=(\rank_{\kf_{\Rcal}}-\rank_{\kf_{\Scal}})(I)=\sum_{\substack{J\supset I\\ J\in \Ical}}\drm_M^{\Ical}(J).\]
    By Definition \ref{def:mobius invertible}, $\rk_M^\Int$ is M\"obius invertible.

    \noindent \ref{item:2} $\Leftrightarrow$ \ref{item:3} follows from Proposition \ref{prop:if rk decomposes then the minimal one exists}. 
    
    \noindent \ref{item:3} $\Leftrightarrow$ \ref{item:4} can be proved using a similar argument as in the proof of \ref{item:1} $\Leftrightarrow$ \ref{item:2}.
    \end{proof}

\subsection{M\"obius invertibility vs. other structural simplicity measures for persistence modules}\label{sec:On structural simplicity of persistence modules}
In this section, we elucidate the relationship among various concepts pertaining to the structural simplicity of persistence modules, as depicted in Figure \ref{fig:tameness + MI visual}. Along the way, we find a persistence module over $\Z^2$ whose Int-GRI is not M\"obius invertible (which implies that its Int-GRI admits no rank decomposition). 

The following proposition is rather straightforward (cf. Definitions \ref{def:finitely presentable}, \ref{def:finite encoding}, and \ref{def:co-closure and constructible}):
\begin{proposition}\label{prop:constructible implies finitely presentable}
Constructible persistence modules are finitely presentable. 
\end{proposition}

\begin{proof}
Let $M:\catP\rightarrow \cvec$ be an $\catS$-constructible module with $\catS \subset \catP$. Then, there exist an $\catS$-constructible $\catP$-module $F_0=\bigoplus_{i=1}^n k_{p_i^\uparrow}$ and a natural transformation $\Psi^0:F_0\Rightarrow M$ such that each component is surjective, i.e. $\Psi_a^0:F_a\rightarrow M_a$ is surjective for each $a\in \catP$ \cite[Proposition 4.13]{gulen2022galois}. It is easy to see that the kernel of $\Psi^0$ is also $\catS$-constructible and thus there exists another $\catS$-constructible $\catP$-module $F_1=\bigoplus_{i=1}^mk_{q_i^\uparrow}$  with a natural transformation $\Psi^1:F_1\Rightarrow F_0$
such that $\im \Psi^1=\ker\Psi^0$.
Now we have that $M$ is the cokernel of $\Psi^1$ , as desired.
\end{proof}

\begin{proposition}[{\cite[Theorem 6.12, Remark 6.15]{miller2020homological}}]\label{prop:finitely presentable implies tame} Finitely presentable persistence modules are tame.
\end{proposition}

\begin{proposition}\label{prop:Mobius invertibility does not imply tameness}
There exists a non-tame $\catP$-module whose Int-GRI is M\"obius invertible.
\end{proposition}

\begin{proof}
Let $\catP$ be an infinite poset, and consider any interval-decomposable $\catP$-module $M$ that has infinitely many summands, e.g. $M=\bigoplus_{p\in \catP}k_{\{p\}}$. Then, observe that $M$ cannot be finitely encoded and thus $M$ is not tame. However, by Theorem \ref{thm:sufficient conditions for invertibility} \ref{item:finitely encodable}, its GRI is M\"obius invertible.
\end{proof}

We consider the dual statement:

\begin{theorem}\label{thm:tame does not imply Int-GRI invertible}
There exists a tame $\catP$-module $M$ whose Int-GRI is not M\"obius invertible.
\end{theorem}
Before proving this, we remark the following: (i) By Remark \ref{rem:monotonicity of mobius invertibility} \ref{item:monotonicity of mobius invertibility}, this theorem implies that the Con-GRI of a tame persistence module is not necessarily M\"obius invertible. (ii) Theorems \ref{thm:rk decomposition is GPD} and \ref{thm:tame does not imply Int-GRI invertible} imply that the existence of a finite encoding of $M$ does not imply the rank decomposability of the Int-GRI of $M$.

In proving Theorem \ref{thm:tame does not imply Int-GRI invertible}, we will utilize the following well-known concrete formulation for the limit and colimit of any $M:\catP\to \cvec$ (see, for instance, \cite[Appendix E]{kim2021generalized}):
\begin{convention}\phantomsection\label{con:limit and colimit} 
\begin{enumerate}[label=(\roman*),topsep=0pt,itemsep=-1ex,partopsep=1ex,parsep=1ex,leftmargin=*]
    \item The limit of $M$ is the pair
    $(W,(\pi_p)_{p\in \catP})$ described as:
    \[W:=\left\{ (\ell_p)_{p\in \catP}\in \prod_{p\in \catP} M_p: \ \forall p\leq q\in \catP, \varphi_M(p,q)(\ell_p)=\ell_q\right\}\]
    where for each $p\in \catP$, the map $\pi_p:W\rightarrow M_p$ is the canonical projection. Elements of $W$ are called  \textbf{sections} of $M$.
    \item The colimit of $M$ is isomorphic to a pair $(C,(\iota_p)_{p\in \catP})$ that is described as follows.
    For $p\in \catP$, let $\bar{\iota}_p:M_p\to \bigoplus_{p\in \catP}M_p$ be the canonical injection.
    $C$ is the quotient space $\left(\bigoplus_{p\in\catP} M_p\right)/V$, where $V$ is generated by the vectors $\bar{\iota}_p(v_p)-\bar{\iota}_{p'}(v_{p'})$ over all $p\leq p'$ in $\catP$, with $v_p\in M_p,v_{p'}\in M_{p'}$. 
    Letting $f$ be the quotient map from $\bigoplus_{p\in \catP}M_p$ to $C$, for $p\in \catP$, $\iota_p:M_p\to C$ is the composition $f\circ \bar{\iota}_p$.
\end{enumerate}
\end{convention}

\begin{proof}[Proof of Theorem \ref{thm:tame does not imply Int-GRI invertible}]

\begin{figure}
    \centering
    \includegraphics[width=0.99\linewidth]{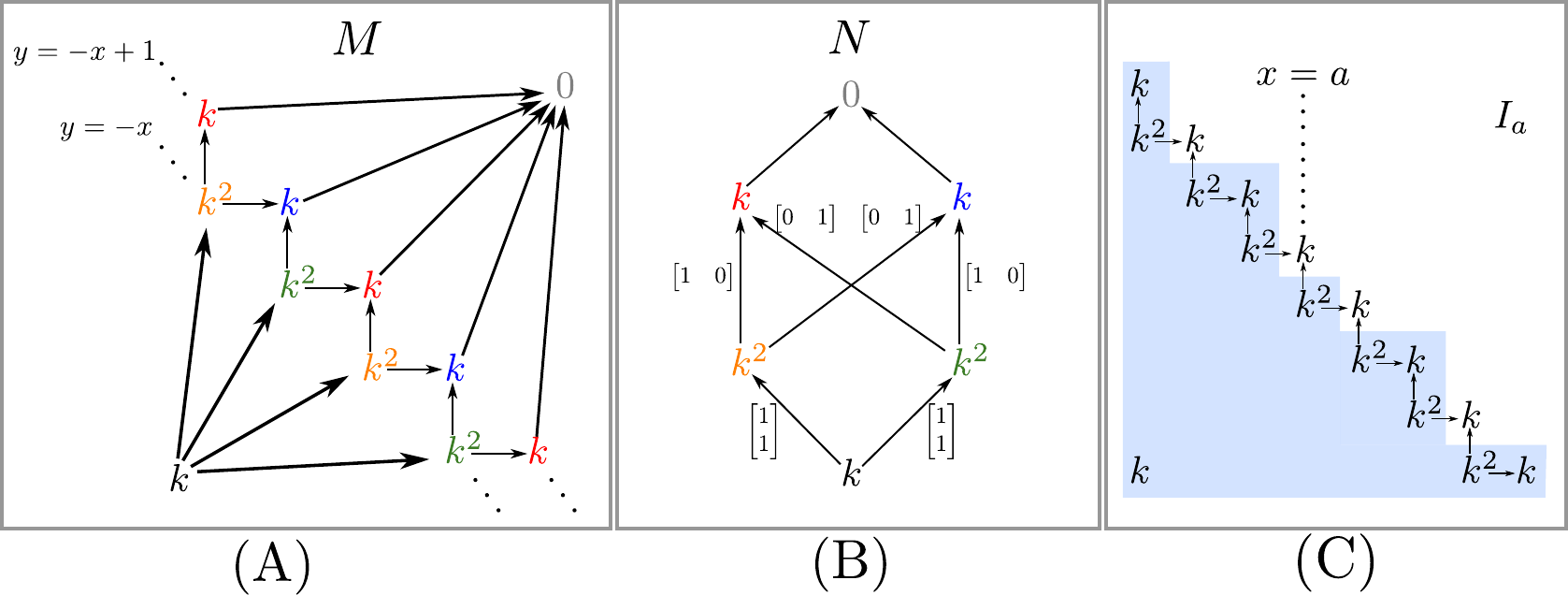}
    \caption{(A) and (B) Illustrations of the $\Z^2$-module $M$ and $\catQ$-module $N$ from the proof of Theorem \ref{thm:tame does not imply Int-GRI invertible}. (C) A visualization of an interval $I_a$ used in the proof of Theorem \ref{thm:tame does not imply Int-GRI invertible}. 
    }
    \label{fig:tame does not imply Int-GRI invertible module}
\end{figure}

    We prove the statement by constructing a $\Z^2$-module that is tame while its Int-GRI is not M\"obius invertible.
    \ok{Let $N$ be the persistence module defined in Figure \ref{fig:tame does not imply Int-GRI invertible module} (B), over the 6-point poset 
    which we denote $\catQ$.
    Define the order-preserving map $\pi:\Z^2\to \catQ$ by mapping like colors, i.e. in Figure \ref{fig:tame does not imply Int-GRI invertible module} red points in (A) map to red points in (B), blue points in (A) map to blue points in (B), etc.
    Define $M:=\pi^\ast N$, which is visualized in Figure \ref{fig:tame does not imply Int-GRI invertible module} (A).
    It is immediate by this definition that $M$ is tame.}

    Now we show that $\rk^\Int_M$ is \emph{not} M\"obius invertible.
    Define the interval $I_0$ as follows:
    \[I_0:=\{(x,y):y\leq-x\} \cup \{(-2i-1,2i+2)\}_{i=0}^\infty \cup \{(2i+2,-2i-1)\}_{i=0}^\infty.\]
    Then for $a\in \Z$, define the interval $I_a$ as the shift of $I_0$ by $(a,-a)$.
    $I_a$ contains the part of $\Z^2$ at or below $y=-x$, and skips every other point along the line $y=-x+1$ except for skipping two consecutive points at $(a,-a+1)$ and $(a+1, -a)$.
    See Figure \ref{fig:tame does not imply Int-GRI invertible module} (C) for a visualization.

    \begin{claim}\label{claim:I_a nonzero rank}
        For each $a\in \Z$, $\rk^\Int_M(I_a) = 1$.
    \end{claim}
    \begin{proof}
        Following Convention \ref{con:limit and colimit}, we first describe a nonzero section of  $M\vert_{I_a}$.
        In fact, a basis for the only section of $M\vert_{I_a}$ which is nonzero everywhere consists of $1\in \kf$ for every copy of $\kf$ appearing in $M\vert_{I_a}$, and $(1,1)\in \kf^2$ for every $\kf^2$ appearing in $M\vert_{I_a}$.
        It is immediate to see this satisfies the definition (Convention \ref{con:limit and colimit}) to be a section.
        It is straightforward to check that the image of this section in the colimit is nonzero, and as a result $\rk^\Int_M(I_a)\geq 1$.

        Lastly, as $(-1,-1)\in I_a$ and $M_{(-1,-1)}=\kf$ has dimension 1, we know $\rk^\Int_M(I_a)\leq 1$, and thus $\rk^\Int_M(I_a)=1$.
    \end{proof}

    \begin{claim}\label{claim:I_a maximal}
        For $a\in \Z$ and $J\in \Con(\catP)$ such that $J\supsetneq I_a$, $\rk_M^\Con(J) = 0$.
    \end{claim}
    \begin{proof}
        Let $J\in \Con(\catP)$ such that $J\supsetneq I_a$ for some fixed $a\in \Z$.
        Then there are two possibilities.
        First, suppose $J$ contains a point $(x,y)$ with $y>-x+1$.
        In this case, it is immediate $\rk^\Con_M(J)=0$ as $M(x,y) = 0$.

        In the second case, suppose $J$ contains a point $(x,y)\notin I_a$ with $y=-x+1$.
        By definition of $I_a$, we have that at least one of $(x+1,y-1)$ or $(x-1,y+1)$ must be in $J$ as well.
        This creates a sub-diagram of $M\vert_J$ of the form $\kf \xleftarrow{\begin{bmatrix} 0 & 1\end{bmatrix}} \kf^2 \xrightarrow{\begin{bmatrix}1 & 0\end{bmatrix}} \kf$.
        The rank over this diagram is 0, and thus it follows by Remark \ref{rem:properties of rank invariant} \ref{item:monotonicity} that $\rk^\Con_M(J)=0$.
    \end{proof}

    Assume by way of contradiction that $\rk^\Int_M$ is M\"obius invertible, i.e. 
    that there exists a map $d_M:\Int(\Z^2)\rightarrow \Z$ as stated in Theorem \ref{thm:rk decomposition is GPD} \ref{item:5}.

We claim that for any $a\in \Z$ there is no $J\in \Int(\Z^2)$ such that $I_a\subsetneq J$ and $d_M(J)\neq 0$.
Suppose not, i.e. there exists a fixed $a\in \Z$    and $J_1\in \Ical$ with $I_a\subsetneq J_1$ and $d_M(J_1)\neq 0$.
    By Claim \ref{claim:I_a maximal}, we know $\rk_M(J_1) = 0$, but as $d_M(J_1)\neq 0$, this alongside Equation (\ref{eq:mobius invertible}) implies there must exist $J_2\in \Ical$ with $J_1\subsetneq J_2$, and $d_M(J_2)\neq 0$.
    Again, Claim \ref{claim:I_a maximal} alongside Equation (\ref{eq:mobius invertible}) imply the existence of a $J_3\in \Ical$ with $J_2\subsetneq J_3$ and $d_M(J_3)\neq 0$.
    Repeating this procedure, it follows that there exists an infinite ascending chain of intervals $I_a\subsetneq J_1\subsetneq J_2\subsetneq\ldots$, such that $d_M(J_i)\neq 0$ for all $i\geq 1$. 
    This implies that the set $\{J\supset I_a: J\in \Int(\catP),\ \ \drm_M(J)\neq 0\}$ is infinite, contradicting Theorem \ref{thm:rk decomposition is GPD}  \ref{item:5}.
  
    Therefore, for any $a\in \Z$ there is no $J\in \Ical$ such that $J\supsetneq I_a$ and $d_M(J)\neq 0$.
    This fact alongside Claim \ref{claim:I_a nonzero rank} implies that $d_M(I_a)=1$.

 Now, observe that for all $a\in \Z$, the interval $I_a$ contains the point $(-1,-1)\in \Z^2
$. This implies that for the singleton interval $I=\{(-1,-1)\}$, the set $\{J\supset I: J\in \Int(\catP),\ \ \drm_M(J)\neq 0\}$ is infinite, contradicting Theorem \ref{thm:rk decomposition is GPD}  \ref{item:5} and completing the proof.
\end{proof}

\begin{remark}\label{rem:higher dimension}
The module $M:\Z^2\to \cvec$ from the previous gives rise to infinitely many modules which have non-M\"obius invertible Int-GRI's.
For example, let $P$ be an any poset into which $\Z^2$ can be embedded via some $f:\Z^2\hookrightarrow \catP$ such that $f(\Z^2)\in \Int(\catP)$ (e.g. $\catP=\Z^d$ for $d>2$).
Define a $\catP$-module $N$ by setting $N\vert_{f(\Z^2)}$ as the push-forward of $M$ along $f$, and setting $N:=0$ outside of $f(\Z^2)$.
Then $N$ has a non-M\"obius invertible Int-GRI.

As another example, we can naturally extend $M$ to an $\R^2$-module $M'$
where $M'_{p}:=M_{\lfloor p \rfloor_{\Z^2} }$ for all $p\in \Z^2$. Then, the Int-GRI of $M$ is not M\"obius invertible.

\end{remark}

\begin{remark}\label{rem:implication of the counterexample} 
For the $\Z^2$-module $M$ considered in the previous proof, we observe the following.
\begin{enumerate}[label=(\roman*)]
\item We have actually showed that the $\Int$-GRI of $M$ is not M\"obius invertible over \emph{any} subcollection of $\Con(\catP)$ containing $\Int(\catP)$.\label{item:implication of the counterexample1} 

\item By Item \ref{item:implication of the counterexample1} and Remark \ref{rem:monotonicity of mobius invertibility} 
\ref{item:monotonicity of mobius invertibility}, the Con-GRI of $M$ is also not M\"obius invertible.
\item The map $\pi$ is \emph{not} `rank-preserving', e.g. $\rk_M(I_a)=1\neq 0=\rk_N(\pi(I_a))$ for each $a\in \Z$. \label{item:not rank-preserving}
\end{enumerate}
\end{remark}

It is somewhat surprising that even for a persistence module $M$ that is simple enough to admit a finite encoding, both the Con-GRI and Int-GRI of $M$ might not be M\"obius invertible. Motivated by Remark \ref{rem:implication of the counterexample} \ref{item:not rank-preserving}, in the next section, we adapt the notion of finite encoding, aimed at ensuring the M\"obius invertibility of the GRI. 

\subsection{Sufficient conditions for M\"obius invertibility of the GRI}\label{sec:sufficient conditions}
In this section, we explore several conditions on persistence modules $M$ that ensure the M\"obius invertibility of the Con- or Int-GRI of $M$. 

Let $\pi:\catP\to \catQ$ be an order-preserving map.
If for all $I\in \Int(\catP)$, the image $\pi(I)$ belongs to $\Int(\catQ)$, then we say $\pi$ is \emph{interval-preserving}.
Let $M$ be a $\catP$-module and $N$ be a $\catQ$-module. We say that $\pi$ is Int\emph{-rank-preserving} (from $M$ to $N$), if $\forall I\in \Int(\catP)$,  $\rk_M(I) = \rk_N(\pi(I))$. 

\begin{theorem}\label{thm:sufficient conditions for invertibility} \contributionn Each of the following assumptions implies that {the Int-GRI of} a given $\catP$-module $M$ is \invertible{} over intervals.
\begin{enumerate}[label=(\roman*),leftmargin=*]
  \item  $M$ is interval-decomposable. 
    \label{item:interval-decomposable}
    \item There exist a finite poset $\catQ$, an order-preserving map $\pi:\catP\rightarrow \catQ$, and a $\catQ$-module $N$ such that $\pi$ is interval-preserving and Int-rank-preserving.
   \label{item:finitely encodable}  
   \item $\catP=\R^d$ and $M$ is finitely presentable (we refine this statement in Proposition \ref{prop:finitely presentable invertiver over intmn}).\label{item:finitely presentable invertible}
\end{enumerate}
\end{theorem}

\begin{remark}\phantomsection\label{rem:about thm sufficient conditions for invertibility}

Although the condition on the map $\pi$ given in Item \ref{item:finitely encodable} seems weaker than that of a \emph{finite encoding}, it is actually not: see Remark \ref{rem:implication of the counterexample} \ref{item:not rank-preserving}. Conversely, $\pi$ described in Item \ref{item:finitely encodable} is also not necessarily a finite encoding: Consider any pair of persistence modules $M,N$ over the same finite poset $\catP$ that are not isomorphic, but have the same Int-GRI. Then the identity map $\id_{\catP}$ on $\catP$ is interval-preserving and Int-rank-preserving, but $\id_{\catP}$ is not a finite encoding.

\ok{Nevertheless, natural candidates for such $\pi$ are finite encodings:}
see Example \ref{ex:anti-diagonal}. 
\end{remark}

Theorem \ref{thm:sufficient conditions for invertibility} \ref{item:interval-decomposable} and \ref{item:finitely encodable} respectively correspond to implications (3) and (4) of Figure \ref{fig:tameness + MI visual}. Within the proof of Theorem \ref{thm:sufficient conditions for invertibility} \ref{item:finitely presentable invertible}, Implication (2) of Figure \ref{fig:tameness + MI visual} is established.

Now we prove Theorem \ref{thm:sufficient conditions for invertibility}. Item \ref{item:interval-decomposable} is  straightforward: Define $\drm_M:\Int(\catP)\rightarrow \Z$ by sending each interval $I$ to the multiplicity of $I$ in the barcode of $M$. 
Then, the function $\drm_M$ satisfies the condition given in item \ref{item:5} of Theorem \ref{thm:rk decomposition is GPD}, completing the proof. In proving Item \ref{item:finitely encodable}, the proposition below is useful.

\begin{proposition}\label{prop:each component is an interval}
Let $\pi:\catP\rightarrow \catQ$ be an order-preserving map and let $J$ be an interval of $\catQ$. Then $\pi^{-1}(J)$ is a disjoint union of intervals of $\catP$.
\end{proposition}
\begin{proof}
It suffices to show that $\pi^{-1}(J)$ is convex (Definition \ref{def:intervals} \ref{item:convexity}). Let $p,r\in \pi^{-1}(J)$ and $q\in \catP$ with $q\in [p,r]$. Since $\pi$ is order-preserving, we have that $\pi(q)\in [\pi(p),\pi(r)]\subset J$. Hence, $q$ belongs to $\pi^{-1}(J)$, as desired.
\end{proof}

\begin{proof}[Proof of Theorem \ref{thm:sufficient conditions for invertibility}
\ref{item:finitely encodable}] 

Since \ok{$\catQ$ is finite, the set} $\pi^{-1}\Int(\catQ):=\{\pi\inv(J):J\in \Int(\catQ)\}$ is finite. For $J\in \Int(\catQ)$, let $C(\pi\inv(J))$ be the set of connected components of $\pi\inv(J)$. \ok{This implies that any pair of elements in $C(\pi\inv(J))$ are disjoint. Also,}
by Proposition \ref{prop:each component is an interval}, every element of $C(\pi\inv(J))$ is in $\Int(\catP)$.

 Let $\Ical := \bigcup_{J\in \Int(\catQ)} C(\pi\inv(J))$. \ok{For each $I\in \Int(\catP)$,} the set
\[\Ical_{\supset I}:=\{J\in \Ical:J\supset I\}
\]
is finite, \ok{because for each $J\in\Int(\catQ)$, the intersection $\Ical_{\supset I} \cap C(\pi^{-1}(J))$ contains at most one element. This proves that} $\rk^\Ical_M$ is convolvable. 

Next, for $I\in \Int(\catP)$, let $I^\Ical\in \Ical$ be
the connected component of $\pi\inv(\pi(I))$ containing $I$.
Since $\pi(I^\Ical)=\pi(I)$ \ok{and $\pi$ is $\Int$-rank-preserving,}
we have $\rk_M(I)=\rk_M(I^\Ical)$. 
\ok{Also, observe that for any $J\in \Ical$, we have $I\subset J$ if and only if $I^\Ical\subset J$. Therefore,}

\[\rk_M(I) = \rk_M(I^\Ical) = \sum_{\substack{J\supset I^\Ical\\ J\in \Ical}} \dgm_M^\Ical(J) = \sum_{\substack{J\supset I\\ J\in \Ical}} \dgm_M^\Ical(J), \]
i.e. $\rk^\Int_M$ is M\"obius invertible over $\Ical\subset \Int(\catP)$, as desired.
\end{proof}

\ok{We now prove Theorem \ref{thm:sufficient conditions for invertibility} \ref{item:finitely presentable invertible}, building on Theorem \ref{thm:sufficient conditions for invertibility} \ref{item:finitely encodable}. 
When a given $\R^d$-module $M$ is finitely presentable, a natural finite encoding of $M$ exists. Proving that the finite encoding is Int-rank-preserving demands careful scrutiny.} 

\begin{proof}[Proof of Theorem \ref{thm:sufficient conditions for invertibility} \ref{item:finitely presentable invertible}] 
Let $M$ be a finitely presentable $\R^d$-module.
By assumption, $M$ is the cokernel of a morphism $\bigoplus_{a\in A} \kf_{a^\uparrow}\rightarrow \bigoplus_{b\in B}\kf_{b^\uparrow}$ where $A$ and $B$ are some finite multisets of elements from $\R^d$. Let $G_1,\ldots,G_d$ be finite subsets of $\R$ such that $\catQ'=\Pi_{i=1}^d G_i\subset 
\R^d$ is the smallest grid including all the elements in $A$ and $B$.  Let $q_\ast$ be the unique minimal element of $\catQ'$. Let $\catQ:=\catQ'\cup\{-\infty\}$, where we declare that $-\infty<q$ for all $q\in \catQ'$. We define the $\catQ$-module $N$ by $N|_{\catQ'}:=M|_{\catQ'}$ and $N_{-\infty}:=0$.\footnote{It is noteworthy that $M$ is the left-Kan extension of $N$ along the canonical inclusion $Q'\hookrightarrow \R^d$.} 

Let $\lfloor - \rfloor_{\catQ}:\R^d\rightarrow \catQ$ 
be the map sending each $p\in \R^d$ to the maximal $q\in \catQ$ such that $q\leq p$ (we declare that $-\infty<p$ for all $p\in \R^d$). 
\ok{By Theorem \ref{thm:sufficient conditions for invertibility} \ref{item:finitely encodable}, it suffices to show that $\lfloor - \rfloor_\catQ$ is interval-preserving, and $\Int$-rank-preserving with respect to $M$ and $N$.}
\ok{It is not difficult to see that $\lfloor - \rfloor_{\catQ}$ is interval-preserving, and thus we only show that $\lfloor - \rfloor_{\catQ}$ is $\Int$-rank-preserving.}
\ok{First, we make the following observations regarding the map $\lfloor - \rfloor_{\catQ}$.} 
\begin{enumerate}
\item[1.] For any $\ok{p\leq p'\in \R^d}$, if
 $\lfloor p \rfloor_{\catQ} = \lfloor p' \rfloor_{\catQ}$, then $\varphi_M(p,p'):M_p \rightarrow M_{p'}$ is the identity. \label{item:floor1}
 \item[2.]  $\R^d$ is partitioned 
    into the preimages $B_q$ of $q\in\catQ$ under $\lfloor-\rfloor_Q$. In particular, 
    \[B_q= \begin{cases} \R^d\setminus q_\ast^\uparrow,& \mbox{if $q=-\infty$} \\  \prod_{i=1}^d [a_i,b_i),&\mbox{if $q\neq -\infty$}, \end{cases} \]
    for some intervals $[a_i,b_i)\subset \R$. Each $B_q$ will be called a \emph{block}.
\end{enumerate}

If $I\in \Int(\R^d)$ is not contained in $q_\ast^\uparrow=\{p\in \R^d:q_\ast\leq p\}$, then it is clear that
\[\rk_M(I) = \rk_N(\lfloor I \rfloor_{\catQ})=0.\]

Now it remains to prove:
\begin{equation}\label{eq:corollary Rd int}
    \mbox{for all intervals}\ I\subset q_\ast^\uparrow,\ \  \rk_M(I) = \rk_N(\lfloor I \rfloor_{\catQ}).
\end{equation}

Let $I\in \Int(\R^d)$ with $I\subset q_\ast^\uparrow$. Let $\Psi_M$ be the canonical limit-to-colimit map over $M\vert_I$ and $\Psi_N$ be the canonical limit-to-colimit map over $N\vert_{\lfloor I\rfloor_\catQ}$. By the rank-nullity theorem, it suffices to prove that   $\varprojlim M\vert_I\cong \varprojlim N\vert_{\lfloor I \rfloor_{\catQ}}$ and $\ker (\Psi_M)\cong \ker (\Psi_N)$. 
   
Consider the map $\Phi:\varprojlim M\vert_I\to \varprojlim N\vert_{\lfloor I \rfloor_{\catQ}}$ by, for each $(\ell_p)_{p\in I}\in \varprojlim M\vert_I$ (cf. Convention \ref{con:limit and colimit}),
    \[\Phi((\ell_p)_{p\in I}):=(j_q)_{q\in \lfloor I \rfloor_{\catQ}}\ ,\] where $j_q:=\ell_p$ for some $p\in I$ such that $\lfloor p \rfloor_{\catQ} = q$ (see Figure \ref{fig:rd mobius invertible} (A)). It suffices to prove the following subclaims:
    
\textbf{Subclaim 1.} $\Phi$ is well-defined. 

\textbf{Subclaim 2.} $\Phi$ is a linear isomorphism.

\textbf{Subclaim 3.} The restriction of $\Phi$ to $\ker(\Psi_M)$ is a bijection with $\ker (\Psi_N)$.   
\begin{proof}[Proof of Subclaim 1]     Obs.1 implies that $j_q$ does not depend on the choice of $p$. 
    Next, we show that for $(\ell_p)_{p\in I}\in \varprojlim M\vert_I$, $\Phi((\ell_p)_{p\in I})\in \varprojlim N\vert_{\lfloor I \rfloor_\catQ}$. 
    To this end, it suffices to show that if $q\leq q'\in \lfloor I \rfloor_{\catQ}$, then $\varphi_N(q, q')(j_q) = j_{q'}$.
    Let $p,p'\in I$ such that $\lfloor p \rfloor_{\catQ} = q$ and $\lfloor p' \rfloor_{\catQ} = q'$.
    If $p\leq p'$, then $\varphi_N(q, q')(j_q) = j_{q'}$ immediately follows from the fact that $\varphi_M(p, p')(\ell_p) = \ell_{p'}$.

    Thus, suppose that there is no comparable pair $p,p'\in I$ with $\lfloor p\rfloor_\catQ = q$ and $\lfloor p'\rfloor_\catQ = q'$.
    For an example of this scenario, see Figure \ref{fig:rd mobius invertible} (A).
    Since $\lfloor p \rfloor_{\catQ}< \lfloor p' \rfloor_{\catQ}$, we must have that $p\in B_1$ and $p'\in B_n$ are in different blocks $B_1 \neq B_n$.
    We say two blocks $B\neq B'$ are \emph{adjacent} if $B=\prod_{i=1}^d [a_i,b_i)$ and $B'=\prod_{i=1}^d [a'_i,b'_i)$, with $a_i=a_i'$ and $b_i=b_i'$ for all but one $1\leq i\leq d$, for which $b_i = a'_i$.    
    If $B$ and $B'$ are adjacent, with $a_i<a_i'$ for some $1\leq i\leq d$, then we write $B\leq B'$.    
    As $I$ is an interval, there exists a finite sequence of adjacent blocks $B_1\leq B_2\leq \ldots\leq B_{n-1}\leq B_n$, that each intersects $I$.

    Observe that the convexity of $I$ implies if $B$ and $B'$ are adjacent blocks, there must exist $p\in B\cap I$ and $p'\in B'\cap I$ such that $p\leq p'$.
    Thus, for $2\leq i\leq n-1$, we can find pairs of points $p_i,p_i'\in B_i\cap I$ such that $p_{i-1}'\leq p_i$ for $2\leq i\leq n$, further having the property that $p_1 = p$ and $p_n' = p'$.    
    As noted before, since $p_i,p_i'$ are in the same block, we have $\ell_{p_i} = \ell_{p_i'}$ for $1\leq i\leq n$.
    Since $B_i\cap I$ is an interval, there must exist a path connecting $p_i$ and $p_i'$ inside $B_i\cap I$.
    Thus, we have shown that for any $q,q'\in \lfloor I \rfloor_\catQ$, and $p,p'\in I$ such that $\lfloor p\rfloor_\catQ = q$ and $\lfloor p' \rfloor_\catQ = q'$, there exists a path $\Gamma$ of elements of $I$ connecting $p$ and $p'$.
    For an example of such a path $\Gamma$, see Figure \ref{fig:rd mobius invertible} (A).

    \begin{figure}
        \begin{center}
            \includegraphics[width=0.53\linewidth]{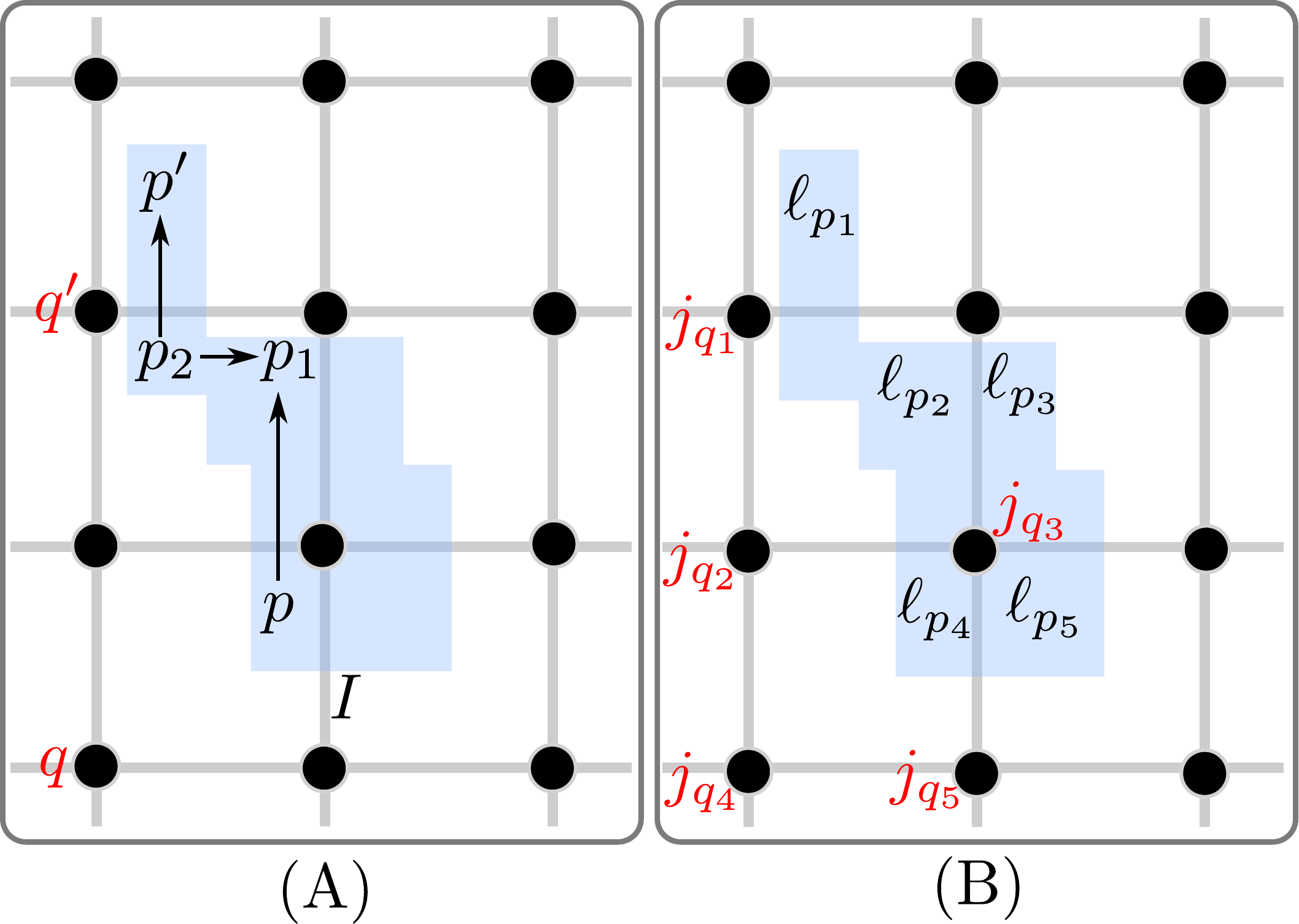}
            \caption{(A) Illustration for $q=\lfloor p \rfloor_Q$ , $q'=\lfloor p' \rfloor_Q$, and an interval $I\in \R^d$ which is the shaded region. Any section over $I$ is mapped to a section over $\lfloor I\rfloor_\catQ$. The existence of a path in $Q$ connecting $q$ and $q'$ guarantees the existence of a path in $I$, which is exemplified by $p'\leftarrow p_2\rightarrow p_1
            \leftarrow p$.          
            (B) A section $(\ell_p)_{p\in I}$ its image $(j_q)_{q\in \lfloor I \rfloor_\catQ}:= \Phi((\ell_p)_{p\in I})$.}
            \label{fig:rd mobius invertible}
        \end{center}
    \end{figure}
    
    Let $q_i:=\lfloor p_i \rfloor_{\catQ}$ for $1\leq i\leq n$.
    It is clear that we have $q = q_1\leq q_2\leq \ldots \leq q_{n-1}\leq q_n = q'$.
    By the fact that $(\ell_p)_{p\in I}$ is a section, we have $\varphi_M(p_{i-1}', p_i)(l_{p_{i-1}'}) = l_{p_i}$ for $2\leq i\leq n$, which implies $\varphi_N(q_{i-1}, q_i)(j_{q_{i-1}}) = j_{q_i}$ for $2\leq i\leq n$.
    By composing the internal morphisms in $N$, this gives us $\varphi_N(q, q')(j_q) =  \varphi_N(q_1, q_n)(j_{q_1}) = j_{q_n} = j_{q'}$.
    Hence, $\Phi((\ell_p)_{p\in \catP}=(j_q)_{q\in \lfloor I \rfloor_\catQ} \in \varprojlim N\vert_{\lfloor I\rfloor_\catQ}$.
    For an example of $\Phi$, see Figure \ref{fig:rd mobius invertible} (B).
    \end{proof}

    \begin{proof}[Proof of Subclaim 2] 
    We construct the inverse $\Phi\inv:\varprojlim N\vert_{\lfloor I\rfloor_\catQ}\to \varprojlim M\vert_I$.
    For $(j_q)_{q\in \lfloor I \rfloor_\catQ}$ a section in $\varprojlim N\vert_{\lfloor I\rfloor_\catQ}$, define $\Phi\inv((j_q)_{q\in \lfloor I \rfloor_\catQ}):= (\ell_p)_{p\in I}$, where $\ell_p :=j_q$ for $q = \lfloor p \rfloor_\catQ$.
    It is straightforward to see that $\Phi\inv$ maps sections to sections as if $p\leq p'\in I$ then $\lfloor p \rfloor_\catQ\leq \lfloor p'\rfloor_\catQ$.
    Lastly, it is immediate to see by the definitions that $\Phi\circ \Phi\inv$ is the identity on $\varprojlim N\vert_{\lfloor I\rfloor_\catQ}$ and $\Phi\inv\circ \Phi$ is the identity on $\varprojlim M\vert_I$.
    \end{proof}

    \begin{proof}[Proof of Subclaim 3.]
    We have already demonstrated that $\Phi$  injective, so it remains to show that $\Phi(\ker \Psi_M) = \ker \Psi_N$. 
    To see this, suppose $(\ell_p)_{p\in I}\in \ker \Psi_M$.
    By Lemma \ref{lem:path characterization of zero in colim}, this means that there exists $p^\ast\in I$ and a path $\Gamma$ of comparable elements in $I$ connecting $p^\ast$ to some $p'\in I$, and a section $(m_p)_{p\in I}\in \varprojlim M\vert_I$ with $m_{p^\ast}=\ell_{p^\ast}$ and $m_{p'}=0$.
    Define $(j_q)_{q\in \lfloor I\rfloor_\catQ}:= \Phi((\ell_p)_{p\in I})$ and $(h_q)_{q\in \lfloor I\rfloor_\catQ}:= \Phi((m_p)_{p\in I})$.
    Then if we denote $q^\ast:=\lfloor p^\ast \rfloor_\catQ$, $q':=\lfloor p'\rfloor_\catQ$ and $\Gamma_\catQ:=\lfloor \Gamma \rfloor_\catQ$, we see that $h_{q*} = \ell_{p^\ast}$, $h_{q'} = 0$, and $\Gamma_\catQ$ is a path of comparable elements in $\lfloor I\rfloor_\catQ$ connecting $q^\ast$ to $q'$, so $\Psi_N((j_q)_{q\in \lfloor I\rfloor_\catQ}) = 0$.
    Hence, $\Phi((\ell_p)_{p\in I})\in \ker \Psi_N$, meaning $\Phi(\ker \Psi_M)\subset \ker \Psi_N$.
    
    On the flip side, suppose we have a section $(j_q)_{q\in \lfloor I\rfloor_\catQ}\in \ker \Psi_N$. 
    Again, this means that there exists $q^\ast \in \lfloor I\rfloor_\catQ$, a path of comparable elements in $\lfloor I\rfloor_\catQ$ connecting $q^\ast$ to some $q'\in \lfloor I \rfloor_\catQ$, and a section $(h_q)_{q\in \lfloor I \rfloor_\catQ}$ such that $h_{q^\ast} = j_{q^\ast}$ and $h_{q'} = 0$.
    We can define $(\ell_p)_{p\in I}:=\Phi\inv((j_q)_{q\in \lfloor I\rfloor_\catQ})$ and $(m_p)_{p\in I}:=\Phi\inv((h_q)_{q\in \lfloor I\rfloor_\catQ})$.
    Fix $p^\ast,p'\in I$ such that $\lfloor p^\ast\rfloor_\catQ = q^\ast$ and $\lfloor p'\rfloor_\catQ = q'$.
    Then we have $\ell_{p^\ast} = m_{p^\ast}$ and $m_{p'} = 0$.
    When we showed $\Phi$ is well-defined, we demonstrated the existence of a (non-unique) path $\Gamma$ in $I$ connecting $p^\ast$ and $p'$, and so $(\ell_p)_{p\in I}\in \ker \Psi_M$.
    As $\Phi\inv$ is the inverse of $\Phi$, this implies $\Phi(\ker \Psi_M)\supset \ker \Psi_N$, and so $\ker \Psi_M\cong \ker \Psi_N$.\end{proof}   

\end{proof}

\begin{remark}\phantomsection\label{rem:rectangle decomposition} 

By Remark \ref{rem:monotonicity of mobius invertibility} \ref{item:monotonicity of mobius invertibility} and Theorem \ref{thm:sufficient conditions for invertibility} \ref{item:finitely presentable invertible}, the GRI over \emph{any} $\Jcal\subset\Int(\R^d)$ of any finitely presentable $\R^d$-module $M$ is M\"obius invertible over intervals (equivalently, by Remark \ref{rem:comparison with BOO21} \ref{item:GPD is more general than rank decomposition}, $\rk_{M}^\Jcal$ admits a rank decomposition).
This generalizes the fact that 
the RI of any finitely presentable $\R^d$-module admits a rank decomposition \cite[Corollary 5.6]{botnan2021signed}.
\label{item:rank decomposition over special dictionary sets}
\end{remark}

For the Con-GRI, the statement analogous to  Theorem \ref{thm:sufficient conditions for invertibility} \ref{item:finitely encodable} also holds:

\begin{proposition}\label{prop:tame+ does imply MI}

    Let $M$ be a $\catP$-module.
    If 
    there exists a finite poset $\catQ$, a poset morphism $\pi:\catP\to \catQ$ and $\catQ$-module $N$ such that $\pi$ is rank-preserving, i.e. for all $I\in \Con(\catP)$, $\rk_M(I) = \rk_N(\pi(I))$, then $\rk^\Con_M$ is M\"obius invertible. 
    
\end{proposition}

\begin{proof}
For any $J\subset \catQ$, let $C(\pi\inv(J))$ be the set of connected components of $\pi\inv(J)$.
This implies $C(\pi\inv(J))\subset \Con(\catP)$.
Define $\Ical\subset \Con(\catP)$ as $\Ical:=\cup_{J\in \Con(\catQ)} C(\pi\inv(J))$.
From here, the argument that $\rk_M^\Con$ is M\"obius invertible over $\Ical$ follows the same as the proof of Theorem \ref{thm:sufficient conditions for invertibility}
\ref{item:finitely encodable} upon replacing all instances of $\Int(-)$ with $\Con(-)$.\qedhere
\end{proof}
Interestingly, the statement analogous to Theorem \ref{thm:sufficient conditions for invertibility} \ref{item:finitely presentable invertible} for the Con-GRI is false:

\begin{remark}\label{rem:finitely presented R^d not Con MI}
       There exist finitely presentable $\R^d$-modules whose Con-GRIs are not M\"obius invertible.   The construction of such an example is parallel to the one described in the proof of Theorem \ref{thm:tame does not imply Int-GRI invertible}, and thus we omit it. In particular, when constructing such an example, the map $\lfloor - \rfloor_\catQ$ considered in the proof of Theorem \ref{thm:sufficient conditions for invertibility} \ref{item:finitely presentable invertible} is not Con-rank-preserving, and thus we cannot utilize Proposition \ref{prop:tame+ does imply MI} in that case. 
\end{remark}

\paragraph{Pulling-back of the GPD.} Theorem \ref{thm:sufficient conditions for invertibility} \ref{item:finitely encodable} and Proposition \ref{prop:tame+ does imply MI} show that the presence of a rank-preserving order-preserving map $\pi:\catP\rightarrow\catQ$ with $\abs{\catQ}<\infty$ ensures the M\"obius invertibility of the GRI. Next, we will demonstrate that by imposing an additional assumption on $\pi$, the GPD of $M$ can be fully determined by the GPD of $N$.
 
 An order-preserving map $\pi:\catP\rightarrow \catQ$ is called a \textbf{covering} if $\pi$ is surjective and for each $J\in \Con(\catQ)$, $\pi^{-1}(J)$ is a disjoint union of connected subposets $I\subset \catP$ such that $\pi(I)=J$.
We say that $\pi$ is Con\emph{-rank-preserving} if for all $I\in \Con(\catP)$, $\rk_M(I) = \rk_M(\pi(I))$.
\begin{theorem}\label{thm:pull-back GPD} Let $M$ be a $\catP$-module and let $N$ be a $\catQ$-module with a covering $\pi:\catP\rightarrow \catQ$ that is Con-rank-preserving. 
Then, the Con-GRI of $M$ is M\"obius invertible, and
 for $I\in \Con(\catP)$, 
\[\dgm_M(I)=\begin{cases}\dgm_N(\pi(I)),&\mbox{if $I$ is a connected component of $\pi^{-1}\pi(I)$}\\0,&\mbox{otherwise.}\end{cases}\]

Also, the analogous statement obtained by replacing every $\Con$ by $\Int$ holds,  
under the extra assumption that $\pi$ is interval-preserving.
\end{theorem}

\begin{proof} 
We only prove the statements for the Con-GRI.   
 Because $\pi$ is a covering and rank-preserving, we can extend the domain $\Con(\catP)$ of $\rk_M$  to $\Con(\catP)\cup \pi^{-1}\Con(\catQ)$ by:
\[I\mapsto\begin{cases} \rk_M(I),&\mbox{$I\in \Con(\catP)$}\\ \rk_M(I'),&\mbox{$I\in\pi^{-1}\Con(\catQ)$  and $I'$ is \emph{any} connected component of $I$.}
\end{cases}
\]
Indeed, this map is  well-defined even on the intersection $\Con(\catP)\cap \pi^{-1}\Con(\catQ)$. For the containment relation $\supset$ on $\Con(\catP)\cup \pi^{-1}\Con(\catQ),$ we claim that $\rk_M$ is $\pi^{-1}\Con(\catQ)$-constructible. Indeed, the rank-preserving property of $\pi$ and the definition of $\rk_M$ on $\pi^{-1}\Con(\catQ)$ imply that for any $I\in \Con(\catP)$, 
\[\rk_M(I)=\rk_N(\pi(I))=\rk_M(\pi^{-1}\pi(I)),\]
showing that the map $\pi^{-1}\pi$ on  $\Con(\catP)\ \cup\ \pi^{-1}\Con(\catQ)$ is the co-closure with image $\pi^{-1}\Con(\catQ)$.
Since $\catQ$ is finite, $\pi^{-1}\Con(\catQ)$ is finite.
Therefore, by Proposition \ref{prop:Mobius inversion of a contructible function}, the Con-GRI of $M$ is M\"obius invertible over $\pi^{-1}\Con(\catQ)$, i.e. there exists a function $d_M:\pi^{-1}\Con(\catQ)\rightarrow \Z$ such that
\begin{equation}\label{eq:sum over pi inverse conQ}
    \rk_M(I)=\sum_{\substack{L\supset I \\ L\in \pi^{-1}\Con(\catQ)}}d_M(L), \ \ \forall I\in \Con(\catP).
\end{equation}
Now, consider the subcollection of $\Con(\catP)$ given by \[\Ical:=\{J:\mbox{$J$ is a connected component of an element in $\pi^{-1}\Con(\catQ)$}\}.\]
Since $\pi^{-1}\Con(\catQ)$ is finite, for each $I\in \Con(\catP)$, we have that \[\Ical_{\supset I}:=\{J\supset I:\mbox{$J$ is a connected component of an element in $\pi^{-1}\Con(\catQ)$}\}\]is finite.  
Now let us define $\bar{d}_M:\Con(\catP)\rightarrow \Z$ by
\[I\mapsto \begin{cases}
d_M(\pi^{-1}\pi(I)), &I\in \Ical\\0
,& \mbox{otherwise.}
\end{cases}\] 
We claim that 
\begin{equation}\label{eq:sum over conP}
\rk_M(I)=\sum_{\substack{J\supset I \\ J\in \Con(\catP)}}\bar{d}_M(J), \ \ \forall I\in \Con(\catP).
\end{equation}
Fix any $I\in \Con(\catP)$. 
The RHS contains only finite many nonzero summands since $\Ical_{\supset I}$ is finite. For any $L$ containing $I$ with $L\in \pi^{-1}\Con(\catQ)$, there exists a unique connected component $J\in \Con(\catP)$ of $L$ containing $I$. Furthermore, by construction, we have $\bar{d}_M(J)=d_M(L)$. This proves that the sums given in Equations (\ref{eq:sum over pi inverse conQ}) and (\ref{eq:sum over conP}) coincide.

Note that, since $\pi$ is surjective, the two maps
\[\pi: \pi^{-1}\Con(\catQ) \leftrightarrows \Con(\catQ):\pi^{-1}\]
are inverse to each other and order isomorphisms. Therefore, $d_M(\pi^{-1}(I))=\dgm_N(I)$ for all $I\in \Con(\catQ)$. This implies that 
\[\bar{d}_M(I)=\begin{cases}\dgm_N(\pi(I)),&I\in \Ical\\0,&\mbox{otherwise,} \end{cases}\]
completing the proof.
\end{proof}

\begin{example}\label{ex:anti-diagonal}
Consider the $\Z^2$-module $M$ defined as follows:
\[M_{(x,y)}=\begin{cases} 0,&y<-x \\ k^2,&y=-x \\ k,&y>-x,
\end{cases}
\]
where $M$ is constant on $y>-x$ and each map $k^2\rightarrow k$ in $M$ is the projection $p_1$ to the first factor. Although $M$ is not finitely presentable, $M$ is tame: Let $N$ be the $\{1<2<3\}$-module given by
$0\rightarrow k^2\stackrel{p_1}{\rightarrow} k.$
There exists the obvious order-preserving map $\pi:\Z^2\rightarrow \{1,2,3\}$ such that $M=\pi^\ast N$. This $\pi$ is a covering and Int-rank-preserving. Thus, by the previous theorem, not only $\dgm_M$ exists, but also $\dgm_M$ is directly obtained from $\dgm_N$. 

Another way to show that the GRI of $M$ is M\"obius invertible is to prove that $M$ is interval-decomposable, and invoke Theorem \ref{thm:sufficient conditions for invertibility} \ref{item:interval-decomposable}. 
\end{example}

\ok{By exploiting Theorem \ref{thm:pull-back GPD},} we refine Theorem \ref{thm:sufficient conditions for invertibility} \ref{item:finitely presentable invertible} by specifying the precise types of intervals that appear in the support of the Int-GPD of a finitely presentable $\R^d$-module. 

\begin{proposition}\label{prop:finitely presentable invertiver over intmn}
The Int-GRI of any finitely presentable $\R^d$-module is M\"obius invertible over \ok{a finite subset of } $\Int_{m,n}(\R^d)$ (cf. Equation (\ref{eq:intmn})) for large enough $m,n\in \N$.
\end{proposition}
 \ok{In plain words, the Int-GRI of any finitely presentable $\R^d$-module admits a compact encoding as a  'persistence diagram' consisting of elements from $\Int_{m,n}(\R^d)$. This result sheds light on computational aspects of the Int-GRI and also on its compact
encoding.} 
\begin{proof}
\ok{Let $\pi:\R^d\rightarrow \catQ$ be the order-preserving map from the proof of Theorem \ref{thm:sufficient conditions for invertibility} \ref{item:finitely presentable invertible}, where $\catQ=\catQ'\cup\{-\infty\}$. The map $\pi$ is a covering that is both interval-preserving and $\Int$-rank-preserving. Also, observe that there exist $m,n\in \N$ (which depend on the size of the grid $\catQ'$) such that for all $I\in \Int(\R^d)$, $\pi^{-1}\pi(I)$ belongs to $\Int_{m,n}(\R^d)$. Now, the statement  immediately follows from Theorem \ref{thm:pull-back GPD}.}
\end{proof}

\section{Discriminating power of the generalized rank invariant}\label{sec:main} 

In this section, we investigate the discriminating power of the GRI from various perspectives: Section \ref{sec:known results} recalls known results, and Section \ref{sec:new discriminating power} presents new results.
In Section \ref{sec:gri vs. zz}, we restrict our attention to the setting of 2-parameter persistence modules. In Section \ref{sec:comparison}, we compare the GRI on various domains with other invariants of persistence modules.

\subsection{Known results}\label{sec:known results}

In this section, we review known results regarding the discriminating power of the GRI.
\begin{proposition}
\label{prop:gricontaining}
Let $M$ be any interval-decomposable $\catP$-module. Then, for each $I\in \Con(\catP)$,
$\rank_M(I)$ is equal to the total multiplicity of intervals $J\in \barc(M)$ such that $J\supset I$.
\end{proposition}

\begin{proof}
    When $\catP$ is finite, Remark \ref{rem:properties of rank invariant} \ref{item:additivity} and \ref{item:the GRI of an interval module} directly imply the claim. 
    The claim still holds even without these assumptions, see \cite[Proposition 3.17]{kim2021generalized} and \cite[Proposition 2.1]{botnan2021signed}.
\end{proof}

\begin{oldtheorem}
\label{thm:gricomplete}
Assume that the GRI of a $\catP$-module $M$ is convolvable over $\Ical\subset\Int(\catP)$. If $M$ is $\Ical$-decomposable, then for all $I\in \Ical$,
$\dgm_M^\Ical(I)$ is equal to the multiplicity of $I$ in $\barc(M)$.
\end{oldtheorem}

The proof of Theorem \ref{thm:gricomplete} below is along the same lines as that of \cite[Theorem 9]{dey2022computing}.

\begin{proof}
For $J\in \Ical$, let $m_J$ be the multiplicity of $J$ in $\barc(M)$. 
We have $\displaystyle M\cong \bigoplus_{J\in \Ical}(\kf_J)^{m_J}$ where $(\kf_J)^{m_J}$ denotes the direct sum of $m_J$ copies of $\kf_J$. Then, by Proposition \ref{prop:gricontaining} 
, we have that
\[\rank_M(I)
=\sum_{\substack{J\supset I\\ J\in \Ical}} m_J \ \ \ \forall I\in \Ical.\]
By the uniqueness of $\dgm_M^{\Ical}$ 
(Theorem \ref{thm:mobius} and Definition \ref{def:GPD over I}), 
we have that $\dgm_M^\Ical(I)=m_I$ for all $I\in \Ical$.
\end{proof}

\begin{corollary}\label{cor:gricomplete}Let $\Ical\subset \Int(\catP)$ be such that every principal ideal of $\Ical$ is finite (and thus $\Ical$ is locally finite). Then, the GRI over $\Ical$ is a complete invariant of $\Ical$-decomposable $\catP$-modules. 
\end{corollary}

We emphasize that this corollary is a direct consequence of the M\"obius inversion formula. 
Depending on \ok{the poset structures of} 
$\Int(\catP)$ and $\Ical\subset \Int(\catP)$, the GRI over $\Ical$ can actually be a complete invariant over an even larger collection than the collection of all $\Ical$-decomposable modules.
 Let $\widehat{\Ical}$ be the \textbf{limit completion} of $\Ical$, i.e. the collection of all unions of nested families of intervals in $\Ical$. In other words,
\begin{equation}\label{eq:limit intervals}\widehat{\Ical}:=\left\{\bigcup_{x\in X} I_x: \mbox{$X$ is totally ordered, $I_x\in \Ical$ and $I_x\subset I_y$ for all $x\leq y$ in $X$}\right\},
\end{equation}
which is a subcollection of $\Int(\catP)$ containing $\Ical$.
For example, if $\catP=\R$ and $\Ical:=\{[0,a]\in \Int(\R):a\in [0,1]\}$, then $\widehat{\Ical}=\Ical\cup\{[0,b):b\in (0,1]\}$.
\begin{proposition}[Restatement of {\cite[Proposition 2.10]{botnan2021signed}}]\label{prop:known discriminating power of rk} Let $\Ical\subset \Int(\catP)$.  Then, the GRI over $\Ical$ is a complete invariant on the collection of {$\catP$-modules} $M$ that are $\widehat{\Ical}$-decomposable.\footnote{The original statement includes the assumption $\rank_M(I)<\infty$ for all $I\in \Ical$. This is automatically guaranteed by our assumption that $M$ is pointwisely finite-dimensional.}
\end{proposition}

\subsection{Optimality of Completeness, and Extent of Incompleteness}\label{sec:new discriminating power}

\ok{This section presents three key results concerning the discriminating power of the GRI:
\begin{itemize}
    \item Theorem \ref{thm:mainthm}, which establishes the completeness of the GRI for a specific class of persistence modules,
    \item Theorem \ref{thm:tightness}, which demonstrates the optimality of the previous completeness result
    in a suitable sense, and
    \item  Theorem \ref{thm:coincidence of rank invariants} and Corollary \ref{cor:minimal pairs}, which quantify the failure of the GRI to be complete on a collection of persistence modules.
\end{itemize}  
We emphasize that the M\"obius inversion formula plays a central role in the proofs of all of these theorems, \emph{even without}  the assumption of local finiteness on the domain of the GRI.}

First, we generalize Theorem \ref{thm:gricomplete} and Corollary \ref{cor:gricomplete} by dispensing with redundant assumptions on $\Ical\in\Int(\catP)$.

\begin{theorem}[Completeness]\label{thm:mainthm} \contributionn 
Let $\catP$ be any poset. Then, 
\begin{enumerate}[label=(\roman*),leftmargin=*]
 \item The GRI over any $\Ical\subset \Int(\catP)$ is a complete invariant on the collection of all $\Ical$-decomposable $\catP$-modules.
 \label{item:mainthm1}
 \item The direct sum decomposition of any interval-decomposable $\catP$-module $M$ can be obtained via M\"obius inversion of $\rk_M$ over the subposet  
  $$\Ical_M:=\{I\in \Int(\catP): I\in \barc(M)\}\subset\Int(\catP).$$ 
 \label{item:mainthm2}
 \end{enumerate}
\end{theorem}
The statement given in item \ref{item:mainthm1} is weaker than Proposition \ref{prop:known discriminating power of rk}. 
However, the proof of item \ref{item:mainthm1} given below is not only simpler than that of Proposition \ref{prop:known discriminating power of rk}, 
but also 
simultaneously proves item \ref{item:mainthm2}, that is not implied by Proposition \ref{prop:known discriminating power of rk}.

\begin{proof}[Proof of Theorem {\ref{thm:mainthm}}]
Consider the function $\mult_M^{\Ical_M}:\Ical_M\rightarrow \Zplus$ sending each $I\in \Ical$ to the multiplicity of $I$ in $\barc(M)$. 
  Proposition \ref{prop:gricontaining} implies
\[\rank_M^{\Ical}(I)=\sum_{\substack{J\supset I\\ J\in \Ical_M}}\mult_M^{\Ical_M}(J) \ \ \mbox{ for all $I\in \Ical$.}\]
By Definition \ref{def:mobius invertible}, $\rank_M^\Ical$ is M\"obius invertible over $\Ical_M$. By Proposition \ref{prop:justifying the term mobius invertibility}, the function $\mult_M^{\Ical_M}$ equals $\dgm_M^{\Ical_M}$, the M\"obius inversion of $\rk_M^{\Ical_M}$. Since $\rk_M^{\Ical_M}$ is the restriction of $\rk_M^{\Ical}$ to $\Ical_M$, we have proved that $\rank_M^{\Ical}$ uniquely determines $\mult_M^{\Ical_M}$, Item \ref{item:mainthm1}, as well as Item \ref{item:mainthm2}.  
\end{proof}

We now generalize the well-known result that the RI is a complete invariant on the collection of \emph{rectangle-decomposable} $\Z^2$-modules \cite[Theorem 2.1]{botnan2022rectangle} to the case of $\Z^d$- and $\R^d$-modules for $d\geq 2$:\footnote{{When $\catP=\R^d$ or $\Z^d$ and $\mathcal{I}=\seg(\catP)$, $\Ical$-decomposable $\catP$-modules are often called \emph{rectangle-decomposable}.}}  

\begin{corollary}\label{cor:rectangle decomposition}\contributionn
The RI is a complete invariant for a rectangle-decomposable $\R^d$- or $\Z^d$-module $M$ for any dimension  $d$.  
Furthermore, the multiplicity of each segment $I$ in the barcode of $M$ is the M\"obius inversion of $\rk_M$ (over some $\Ical\subset \Int(\catP)$ with $I\in\Ical$) evaluated at $I$.
\end{corollary}

\begin{proof}
\ok{Theorem \ref{thm:mainthm} \ref{item:mainthm1} (or Proposition \ref{prop:known discriminating power of rk}) implies that the RI is a complete invariant for rectangle-decomposable $\R^d$- or $\Z^d$-modules for any dimension  $d$. Theorem \ref{thm:mainthm} \ref{item:mainthm2} implies that the multiplicity of each $I\in \barc(M)$ can be obtained via M\"obius inversion of $\rk_M$ over \[\{I\in \Int(\catP): I\in \barc(M)\}.\qedhere\]
}
\end{proof}

\begin{remark}\label{rem:relationship with known result} We clarify the relationship between Proposition \ref{prop:known discriminating power of rk} and Theorem \ref{thm:mainthm}. 
\begin{enumerate}[label=(\roman*),leftmargin=*]
    \item If every principal ideal of $\Ical$ is finite, then the limit completion $\widehat{\Ical}$ of $\Ical$  (Equation (\ref{eq:limit intervals})) is the same as $\Ical$, and thus Proposition \ref{prop:known discriminating power of rk} reduces to Item \ref{item:mainthm1} of Theorem \ref{thm:mainthm}.
    \item If $\Ical\subsetneq \widehat{\Ical}$,
by Remark \ref{rem:pointwisely finite multiset}, $\mult_{M}^{\widehat{\Ical}}$ has the support 
$\Ical'\subsetneq \widehat{\Ical}$ for which every principal ideal is finite (and thus locally finite). Then, the GRI over $\Ical'$ is a complete invariant of $M$ (note: $\Ical'$ is not necessarily contained in $\Ical$), \ok{and the barcode of $M$ can be obtained from M\"obius inversion of the GRI over $\Ical'$.} 
\end{enumerate}
\end{remark}
The statement of 
Corollary \ref{cor:gricomplete} is optimal in the following sense.
\begin{theorem}
[Optimality of Corollary  \ref{cor:gricomplete}]
\label{thm:tightness} 
\contributionn Let $\mathcal{I}\subset\Int(\catP)$ where every principal ideal is finite.
Let $\mathfrak{L}$ be any collection of indecomposable $\catP$-modules properly containing $\{\kf_I:I\in \Ical\}$.\footnote{The collection $\mathfrak{L}$ may include non-interval modules.} 
Then, the GRI over $\Ical$ is not a complete invariant on the collection of $\mathfrak{L}$-decomposable modules. 
\end{theorem}

\begin{proof}
It suffices to find a non-isomorphic pair $N,N'$ of $\Lfrak$-decomposable $P$-modules that have the same GRI over $\Ical$. 
Let $M\in \Lfrak$ such that it is not isomorphic to $\kf_I$ for any $I\in \Ical$. 
Consider $\dgm_M^{\Ical}$, the GPD of $M$ over $\Ical$, which exists by the assumption that every principal ideal of $\Ical$ is finite. 
Now consider the two $\catP$-modules
\[N:=\bigoplus_{\substack{ I\in \Ical \\\dgm_M^{\Ical}(I)>0}} (\kf_I)^{\dgm^\Ical_M(I)}\ \ \ \ \mbox{and} \ \ \ \ N':=M\oplus \left(\bigoplus_{\substack{ I\in \Ical \\\dgm_M^{\Ical}(I)<0}} (\kf_I)^{-\dgm_M^\Ical(I)}\right), \]
where $(\kf_I)^{n}$ stands for the direct sum of $n$ copies of $\kf_I$.
While $N$ is $\Ical$-decomposable, $N'$ is not $\Ical$-decomposable as it has $M$ as a summand, and thus $N\not\cong N'$. 
By additivity of $\dgm_M^\Ical$ (cf. Remark \ref{rem:additivity of GPD}), the GPDs of $N$ and $N'$, $\dgm_N^\Ical, \dgm_{N'}^\Ical:\Ical\to \Z$ are both equal as functions to the map $(\dgm_M^\Ical)_+:\Ical\to\Z$ given by $I\mapsto \max(\dgm_M^\Ical(I),0)$. 
Therefore, $\rank_N^\Ical$ coincides with $\rank_{N'}^\Ical$, both of which are equal to $(\dgm_M^\Ical)_+\ast \zeta_\Ical$.
\end{proof}

\begin{corollary}\label{cor:Ical and Jcal}
    \contributionn  
    Let $\mathcal{I}\subset \Int(\catP)$ be such that every principal ideal is finite. Then, for any $\Jcal\subset \Int(\catP)$ with $\Ical\subsetneq \Jcal$, the GRI over $\Ical$ is \emph{not} a complete invariant on the collection of $\Jcal$-decomposable \pmo{}s.
\end{corollary}

Here is an application of Corollary \ref{cor:Ical and Jcal}: We claim that as $m,n\in \N$ increase, the discriminating power of the GRI over $\Int_{m,n}(\R^d)$ (cf. Equation (\ref{eq:intmn}))  on $\R^d$-modules  \emph{strictly} increases:
\begin{corollary}\label{cor:intmn discriminating power}
If $(m,n)<(m',n')$ in $\N^2$, then there is a pair of $\R^d$-modules that are distinguished by their GRIs over $\Int_{m',n'}(\R^d)$, but are confounded  by their GRIs over $\Int_{m,n}(\R^d)$. 
\end{corollary}
For any natural number $n$, let $[n]:=\{1,\ldots,n\}$ be equipped with the canonical order.\label{[n]}
\begin{proof}Let $\ell\in \N$ with $\ell>m',n'$. By Corollary \ref{cor:Ical and Jcal}, there exists a pair of $[\ell]^d$-modules $M$ and $N$ that are distinguished by their GRIs over $\Int_{m',n'}^\mathrm{cc}([\ell]^d)$ but not by $\Int_{m,n}^\mathrm{cc}([\ell]^d)$. 
Let $\iota:[\ell]^d\hookrightarrow \R^d$ be the canonical inclusion and let $\lfloor - \rfloor_{\iota}:\R^d\rightarrow [\ell]^d\cup\{-\infty\}$ 
send each $p\in \R^d$ to the maximal $q\in [\ell]^d\cup\{-\infty\}$ such that $q\leq p$ (where we declare that $-\infty<p$ for all $p\in \R^d$).
Let $M'$ and $N'$ respectively be the $\R^d$-modules given as the left Kan extensions of $M$ and $N$ along $\iota$, i.e. for all $p\leq q$ in $\R^d$,
\[M'_p=M_{\lfloor p \rfloor_\iota } \ \ \mbox{and}\ \  \varphi_{M'}(p, q)=\varphi_{M}({\lfloor p \rfloor_\iota, \lfloor q \rfloor_\iota }),\]
where $M_{-\infty}$ is defined to be $0$. Define $N'$ similarly. Observe that the GRIs of $M'$ and $N'$ over $\Int_{m,n}(\R^d)$ coincide as, for every $I\in \Int_{m,n}(\R^d)$, we have $\lfloor I\rfloor_\iota\in \Int_{m,n}^\mathrm{cc}([\ell]^d\cup \{-\infty\})$ so that  \[\rk_{M'}(I)=\rk_M(\lfloor I\rfloor_\iota)=\rk_N(\lfloor I\rfloor_\iota)=\rk_{N'}(I).\] On the other hand, we claim that the GRIs of $M'$ and $N'$ over $\Int_{m',n'}(\R^d)$ do not coincide. To see this, pick any $K\in  \Int_{m,n}^\mathrm{cc}([\ell]^d)$ such that $\rk_M(K)\neq \rk_N(K)$. Then we pick any element $J\in \Int_{m',n'}(\R^d)$ such that $\lfloor J \rfloor_\iota=K$, which implies
$\rk_{M'}(J)\neq \rk_{N'}(J).$
\end{proof}

Next, we investigate the extent to which the GRI over $\Ical$ fails to be complete on the collection of $\Jcal$-decomposable modules. 
We will see that the failure of completeness of the GRI over $\Ical$ can be quantified via the dimension of the kernel of a linear map, which equals the cardinality of $\Jcal\setminus\Ical$.

For $I\in  \Int(\catP)$, let $\mathbf{1}_I:\Int(\catP)\rightarrow \{0,1\}$ be the indicator supported on $I$. Then, by Remark \ref{rem:convolvability} \ref{item:convolvability1}, $\mathbf{1}_I$ is convolvable over \emph{any} locally finite $\Tcal\subset \Int(\catP)$. For simplicity, we denote the M\"obius inversion of the restriction $\mathbf{1}_I|_{\Tcal}$ over $\Tcal$ as $\mathbf{1}_I\ast\mu_{\Tcal}$.

\begin{theorem}\label{thm:coincidence of rank invariants} \contributionn\  
     Let $\Ical\subset \Jcal\subset \Int(\catP)$. Assume that the GRIs of $\catP$-modules $M$ and $N$ are convolvable over $\Jcal$. 
     Then, $\rank_M^\Jcal$ and $\rank_N^\Jcal$ coincide on $\Ical$ if and only if $\dgm_M^\Jcal-\dgm_N^\Jcal$ is a linear combination of the M\"obius inverses over $\Jcal$ of the indicators $\mathbf{1}_I$ for $I\in \Jcal\setminus\Ical$. 
\end{theorem}

We defer the proof to the end of this section.  

Any pair of non-isomorphic interval-decomposable $\catP$-modules $M$ and $N$ that have the same GRI over $\Ical\subset \Int(\catP)$ is called \textbf{$\Ical$-minimal} if there are no proper nonzero summands $M',N'$ of $M,N$ respectively such that $M'$ and $N'$ have the same GRI over $\Ical$.

\begin{corollary}\label{cor:minimal pairs} \contributionn\  Let $\Ical\subset \Jcal\subset \Int(\catP)$ and assume that every principal ideal of $\Jcal$ is finite.
Then, there exist $\abs{\Jcal\setminus\Ical}$ {distinct $\Ical$-}minimal non-isomorphic pairs of $\Jcal$-decomposable $\catP$-modules whose GRIs coincide on $\Ical$.
\end{corollary}

We remark that when $\Ical=\Jcal$, the corollary above reduces to Theorem \ref{thm:mainthm}.

\begin{proof} Let $I\in \Jcal\setminus\Ical$ and consider the M\"obius inversion $\mathbf{1}_{I}\ast\mu_{\Jcal}$ of the indicator $\mathbf{1}_{I}:\Jcal\rightarrow \{0,1\}$. 
For $J\in \Jcal$, let $d_J:=(\mathbf{1}_{I}\ast\mu_{\Jcal})(J)$.
Clearly,  the two $\Jcal$-decomposable $\catP$-modules 
\begin{equation}\label{eq:minimal non-iso pair}M^+_I:=\bigoplus_{\substack{J\in \Jcal\\ d_J>0}} (\kf_J)^{d_J} \ \ \mbox{and}\ \ \  M^-_I:=\bigoplus_{\substack{J\in \Jcal\\ d_J<0}} (\kf_J)^{-d_J} \end{equation}
form such a minimal non-isomorphic pair. 
Since $\ast\mu_\Jcal$ is an automorphism (cf. Remark \ref{rem:right multiplication} \ref{item:zeta defines an automorphism}), the pair $(M^+_I,M^-_I)$ is different from the pair $(M^+_{I'},M^-_{I'})$ for any $I'\in \Jcal\setminus \Ical$ with $I'\neq I$.
\end{proof}

Theorem \ref{thm:coincidence of rank invariants} and Corollary \ref{cor:minimal pairs} offer a theoretical foundation for well-known examples of non-isomorphic interval-decomposable multi-parameter persistence modules that share the same rank invariant: 

\begin{example}\label{ex:non-iso example 1} 
Let $\catP:=[4]$ and consider the subposets $\Ical:=\{[2,3],[1,3],[2,4]\}$ and $\Jcal:=\Ical\cup\{[1,4]\}$ of $\Int(\catP)$. 
Using the recursion in Equation (\ref{eq:mobius}), it is not hard to verify that $\mathbf{1}_{[1,4]}\ast\mu_{\Jcal}=\mathbf{1}_{[1,4]}-\mathbf{1}_{[1,3]}-\mathbf{1}_{[2,4]}+\mathbf{1}_{[2,3]}$. 
Since $|\Jcal \setminus \Ical| = 1$, Corollary \ref{cor:minimal pairs} implies that there is a single minimal non-isomorphic pair of $\Jcal$-decomposable $\catP$-modules whose GRIs coincide on $\Ical$.
As $\Jcal\setminus \Ical = \{[1,4]\}$, we let $I=[1,4]$ and compute the $\Ical$-minimal pair $M^+_I$ and $M^-_I$ by Equation (\ref{eq:minimal non-iso pair}).
This yields $M^+_I = \kf_{[1,4]}\oplus \kf_{[2,3]}$, and $M^-_I=\kf_{[1,3]}\oplus \kf_{[2,4]}$.
\end{example}

\begin{example}\label{ex:mobius inversion of indicator}
     For the intervals $I,J_1,J_2,J_3$ in the poset $[2]^2$ depicted below,
     let $M,N:[2]^2\rightarrow \cvec$ be given by
    \[M\cong \kf_{I}\oplus \kf_{J_3} \ \ \mbox{and} \ \ N\cong \kf_{J_1}\oplus \kf_{J_2}.\] 
    \begin{center}
    \includegraphics[width=.5\linewidth]{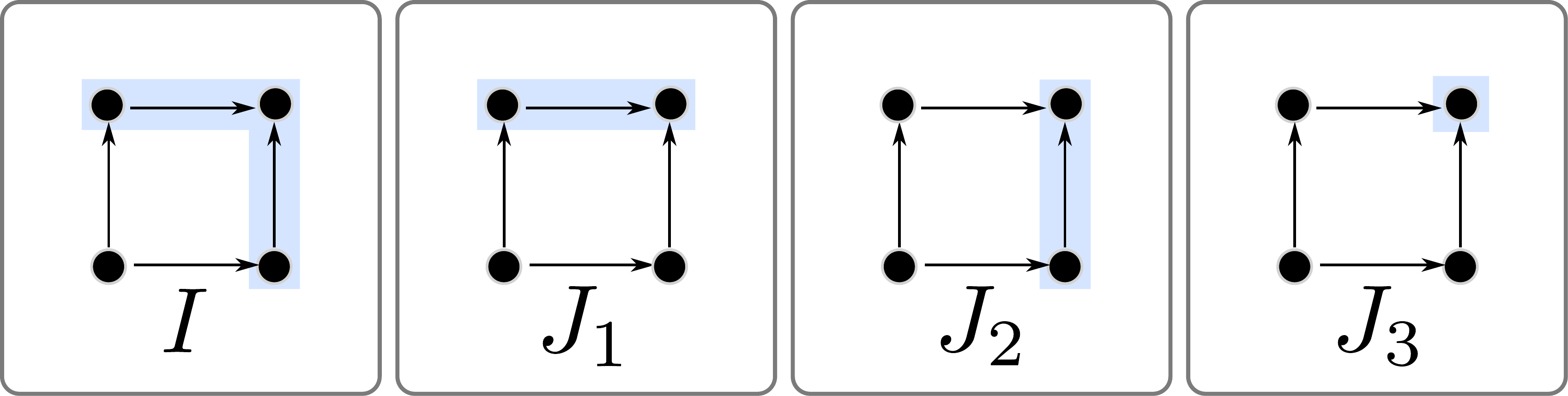} 
 \end{center}
 Note that $M$ and $N$ have the same GRI over $\Ical:=\{J_1,J_2,J_3\}$, but not over $\Jcal:=\Ical\cup \{I\}$. By Theorem \ref{thm:gricomplete}, we have $\dgm_M^{\Jcal}=\mathbf{1}_I+\mathbf{1}_{J_3}$ and $\dgm_N^{\Jcal}=\mathbf{1}_{J_1}+\mathbf{1}_{J_2}$.
 Using the recursion in Equation (\ref{eq:mobius}), one can verify that $\mathbf{1}_I\ast\mu_\Jcal=\mathbf{1}_I-\mathbf{1}_{J_1}-\mathbf{1}_{J_2}+\mathbf{1}_{J_3}=\dgm_M^{\Jcal}-\dgm_N^{\Jcal}$.

Since $|\Jcal\setminus \Ical|=1$, Corollary \ref{cor:minimal pairs} implies that $M$ and $N$ form the unique minimal pair of $\Jcal$-decomposable non-isomorphic $[2]^2$-modules.
\end{example}

Before proving Theorem \ref{thm:coincidence of rank invariants}, we set up notation. Let $\Q$ denote the set of rational numbers. 
 For any $\Jcal\subset \Int(\catP)$, let $\Q^\Jcal_{c}$ be the vector space of convolvable maps $\Jcal\rightarrow \Q$ (cf. Remark \ref{rem:convolvability} \ref{item:convolvability3}). Assume that $\Ical\subset \Jcal\subset \Int(\catP)$.  For the zeta function $\zeta_\Jcal$ of the poset $(\Jcal,\supset)$, the map $\ast\mob_{\Jcal}$ is an automorphism on $\Q_c^\Jcal$ with inverse $\ast\mu_{\Jcal}$ (cf. Remark \ref{rem:right multiplication} \ref{item:zeta defines an automorphism}). 
 Let $\pi_{\Jcal}^{\Ical}:\Q_c^\Jcal\rightarrow \Q_c^\Ical$ be  the restriction $f\mapsto f|_{\Ical}$. We have the following diagram: 
\begin{equation}\label{eq:maps}
    \begin{tikzcd}
\Q_c^\Jcal \arrow[rr, "\ast\zeta_{\Jcal}", shift left=1ex] && \Q_c^\Jcal \arrow[rr, "\pi_{\Jcal}^{\Ical}" ] \arrow[ll, "\ast\mu_{\Jcal}",  shift left=1ex] && \Q_c^\Ical 
\end{tikzcd}
\end{equation}

Given any $\catP$-module $M$ whose GRI over $\Jcal$ is convolvable, the functions $\dgm_M^\Jcal,\rank_M^\Jcal$ and $\rank_M^{\Ical}$ are mapped to each other via the maps given in Equation (\ref{eq:maps}):
\begin{equation}\label{eq:mapsto}
    \begin{tikzcd}
\dgm_M^{\Jcal} \arrow[rr, "\ast\zeta_{\Jcal}", shift left=1ex, mapsto] && \rank_M^{\Jcal} \arrow[rr, mapsto] \arrow[ll, "\ast\mu_{\Jcal}",  shift left=1ex, mapsto] && \rank_M^{\Ical}
\end{tikzcd}
\end{equation}

\begin{remark}\label{rem:kernel}The kernel of $\pi_{\Jcal}^\Ical$ consists of those functions $f:\Jcal\rightarrow \Q$ such that $f(I)=0$ for all $I\in I$. Equivalently, 
$\ker(\pi_\Jcal^\Ical)=\lin\{\mathbf{1}_I:I\in \Jcal\setminus\Ical\}.$    
\end{remark}

\begin{proof}[Proof of Theorem \ref{thm:coincidence of rank invariants}] From the diagram given in Equation (\ref{eq:maps}), we have:
    \begin{align*}        &\rank_M^\Ical=\rank_N^\Ical\\\Leftrightarrow\ &\rank_M^\Ical-\rank_N^{\Ical}=0\\\Leftrightarrow\ &\pi_{\Jcal}^\Ical(\rank_M^{\Jcal}-\rank_M^{\Jcal})=0\\\Leftrightarrow\ & \rank_M^{\Jcal}-\rank_M^{\Jcal}\in \lin\{\mathbf{1}_I:I\in\Jcal\setminus\Ical\}  &&\mbox{by Remark \ref{rem:kernel}}
    \\\Leftrightarrow\ & (\rank_M^\Jcal-\rank_N^{\Jcal})\ast \mu_{\Jcal}\in \lin\{\mathbf{1}_I\ast \mu_{\Jcal}:I\in\Jcal\setminus\Ical\}&&\mbox{by Remark \ref{rem:right multiplication} \ref{item:zeta defines an automorphism}}  \\
        \Leftrightarrow\ &\dgm_M^\Jcal-\dgm_N^{\Jcal}\in \lin\{\mathbf{1}_I\ast\mu_\Jcal:I\in \Jcal\setminus\Ical\}&&\mbox{by Definition \ref{def:GPD over I}.}\qedhere
    \end{align*}
\end{proof}

\begin{remark}[Another proof of Corollary \ref{cor:Ical and Jcal} for the case when $\Ical$, $\Jcal$ are  finite] When $\Ical\subset\Jcal\subset \Int(\catP)$ are finite, we can prove Corollary \ref{cor:Ical and Jcal} using the rank-nullity theorem. 
Note that, in Equation (\ref{eq:maps}), (i) $\Q_c^\Ical=\Q^\Ical$ and $\Q_c^\Jcal$=$\Q^\Jcal$, (ii) $\{\mathbf{1}_I:I\in \Ical\}$ and $\{\mathbf{1}_I:I\in \Jcal\}$ are bases for $\Q^\Ical$ and $\Q^\Jcal$ respectively, and (iii) Since $(\ast\zeta_{\Jcal})$ is an automorphism with the inverse $(\ast\mu_{\Jcal})$, the kernel of the composition $\pi_\Jcal^\Ical\circ (\ast\zeta_{\Jcal})$ coincides with the image of $\ker(\pi_{\Jcal}^\Ical)$ via $(\ast\mu_{\Jcal})$, which is $\abs{\Jcal\setminus\Ical}$-dimensional space by the rank-nullity theorem. 
We omit further details.
\end{remark}

\subsection{Results for 2-parameter persistence modules}\label{sec:gri vs. zz}
  In this section, we focus on clarifying the discriminating power of the GRI of $\Z^2$-modules. Specifically, we do this by comparing the GRI with the ZIB (Section \ref{sec:5.1}), and with the bigraded Betti numbers (Section \ref{sec:bigraded Bettis}).

\subsubsection{Comparison with ZIB over simple paths}\label{sec:ZIB over Z^2}

Let $\con(\catP)$ and $\fint(\catP)$ denote the sets of  \emph{finite} connected subsets and \emph{finite}  intervals of a given poset $\catP$, respectively. We use the term $\fint$-GRI to refer to the GRI over $\fint(\catP)$.\label{ex:mn subposet}
We already know that the ZIB of a $\Z^2$-module $M$ \textbf{determines} the int-GRI of $M$, i.e. if two $\Z^2$-modules have the same ZIB, then they have the same int-GRI; this is implied by \cite[Theorem 24]{dey2022computing} stating that for $I\in \fint(\Z^2)$, $\rk_M(I)$ equals the rank of a zigzag module arsing as the restriction of $M$ to a certain path $\partial I$ along the boundary of $I$.  A priori, the path $\partial I$  can contain repeated points in $\Z^2$. Hence, a natural question for efficient computation of the int-GRI is whether the ZIB over only \emph{simple paths} (i.e. paths without repeated points)  determines the int-GRI.

In this section, we show that the answer is negative. 
 Furthermore, we show that the int-GRI also does not determine the ZIB over simple paths. Hence, the discriminating power of the int-GRI and of the ZIB over simple paths are \emph{not} well-ordered. It then follows that the ZIB over \emph{all} paths is a \emph{strictly} finer invariant than both the ZIB over \emph{simple} paths and the $\fint$-GRI: see {Examples \ref{ex:Int does not recover simple ZZ}, \ref{ex:intvszz}} and {Figure \ref{fig:intro}}.   We take one step further and investigate how much the int-GRI can be recovered from the ZIB over simple paths, and the other way around: see {Remark \ref{rem:approximating Int rank invariant via simple zz}} and {Proposition \ref{prop:approximating zz}}. 

\paragraph{Comparison between the ZIB and the GRI.}
\label{sec:5.1}
 For $a,b\in \Z^2$, we write $a\triangleleft b$ if and only if $a<b$ and $[a,b]=\{a,b\}$. Let $\Gamma:p_1,p_2,\ldots,p_n$ be a path in $\Z^2$. 
We call $\Gamma$ \textbf{faithful} if  $p_i\triangleleft p_{i+1}$ or $p_{i+1} \triangleleft p_{i}$ for each $i=1,\ldots,  n-1$. In what follows, we consider only faithful paths; this is because for any $\Z^2$-module $M$ and for any non-faithful path $\Gamma$, there is a faithful path $\Gamma'$ containing $\Gamma$ as a subsequence, and thus $\barc(M_{\Gamma})$ can be read off from $\barc(M_{\Gamma'})$.

We call $\Gamma$ \textbf{simple} if  all of the points $p_1,p_2,\ldots,p_n$ are distinct from each other. 
There are two special types of simple paths: We call $\Gamma$ a \textbf{monotone path} or \textbf{positive path}, if $p_i\triangleleft p_{i+1}$ for each $i$.  We call $\Gamma$ a \textbf{negative path} if $\Gamma$ is obtained from the reflection of a monotone path with respect to the $y$-axis. 

\begin{definition}\label{def:zigzag-path-indexed barcode}
The restriction of the domain of the ZIB of any $\catP$-module to the set of all \emph{simple paths} of $\catP$ called the \textbf{ZIB over simple paths} and is denoted by $M_{\SZZ}$.
\end{definition}

Let $F$ and $G$ be two invariants of $\Z^2$-modules.
If $F$  determines $G$, then we write $F\Rightarrow G$, which defines a transitive relation on the class of invariants of $\Z^2$-modules.  For example, ZIB $\Rightarrow$ ZIB over simple paths.
If $F$ determines $G$ and vice versa, then we write $F \Leftrightarrow G$. 
We say that $F$ \textbf{strictly determines} $G$ if $F\Rightarrow G$ and $G\not\Rightarrow F$. Now,  we compare the ZIB and the GRI of $\Z^2$-modules over various domains:

\begin{figure}
    \centering
\includegraphics[width=\textwidth]{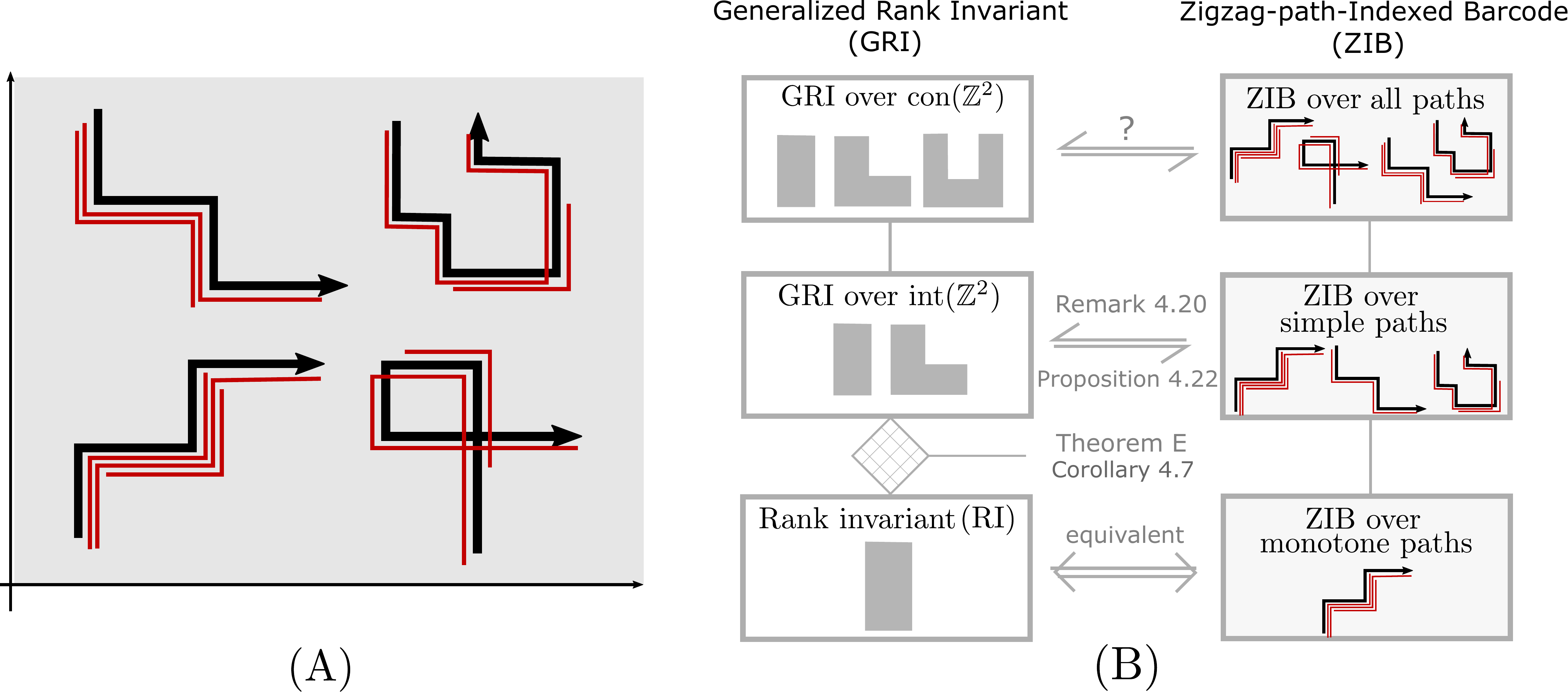}
    \caption{(A) Barcodes of a $\Z^2$-module over zigzag paths. 
    (B) Hierarchy of invariants for $\Z^2$-modules. 
    The hierarchy of the GRI (left column) is comparable to the hierarchy of barcodes over zigzag paths (right column). 
    In this figure, an invariant $F$ strictly determines another invariant $G$ at a lower height, regardless of which columns $F$ and $G$ belong to.
    The notations $\leftharpoonup$ and $\rightharpoondown$ indicate that one invariant does not determine another, but rather can be used to estimate the other. 
    See Remark \ref{rem:implications of gen by zz} for a full explanation.}
    \label{fig:intro}
\end{figure}

\begin{remark}
\phantomsection\label{rem:implications of gen by zz}
In Figure \ref{fig:intro} (B):
\begin{enumerate}[label=(\roman*),topsep=0pt,itemsep=-1ex,partopsep=1ex,parsep=1ex,leftmargin=*]
    \item If $F$ and $G$ are in the same column and $F$ is at a higher level than $G$, then $F\Rightarrow G$ is clear by definition.
     \item  {(GRI over $\con(\Z^2)$ $\Rightarrow$ ZIB over simple paths) follows from the definition of GRI over $\con(\catP)$.} \label{item:implications of gen by zz5} 
    \item (RI $\Leftrightarrow$ ZIB over monotone paths) was noted in \cite[Proposition 1.2]{lesnick2015interactive}.

    \item (ZIB over all paths $\Rightarrow$ GRI {over $\fint(\Z^2)$}) is a direct corollary of \cite[Theorem 24]{dey2022computing}.
    
     \item Theorem \ref{thm:mainthm} and Corollary \ref{cor:Ical and Jcal} demonstrate an interpolation between the $\fint$-GRI and RI of $\Z^2$-modules.
     Namely, these results show that if $\seg(\Z^2)\subset \Ical\subsetneq \Jcal\subset \Int(\Z^2)$, and every principal ideal of $\Ical$ is finite, then $\rk^\Jcal\Rightarrow \rk^\Ical$ but $\rk^\Ical\not\Rightarrow \rk^\Jcal$.\footnote{Theorem \ref{thm:mainthm} and Corollary \ref{cor:Ical and Jcal} claim more general statements as they consider $\catP$-modules not just $\Z^2$-modules.} \label{item:implications of gen by zz3}
\item Overall, by transitivity of the relation $\Rightarrow$, ($F\Rightarrow G$) holds if $F$ is at a higher level than $G$ regardless of the columns they belong to.

\end{enumerate}
\end{remark}

In the following two examples, we will see that ($\fint$-GRI$\not\Rightarrow$ZIB over simple paths) and (ZIB over simple paths $\not\Rightarrow$ $\fint$-GRI). 

\begin{example}[$\fint$-GRI$\not\Rightarrow$ZIB over simple paths]\label{ex:Int does not recover simple ZZ} Let $M,N$ be  $\Z^2$-modules defined as below whose supports are contained in $[3]^2\subset\Z^2$. Also, let $\Gamma:(1,2),(1,3),(2,3),(3,3),(3,2),$ $(3,1),(2,1)$ (cf. Figure \ref{fig:scaffold} (C)). It is not difficult to check that $M_{\Gamma}\not\cong N_{\Gamma}$, whereas
 $\rank_M(I)=\rank_N(I)$ for all $I\in \Int([3]^2)$ \cite[Example A.2]{kim2021bettis}. This shows that the GRI over $\fint(\Z^2)$ cannot fully recover the ZIB over simple paths, the ZIB over all paths, nor the GRI over $\con(\Z^2)$.

\[
M:=\begin{tikzcd}
\kf \arrow[r,"1"] & \kf \arrow[r,"1"] & \kf\\
\kf \arrow[u,"1"]\arrow[r,"\begin{bmatrix}0\\ 1\end{bmatrix}"] & \kf^2 \arrow[r,"\begin{bmatrix} 0 \ 1\end{bmatrix}"']\arrow[u,"\begin{bmatrix}0 \ 1\end{bmatrix}"'] & \kf\arrow[u,"1"]\\
0\arrow[r]\arrow[u] & \kf\arrow[u,"\begin{bmatrix}1\\ 1\end{bmatrix}"]\arrow[r,"1"] & \kf\arrow[u,"1"]
\end{tikzcd}
\oplus
\begin{tikzcd}
\kf \arrow[r,"1"] & \kf \arrow[r,"1"] & \kf\\
0 \arrow[u]\arrow[r] & \kf \arrow[r,"1"]\arrow[u,"1"] & \kf\arrow[u,"1"]\\
0\arrow[r]\arrow[u] & 0 \arrow[u]\arrow[r] & \kf\arrow[u,"1"]
\end{tikzcd}
\]
\[
N:=
\begin{tikzcd}
\kf \arrow[r,"1"] & \kf \arrow[r,"1"] & \kf\\
\kf \arrow[u,"1"]\arrow[r,"1"] & \kf \arrow[r,"1"]\arrow[u,"1"] & \kf\arrow[u,"1"]\\
0\arrow[r]\arrow[u] & 0 \arrow[u]\arrow[r] & \kf\arrow[u,"1"]
\end{tikzcd}
\oplus
\begin{tikzcd}
\kf \arrow[r,"1"] & \kf \arrow[r,"1"] & \kf\\
0 \arrow[u]\arrow[r] & \kf \arrow[r,"1"]\arrow[u,"1"] & \kf\arrow[u,"1"]\\
0\arrow[r]\arrow[u] & \kf \arrow[u,"1"]\arrow[r,"1"] & \kf\arrow[u,"1"]
\end{tikzcd}
\oplus 
\begin{tikzcd}
0 \arrow[r] & 0 \arrow[r,] & 0\\
0 \arrow[u]\arrow[r] & \kf \arrow[r,]\arrow[u] & 0\arrow[u,]\\
0\arrow[r]\arrow[u] & 0 \arrow[u]\arrow[r,] & 0\arrow[u,]
\end{tikzcd}
\]
\end{example}

\begin{example}[ZIB over simple paths $\not\Rightarrow$ $\fint$-GRI]\label{ex:intvszz}
Let $I\in \fint(\Z^2)$ with the following directed Hasse diagram. 
\[I:=
\begin{tikzcd}
\bullet\arrow[r] & \bullet\arrow[r] & \bullet \\
& \bullet \arrow[r]\arrow[u] & \bullet\arrow[r]\arrow[u] & \bullet
\end{tikzcd}\]
Let us define the following $\Z^2$-modules $M$ and $N$ supported on $I$:
\[M:= \kf_I \oplus 
\begin{tikzcd}
\kf\arrow[r,"\begin{bmatrix} 1\\0\end{bmatrix}"] & \kf^2\arrow[r,"\begin{bmatrix} 1 \  1\end{bmatrix}"] & \kf\\
& \kf\arrow[u,"\begin{bmatrix} 0\\ 1\end{bmatrix}"]\arrow[r,"\begin{bmatrix} 1\\ 1\end{bmatrix}"'] & \kf^2\arrow[r,"\begin{bmatrix} 0 \  1\end{bmatrix}"']\arrow[u,"\begin{bmatrix} 1 \ 0\end{bmatrix}"'] & \kf
\end{tikzcd}
\]
\[N:= \begin{tikzcd}
\kf\arrow[r,"\begin{bmatrix} 1\\ 0\end{bmatrix}"] & \kf^2\arrow[r,"\begin{bmatrix} 1 \ 1\end{bmatrix}"] & \kf\\
& \kf\arrow[u,"\begin{bmatrix} 0\\ 1\end{bmatrix}"]\arrow[r,"\id"'] & \kf\arrow[r,"\id"']\arrow[u,"\id"'] & \kf
\end{tikzcd}
\hspace{5mm}
\oplus 
\begin{tikzcd}
\kf\arrow[r,"\id"] & \kf\arrow[r,"\id"] & \kf\\
& \kf\arrow[u,"\id"]\arrow[r,"\begin{bmatrix} 1\\ 1\end{bmatrix}"'] & \kf^2\arrow[r,"\begin{bmatrix} 0 \ 1\end{bmatrix}"']\arrow[u,"\begin{bmatrix} 1 \ 0\end{bmatrix}"'] & \kf
\end{tikzcd}\]
Since each summand of $M$ and $N$ is indecomposable, by Theorem \ref{thm:rkequalsintervals}, we have $\rank_M(I) = 1$ and $\rank_N(I)=0$. 
Now we claim that $M_{\SZZ} = N_{\SZZ}$. Since $M$ and $N$ are supported on $I$, it suffices to show $M_\Gamma\cong N_{\Gamma}$ for all maximal simple paths $\Gamma$ in $I$. Indeed, it is not difficult to check that $M_\Gamma\cong N_\Gamma$ for all of the six maximal simple paths $\Gamma$ in $I$: 

\begin{center}
\begin{tikzcd}[scale cd=.3]
\bullet\arrow[r] & \bullet \arrow[r,gray] & \color{gray}{\bullet} \\
& \bullet \arrow[r]\arrow[u] & \bullet\arrow[r]\arrow[u,gray] & \bullet
\end{tikzcd}\hspace{-3mm}
\begin{tikzcd}[scale cd=0.17]
\bullet\arrow[r] & \bullet\arrow[r] & \bullet \\
& \bullet\arrow[r,gray]\arrow[u,gray] & \bullet\arrow[r]\arrow[u] & \bullet
\end{tikzcd}\hspace{-3mm}
\begin{tikzcd}[scale cd=0.17]
\bullet\arrow[r] & \bullet\arrow[r,gray] & \bullet &\\
& \bullet \arrow[r]\arrow[u] & \bullet\arrow[u]\arrow[r,gray]&\color{gray}{\bullet}
\end{tikzcd}\\\vspace{5mm}

\begin{tikzcd}[scale cd=0.17]
\bullet\arrow[r] & \bullet \arrow[r]& \bullet &\\
& \bullet \arrow[u,gray]\arrow[r] & \bullet\arrow[u]\arrow[r,gray]&\color{gray}{\bullet}
\end{tikzcd}\hspace{-3mm}
\begin{tikzcd}[scale cd=0.17]
\color{gray}{\bullet}\arrow[r,gray]& \bullet \arrow[r]& \bullet &\\
& \bullet \arrow[r]\arrow[u] & \bullet \arrow[r]\arrow[u,gray] &\bullet
\end{tikzcd}\hspace{-3mm}
\begin{tikzcd}[scale cd=0.17]
\color{gray}{\bullet}\arrow[r,gray]& \bullet \arrow[r]& \bullet &\\
& \bullet \arrow[u]\arrow[r,gray] & \bullet\arrow[u]\arrow[r] &\bullet
\end{tikzcd}
\end{center}
Therefore, the ZIB over simple paths cannot determine the int-GRI in general. In addition, invoking the fact that the ZIB determines the int-GRI, this example proves that the ZIB  \emph{strictly} determines the ZIB over simple paths. 
\end{example}

\begin{remark}Let $d\geq 3$. Since $\Z^2$ can be embedded into $\Z^d$, for $\Z^d$-modules, neither the int-GRI nor the ZIB over simple paths determine the other.
\end{remark}

\begin{figure}
    \centering
    \includegraphics[width=0.72\textwidth]{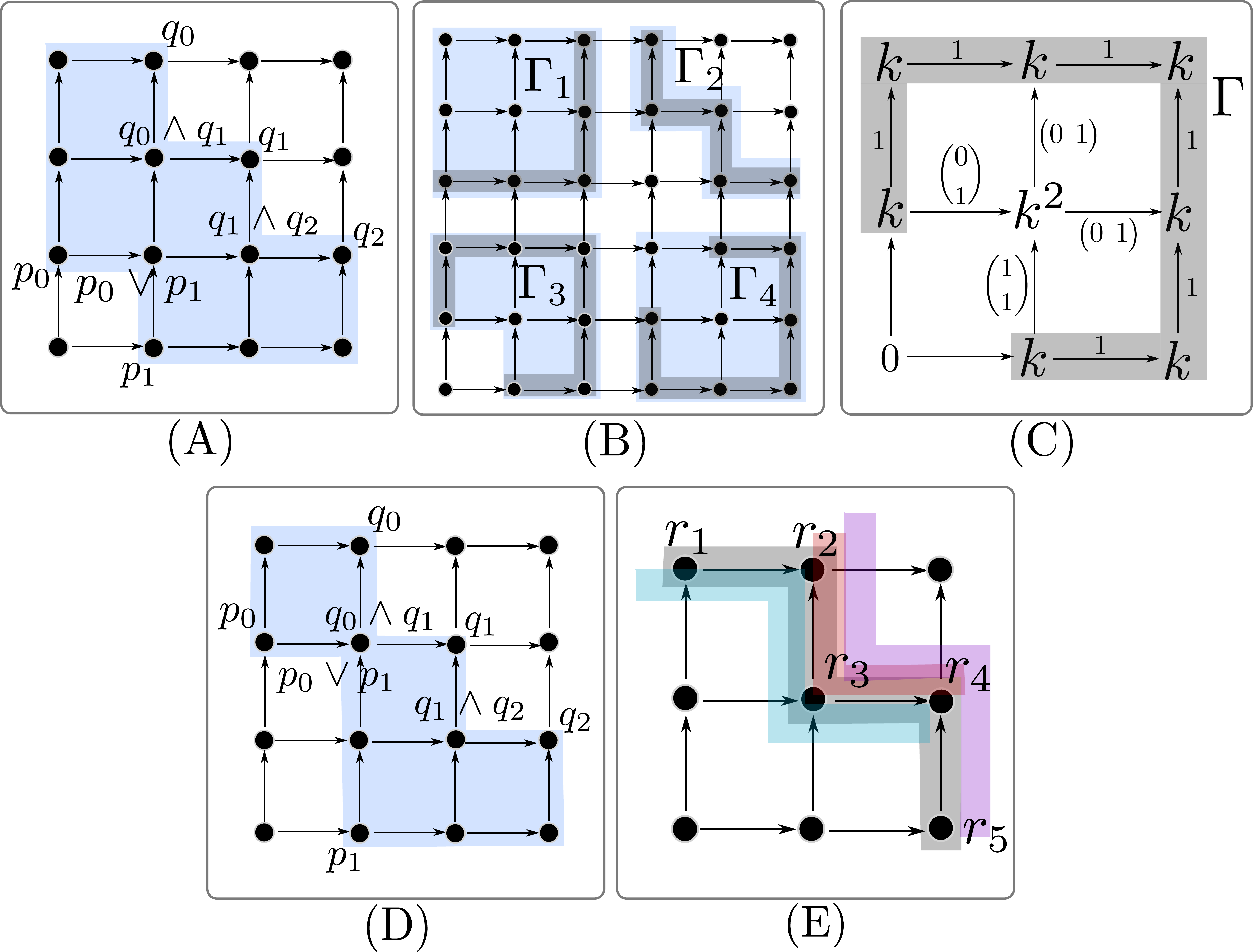}
    \caption{(A) An interval $I$ in $\Z^2$ with points in $\overmin(I)$ and $\overmax(I)$ labeled. (B) Paths $\Gamma_i$ for $i=1,2,3,4$ in $\Z^2$ (directions are not specified) and their interval-hulls. Whereas $\Gamma_i$ is a tame path for  $i=1,2,4$, $\Gamma_3$ is \emph{not} a tame path. (C) An example showing that the rank over the non-tame path $\Gamma$ does not coincide with the rank over the interval-hull $I_{\Gamma}$. (D) An interval in $\Z^2$ which is not solid nor thin. The interval in (A) is solid. The interval $I_{\Gamma_2}$ in (B) is thin. (E) Let $\Gamma$ be the path $r_1,\ldots,r_5$ and $\Gamma'$ be the subpath $r_2,r_3,r_4$. Then $\Gamma'^-$, $\Gamma'^+$, $\Gamma'^\pm$ are the paths $r_1,r_2,r_3,r_4$ and $r_2,r_3,r_4,r_5$ and $r_1,r_2,r_3,r_4,r_5$, respectively.  }
    \label{fig:scaffold}
\end{figure}

\paragraph{ZIB over simple paths and the int-GRI estimate each other.}\label{sec:5.2}

Although the int-GRI and the ZIB over simple paths fail to determine each other, we can \emph{estimate} one from the other. 
We describe how this is done in Remark \ref{rem:approximating Int rank invariant via simple zz} and Proposition \ref{prop:approximating zz} below.

Let $\Gamma$ be a path in $\Z^2$. Let $I_{\Gamma}$ be the smallest interval of $\Z^2$ that contains $\Gamma$, i.e. $$I_{\Gamma}:=\{q\in \Z^2:\exists p,r\in \Gamma,\ \ p\leq q\leq r \}.$$ We call $I_{\Gamma}$ the \textbf{interval-hull} of $\Gamma$.   See Figure \ref{fig:scaffold} (B) for illustrative examples.

Given any $I\in \fint(\Z^2)$, the set $\min (I)$ (resp. $\max (I)$) of minimal (resp. maximal) points of $I$ forms an  \emph{antichain}, i.e. any two different points in $\min (I)$ (resp. in $\max (I)$) are \emph{not} comparable. Hence, 
we can list the elements of $\min( I)$ in ascending order of their $x$-coordinates, i.e. $\min (I):=\{p_0,\ldots,p_k\}$ and such that for each $i=0,\ldots,k,$ the $x$-coordinate of $p_i$ is less than that of $p_{i+1}$. Similarly, let $\max (I):=\{q_0,\ldots,q_\ell\}$ be ordered in ascending order of their $x$-coordinates. We have that $p_0\leq q_0$ (cf. Figure \ref{fig:scaffold} (A)).
Let $\overmin(I)$ be the shortest faithful path in $\Z^2$ that contains the following sequence.
\begin{align}
    p_0<(p_0\vee p_1)>p_1<(p_1\vee p_2)>\cdots<(p_{k-1}\vee p_k)>p_k.
    \label{eq:zigzag posets}
\end{align}
Similarly, let $\overmax(I)$ be the shortest faithful path in $\Z^2$ containing
\begin{align}
    q_0>(q_0\wedge q_1)<q_1>(q_1\wedge q_2)<\cdots>(q_{\ell-1}\wedge q_\ell)>q_\ell. \label{eq:zigzag posets2}
\end{align}

 Let $\Gamma:p_1,\ldots, p_n$ be a path in $\Z^2$. By $\Gamma^{-1}$ we will denote its reverse path $p_n, p_{n-1}\ldots, p_1$. We call $\Gamma$ \textbf{tame} if:
\[\left[\overmin(I_{\Gamma})\leq \Gamma \mathrm{\ or\ } \overmin(I_{\Gamma})^{-1}\leq \Gamma\right] \mathrm{\ and\ }  \left[\overmax(I_{\Gamma})\leq \Gamma \mathrm{\ or\ } \overmax(I_{\Gamma})^{-1}\leq \Gamma\right].\]
For example, in Figure \ref{fig:scaffold} (B), $\Gamma_1$, $\Gamma_2$, and $\Gamma_4$ are tame but $\Gamma_3$ is not. Given a finite interval $I\subset \Z^2$, the concatenation of the paths  $\overmin(I)^{-1}$ and $\overmax(I)$ is called the \textbf{boundary cap} of $I$  \cite{dey2022computing}, which is an instance of a tame path. The boundary cap of $I$ is a simple path iff $\min_{\ZZ}(I)$ and $\max_\ZZ(I)$ do not intersect.

Let us fix a $\Z^2$-module $M$. 

\begin{proposition}\label{prop:rank via zigzag}  Given any tame path $\Gamma$ in $\Z^2$, it holds that $\rank(M_\Gamma)=\rank_M({I_\Gamma})$.  
\end{proposition}

This proposition is a generalization of \cite[Theorem 24]{dey2022computing} and its proof, which we defer to the appendix, is similar to that of \cite[Theorem 24]{dey2022computing}.
\medskip

Now, we describe how to estimate the $\fint$-GRI via the ZIB over simple paths.
An $I\in \fint(\Z^2)$ is called \textbf{solid} if $\min_{\ZZ}(I)$ and $\max_{\ZZ}(I)$ do not intersect (e.g. the interval given in Figure \ref{fig:scaffold} (A)). On the contrary, if $I$ equals $I_\Gamma$ for some negative path $\Gamma$, then $I$ is called \textbf{thin} (e.g. $\Gamma_2$ in Figure \ref{fig:scaffold} (B)). 

\begin{remark}
[Estimating the $\fint$-GRI via the ZIB over simple paths]
\label{rem:approximating Int rank invariant via simple zz} Let $I\in \fint(\Z^2)$. 
\begin{enumerate}[label=(\roman*),topsep=0pt,itemsep=-1ex,partopsep=1ex,parsep=1ex,leftmargin=*]
    \item If $I$ is a thin interval, then $I$ is covered by a simple path $\Gamma$. Thus, by Theorem \ref{thm:rkequalsintervals} and Proposition \ref{prop:rank via zigzag}, $\rank_M(I)$ equals the multiplicity of the `full bar' $I$ in $\barc(M_{\Gamma})$.
    \item If $I$ is solid, then the boundary cap $\Gamma$ of $I$ is a simple tame path such that $I=I_{\Gamma}$.
    Hence, $\rank_M(I)$  equals the multiplicity of the full bar in $\barc(M_{\Gamma})$. \label{item:approximating Int rank invariant via simple zz2}
    \item If $I$ is not solid nor thin, then there is no simple tame path spanning $I$ (see e.g. Figure \ref{fig:scaffold} (D)). 
    However, by monotonicity of the GRI (Remark \ref{rem:properties of rank invariant} \ref{item:monotonicity}), we have the following upper and lower bounds for $\rank_M(I)$: 
    \[ \max_{J} \rank_M(J)  \leq  \rank_M(I)\leq\min_{\Gamma}\rank(M_\Gamma)\]
    where the minimum is taken over all simple paths $\Gamma$ in $I$ and the maximum is taken over all solid intervals $J\supset I$. The both bounds can be obtained from $M_{\SZZ}$: Clearly, $\max_{\Gamma}\rank(M_\Gamma)$ is determined by $M_{\SZZ}$. From item \ref{item:approximating Int rank invariant via simple zz2}, $\max_{J} \rank_M(J)$ can also be computed from $M_{\SZZ}$.
\end{enumerate}
\end{remark}

\medskip

Conversely, we now describe how to estimate the ZIB over simple paths via the int-GRI.
Let $\Gamma$ be a simple path in $\Z^2$ and consider the following nonnegative integers:
\[ m_{\Gamma} := \rank_M(I_\Gamma) \mathrm{\ and \ } \ell_\Gamma:=\min \{\rank(M_{\Gamma'}): \Gamma'\leq \Gamma,\ \mathrm{and }\  \Gamma' \mathrm{\ is \ tame}\}.\]
By monotonicity of $\rank_M$, 
we have
\begin{equation}\label{eq:rank bounds}
    m_{\Gamma}\leq \rank(M_{\Gamma}) \leq \ell_{\Gamma}.
\end{equation}

\begin{remark}\phantomsection\label{rem:remark about bounding} By Proposition \ref{prop:rank via zigzag},
\begin{enumerate}[label=(\roman*),topsep=0pt,itemsep=-1ex,partopsep=1ex,parsep=1ex,leftmargin=*]  
    \item the int-GRI of $M$ \ok{can be used to compute}
    $m_{\Gamma}$ and $\ell_{\Gamma}$.\label{item:remark about bounding1}
    \item  if $\Gamma$ is tame, then $m_{\Gamma}=\ell_{\Gamma}=\rank(M_{\Gamma})$.\label{item:remark about bounding2}
\end{enumerate}
\end{remark}

 Let $\Gamma:r_1,r_2,\ldots,r_n$ be a path and let $\Gamma':r_k,r_{k+1},\ldots,r_{\ell}$ be a subpath of $\Gamma$. When $k\neq 1$, we consider the one-point extension $\Gamma'$ to \emph{the left}, i.e. $\Gamma'^-:r_{k-1},r_k,\ldots,r_\ell$. When $\ell\neq n$,  we consider the one-point extension of $\Gamma'$ to \emph{the right}, i.e. $\Gamma'^+:r_k,\ldots,r_{\ell} ,r_{\ell+1}$. When $k\neq 1$ and $\ell \neq n$, we consider the two-point extension $\Gamma'^\pm:r_{k-1},r_k,\ldots,r_\ell,r_{\ell+1}$ of $\Gamma'$ within $\Gamma$ (see Figure \ref{fig:scaffold} (E)).
 
 We obtain upper and lower bounds on the multiplicity of each subpath $\Gamma'$ of $\Gamma$ in $\barc(M_{\Gamma})$:

\begin{proposition}[Estimating the ZIB over simple paths via the $\fint$-GRI]
\label{prop:approximating zz}
Let $n_{\Gamma'}$ be the multiplicity of $\Gamma'$ in $\barc(M_\Gamma)$. Then, we have \begin{equation}\label{eq:multiplicity bounds}
    m_{\Gamma'}-\ell_{\Gamma'^+}-\ell_{\Gamma'^-}+m_{\Gamma'^{\pm}}\leq n_{\Gamma'} \leq \ell_{\Gamma'}-m_{\Gamma'^+}-m_{\Gamma'^-}+\ell_{\Gamma'^{\pm}},
\end{equation}
where the undefined terms in the upper and lower bounds are set to be zero.\footnote{For example, if $\Gamma$ and $\Gamma'$ have the same starting point but different end points, then $\Gamma'{^-}$ and $\Gamma'{^\pm}$ are undefined and thus Equation (\ref{eq:multiplicity bounds}) reduces to $  m_{\Gamma'}-\ell_{\Gamma'^+}\leq n_{\Gamma'}\leq \ell_{\Gamma'}-m_{\Gamma'^+}$.}
\end{proposition}

We remark that, by Remark \ref{rem:remark about bounding} \ref{item:remark about bounding2}, the upper and lower bounds match when (i) $\Gamma'^{\pm}\leq \Gamma$, and (ii) $\Gamma$ is either a monotone or negative path. 

\begin{proof}
By the principle of inclusion and exclusion, we have that \begin{equation}\label{eq:inclusion-exclusion over zigzag poset}
    n_{\Gamma'}=\rank(M_{\Gamma'})-\rank(M_{\Gamma'^+})-\rank(M_{\Gamma'^-})+\rank(M_{\Gamma'^\pm}) \ \ \ \mbox{(cf. \cite[Section 3]{kim2021generalized}).}
\end{equation}
The claimed inequalities follow from the inequalities in Equation (\ref{eq:rank bounds}). 
\end{proof}

\begin{remark}
  By Theorem \ref{thm:sufficient conditions for invertibility} \ref{item:finitely presentable invertible} and its proof, 
    if $M:\R^2\to \cvec$ is finitely presentable, then there is a discrete grid $G\subset \R^2$ such that the $\Int$-GRI of $M$ can be computed from the $\Int$-GRI of $M\vert_G$.
    The definitions and results in this section pertaining to the ZIB over $\Z^2$ naturally adapt to definitions and results for the setting of finite grids in $\R^2$, and so the ZIB is a useful tool for finitely presentable $\R^2$-modules as well.
\end{remark}

\begin{remark}\label{rem:estimation generalization}Generalizing Remark \ref{rem:approximating Int rank invariant via simple zz} and Proposition \ref{prop:approximating zz} to the setting of $\Z^d$-modules for $d\geq 3$ seems difficult, as it is 
unclear how to define 'tame paths in $\Z^d$' that can bridge between the GRI and the ZIB, as in Proposition \ref{prop:rank via zigzag}.
This uncertainty is supported by the fact that,  
for $d\geq 3$, there exists $\catP\in \fint(\Z^d)$ and a $\catP$-module $M$ with $\rk(M)\neq \rk(M_{\Gamma})$ for \emph{every} path $\Gamma$ in $\catP$.\footnote{
The recent work \cite{dey2024computing} utilizes zigzag persistence for computing the generalized rank of a persistence module over a general poset. This work is \emph{not} directly applicable to our setting. In that work, given a finite poset $\catP$ and a $\catP$-module $M$, the authors consider a path $\Gamma$ that covers $\catP$ and for it they compute a special decomposition of the zigzag module $M_{\Gamma}$ with the goal of determining the generalized rank of $M$. 
This special decomposition contains information that is not detected by $\barc(M_{\Gamma})$.}
For example, let $\catP:=\{(x,y,z)\in \Z^3:0\leq x,y,z\leq 1\}$ and consider the $\catP$-module:
\begin{center}
\tdplotsetmaincoords{60}{125}
\tdplotsetrotatedcoords{0}{30}{120} 
\begin{tikzpicture}[scale=2,tdplot_rotated_coords,
                    cube/.style={very thick,black},
                    grid/.style={very thin,gray},
                    axis/.style={->,blue,ultra thick},
                    rotated axis/.style={->,purple,ultra thick}]
\foreach \x in {0,1}
   \foreach \y in {0,1}
      \foreach \z in {0,1}{
           \ifthenelse{\x<1 }
           {
                \draw [->,black]   (\x+0.15,\y,\z) -- (\x+0.85,\y,\z);
           }
           {
           }
           \ifthenelse{ \y<1  }
           {
                \draw [->,black]   (\x,\y+0.15,\z)-- (\x,\y+0.85,\z);
           }
           {
           }
           \ifthenelse{ \z<1  }
           {
                \draw [->,black]   (\x,\y,\z+0.15) -- (\x,\y,\z+0.85);
           }
           {
           }}
         \node[scale=1] at (0,0,0) {$0$};
            \node at (0,0,1) {$k^2$};          
        \node at (1,0,0) {$k^2$};
                   \node at (0,1,0) {$k^2$};
            \node[scale=0.9] at (1,1,1) {$0$};
              \node[scale=0.9] at (1,1,0) {$k^2$};
                \node[scale=0.9] at (1,0,1) {$k^2$};
                  \node[scale=0.9] at (1,1,0) {$k^2$};                  
                  \node[scale=0.9] at (0,1,1) {$k^2$};
                  \node[scale=0.9] at (0.73,0.35,0) {$\phi_{\lambda}$}; 
 \end{tikzpicture}
\end{center}
in which all  arrows $k^2\rightarrow k^2$ denote the identity map with the exception of the arrow labeled with $\phi_\lambda$ where $\phi_{\lambda}:=\begin{pmatrix}\lambda &1 \\0 &\lambda \end{pmatrix}$, $\lambda\neq 0$. Let $I:=\{p\in \catP:M_p=k^2\}$, \ok{which is} an interval of $\catP$. Then, it is not difficult to verify that for \emph{every} path $\Gamma$ in $I$, we have $\rk(M_{\Gamma})=2$, whereas $\rk_M(I)=0$.
\end{remark}

\subsubsection{Comparison with  bigraded Betti numbers}\label{sec:bigraded Bettis}

In this section, we show that the bigraded Betti numbers of a \(\mathbb{Z}^2\)-module \(M\) are a much weaker invariant than the int-GRI of \(M\). This is in stark contrast with the fact that, for \(\mathbb{Z}^d\)-modules with \(d \geq 3\), the int-GRI is not a stronger invariant than the \(d\)-graded Betti numbers; see \cite[Theorem 4.1]{kim2021bettis}. We do this by showing that the bigraded Betti numbers do not even determine the GRI over simple paths of length 3, while the reverse is already known to hold.

 \begin{proposition}\label{prop:GRI determines bigraded} The GRI of any $\Z^2$-module over the zigzag posets
 in Equation (\ref{eq:zigzag posets in Z2}) below
 determines the bigraded Betti numbers of $M$.
 \end{proposition}

 \begin{proof}The bigraded Betti numbers of $M$ is determined by the barcodes over the zigzag posets 
\begin{equation}\label{eq:zigzag posets in Z2}
(x,y+1)\leftarrow (x,y)\rightarrow (x+1,y) \ \mbox{and} \ (x-1,y)\rightarrow (x,y) \leftarrow (x,y-1)\ \mbox{for all $x,y\in \Z^2$};
\end{equation}
 see \cite[Theorem 2.1]{moore2020combinatorial}. Since the zigzag modules are interval-decomposable (Theorem \ref{thm:zigzag modules are interval-decomposable}), and the Int-GRI of any interval-decomposable module determines its barcode (Corollary \ref{cor:gricomplete}), the claim follows.
 \end{proof}

The reverse statement does not hold.
\begin{proposition}\label{prop:bigraded Bettis do not determine GRI over length-3 zigzags}\contributionn
The bigraded Betti numbers of a $\Z^2$-module do not determine its GRI over zigzag paths of length 3.
\end{proposition}
\begin{proof} It suffices to construct a pair of $\Z^2$-modules that have the same bigraded Betti numbers but different barcodes over some zigzag path of length 3.
    Let $a = (0,2)$, $b = (1,1)$, $c = (2,0)$ and $D = (2,2)$.
    Let $M$ and $N$ be $\Z^2$-modules defined as 
    \[M:=\kf_{a^\uparrow \cup c^\uparrow}\oplus \kf_{b^\uparrow}  \ \ \mbox{and} \ \  N:=\kf_{a^\uparrow}\oplus \kf_{c^\uparrow}\oplus \kf_{b^\uparrow \setminus D^\uparrow}.\]
    It is straightforward to compute that the 
    zeroth Betti number for each module is 1 at $a$, $b$ and $c$, and that the first Betti number for each is 1 at $D$. For the both modules, the second and higher Betti numbers vanish.

    Now, consider $\Gamma$, the zigzag path of length 3 given by $a \rightarrow D \leftarrow c$.
    Since there is an interval summand over the upset $a^\uparrow \cup c^\uparrow$ in $M$ but not $N$, the zigzag barcode of M over $\Gamma$ includes a full bar, whereas the zigzag barcode of $N$ over $\Gamma$ does not include a full bar.
\end{proof}
\subsection{Generalized rank invariant compared to other invariants}\label{sec:comparison}
This section is dedicated to providing detailed explanations of Table \ref{table:comparison}.

\begin{enumerate}
\item[Row 1.] The dimension function of $M:\catP\rightarrow \cvec$, defined by the map $p\mapsto \dim (M_p)$ for $p\in \catP$, can be viewed as the  GRI over $\Jcal:=\{\{p\}:p\in \catP\}$. Clearly, $\rk_M^{\Jcal}$ is M\"obius invertible over $\Jcal$, and its M\"obius inversion is identical to itself.
\item[Row 3.] We call a line $\ell$ in $\R^d$ \emph{monotone}, if $\ell$ forms a totally ordered set with the order inherited from $\R^d$. In a fixed line $\ell$, let $\seg(\ell)$ be ordered by $\supset$. For different monotone lines, we declare that $\ell$ and $\ell'$, any pair of elements from $\seg(\ell)$ and $\seg(\ell')$ is not comparable. For any fixed line $\ell$, since $M|_{\ell}$ is interval-decomposable \cite{crawley2015decomposition}, $\rk_M^{\seg(\ell)}$ is M\"obius invertible (Theorem \ref{thm:sufficient conditions for invertibility} \ref{item:interval-decomposable}). Hence, the GRI of $M$ is also M\"obius invertible over the disjoint union $\bigsqcup \{\seg(\ell):\mbox{$\ell$ is monotone in $\R^d$}\}$.
\item[Row 6-7.] See \cite{asashiba2019approximation}.
\item[Row 8.] `$\Rightarrow$ Bigraded Betti numbers' follows from Propositions \ref{prop:GRI determines bigraded} and \ref{prop:bigraded Bettis do not determine GRI over length-3 zigzags}.
\item[Row 9.] `$\Rightarrow$ Int-GRI' follows from \cite[Theorem 3.12]{dey2022computing}.
\item[Row 10] `$\not\Leftrightarrow$ Int-GRI' follows from Examples \ref{ex:Int does not recover simple ZZ} and \ref{ex:intvszz}. 
\end{enumerate}

\section{Stability of the generalized rank invariant}\label{sec:stability}

In Section \ref{sec:stability of GRI}, we prove that the GRI  is stable in the erosion distance sense and thereby, in Section \ref{sec:5.3}, we establish a stability result for ZIBs.

\subsection{Stability of the generalized rank invariant}\label{sec:stability of GRI}

In this section, we prove that the GRI over $\Ical$ is stable in the erosion distance sense as long as $\Ical$ is closed under thickenings: see Definition \ref{def:closed under thickenings} and Theorem \ref{thm:gristability}.  

\ok{For ease of notation, in this section we focus on the setting of $\R^d$ or $\Z^d$-modules with the usual interleaving distance. Appendix \ref{sec:appendix for stability} deals with more general settings, specifically the stability of GRIs of $\catP$-modules under appropriate assumptions on $\catP$.} 

We begin by reviewing the definition of interleaving distance between $\R^d$ (or $\Z^d$)-modules \cite{lesnick2015theory}.
Let $M:\R^d\to \cvec$ and $\epsilon \in \R_{\geq 0}$.
Denote $\bareps:=\epsilon(1,\ldots,1)\in \R^d$.
Define the $\bareps$-shift of $M$, $M^{\bareps} :\R^d\to \cvec$ by $M^{\bareps}(a) := M(a)$ and $\varphi_{M^{\bareps}} (a, b):= M(a+\bareps\leq b+\bareps)$ for all $a,b\in \R^d$.
For a morphism of modules $\alpha$, define $\alpha(\bareps)_a=\alpha_{a+\bareps}$. 
Define the transition morphism $\varphi_M^{\bareps}:M\to M^{\bareps}$ as the morphism whose restriction to $M(a)$ is the linear map $\varphi_M(a, a+\bareps)$ for each $a\in \R^d$.
For $\epsilon\geq 0$, we say $\R^d$-modules $M$ and $N$ are $\epsilon$-\textbf{interleaved} if there exist morphisms $\alpha:M\to N^{\bareps}$ and $\beta:N\to M^{\bareps}$ such that $\beta(\bareps)\circ \alpha = \varphi_M^{2\bareps}$ and $\alpha(\bareps)\circ \beta = \varphi_N^{2\bareps}$.
The \textbf{interleaving distance} is defined as:
\[\dI(M,N) := \inf\{\epsilon\geq 0  :  M \,\mathrm{and}\, N \,\mathrm{are}\,\epsilon-\mathrm{interleaved}\},\]
and $\dI(M,N):=\infty$ if no such $\epsilon$-interleavings exist.

For $I\in \Int(\R^d)$, define the $\epsilon$-\emph{thickening} of $I$, $I^\epsilon$, as the set 
\[I^\epsilon := \{p\in \R^d\  :  \exists q\in I \, s.t. \, \|p-q\|_\infty \leq \epsilon\}.\]

\begin{definition}\label{def:closed under thickenings}
Let $\mathcal{I}\subset \Int(\R^d)$. 
We say that $\mathcal{I}$ is \textbf{closed under thickenings} if for all $\epsilon\geq 0$ and $I\in \mathcal{I}$, $I^\epsilon\in \mathcal{I}$.
\end{definition}

 For example, the set $\Int_{m.n}(\R^d)$ given in Equation (\ref{eq:intmn}) is closed under thickenings.
   
We now adapt the definition of erosion distance from \cite{patel2018generalized} to our setting:
\begin{definition}\label{def:erosion distance Z2}
Let $\mathcal{I}\subset \Int(\R^d)$ be closed under thickenings, and $M,N:\R^d\to \cvec$.
We say there is an $\epsilon$-\textbf{erosion} between $\rank_M^\mathcal{I}$ and $\rank_N^\mathcal{I}$ if for all $I\in \mathcal{I}$, we have
\[\rank_M(I^\epsilon)\leq \rank_N(I) \hspace{10mm} \mathrm{and} \hspace{10mm} \rank_N(I^\epsilon)\leq \rank_M(I)\]
The \textbf{erosion distance} between $\rank_M^\mathcal{I}$ and $\rank_N^\mathcal{I}$ is defined as:
\[\dE(\rank_M^\mathcal{I},\rank_N^\mathcal{I}):=\inf\{\epsilon\geq 0 :  \exists \, \mathrm{an} \, \epsilon-\mathrm{erosion \ between\ }\rank_M^\mathcal{I} \,\mathrm{and} \, \rank_N^\mathcal{I}\},\]
and $\dE(\rank_M^\mathcal{I},\rank_N^\mathcal{I}):=\infty$ if no such erosion exists.
\end{definition}

\ok{We remark that if $\dE(\rank_M^\mathcal{I},\rank_N^\mathcal{I})<\infty$, then $\dE(\rank_M^\mathcal{I},\rank_N^\mathcal{I})$ is a non-negative integer.}
We establish the following stability result:\footnote{\ok{A similar theorem is provided in the arXiv version of \cite{kim2021generalized}, but not in the published version.} }

\begin{theorem}\label{thm:gristability}
Let $\mathcal{I}\subset \Int(\R^d)$ be closed under thickenings. Then, for any $M,N:\R^d\to \cvec$,
\begin{equation}\label{eq:gristabilityeq}
    \dE(\rank_M^\mathcal{I},\rank_N^\mathcal{I})\leq \dI(M,N).
\end{equation}
\end{theorem}

\begin{proof}[Proof of Theorem \ref{thm:gristability}]
If $\dI(M,N)=\infty$, there is nothing to prove. 
Let $\epsilon\geq 0$, and suppose that $\alpha:M\to N^{\bareps}$ and $\beta:N\to M^{\bareps}$ give an $\epsilon$-interleaving. 
Fix $I\in \mathcal{I}$. 
We will show that $\rank_N(I^\epsilon)\leq \rank_M(I)$. To this end, 
it suffices to show that there exist morphisms $\alpha'$ and $\beta'$ that make the following diagram commute: 
\begin{equation}\label{eq:lim-to-colim diagram}
\begin{tikzcd}
\varprojlim N\vert_{I^\epsilon} \arrow[r,"\psi_{N\vert_{I^\epsilon}}"]\arrow[d,"\alpha'"] & \varinjlim N\vert_{I^\epsilon}\\
\varprojlim M\vert_I \arrow[r,"\psi_{M\vert_I}"] & \varinjlim M\vert_I \arrow[u,"\beta'"]
\end{tikzcd}\end{equation}
Indeed, if such $\alpha'$ and $\beta'$ exist, then the limit-to-colimit map $\psi_{N\vert_{I^\epsilon}}$, whose rank is $\rank_N(I^\epsilon)$, factors through the  limit-to-colimit map $\psi_{M\vert_{I}}$, whose rank is $\rank_M(I)$, implying the desired bound.

Define $\alpha'$ by $(\ell_p)_{p\in I^\epsilon}\mapsto (\alpha_p(\ell_p))_{p+\bareps\in I}$. 
Then, from naturality of $\alpha$, and the fact that  $p+\bareps\in I$ implies $p\in I^\epsilon$, we have that $(\alpha_p(\ell_p))_{p+\bareps\in I}$ is a section of $M\vert_I$. 
Define $\beta'$ by $[v_p]\mapsto [\beta_p(v_p)]$ for any $p\in I$ and $v_p\in M_p$. 
By naturality of $\beta$, and since $p+\bareps\in I^\epsilon$, $\beta'$ is well-defined.

Fix any $p_0\in I$. 
Then, both $p_0+\bareps$ and $p_0-\bareps$ belong to $I^\epsilon$. 
Let $(\ell_p)_{p\in  I^\epsilon}$ be an element of $\varinjlim N\vert_{I^\epsilon}$ and observe commutativity of the following diagram:

\[\begin{tikzcd}
(\ell_p)_{p\in I^\epsilon} \arrow[r]\arrow[d,"\alpha'"] & \left[\ell_{p_0-\bareps}\right] = \left[\varphi_N(p_0-\bareps,p_0+\bareps)(\ell_{p_0-\bareps})\right]\\
(\alpha_p(\ell_p))_{p+\bareps\in I} \arrow[r] & \left[\alpha_{p_0-\bareps}(\ell_{p_0-\bareps})\right] \arrow[u,"\beta'"]
\end{tikzcd}\]
Thus, $\rank_N(I^\epsilon)\leq \rank_M(I)$, and a symmetric argument gives $\rank_M(I^\epsilon)\leq \rank_N(I^\epsilon)$, so we obtain $d_E(\rank_M^\Ical,\rank_N^\Ical)\leq \epsilon$, as desired.
\end{proof}

\begin{remark}[Computational efficiency vs. discriminating power]\label{rem:tension remark}
  Corollary \ref{cor:intmn discriminating power} implies that, as $m$ and $n$ increase, the discriminating power of $\rk^{\Int_{m,n}(\R^d)}$  strictly increases. 
    Hence, if we let $\Ical = \Int_{m,n}(\R^d)$
    in 
    Theorem \ref{thm:gristability},  we expect $\dE(\rank_M^\mathcal{I},\rank_N^\mathcal{I})$ to give an increasingly better approximation to $\dI(M,N)$ as $m,n$ increase.

    On the other hand, as $m$ and $n$ increase, the cost associated to computing $\dE(\rank_M^\mathcal{I},\rank_N^\mathcal{I})$ is also expected to increase. To illustrate this, by adapting the binary-search based algorithm described in \cite[Section 5]{kim2021spatiotemporal},\footnote{This algorithm was implemented in \cite{PHoDMSs}.} we can  verify that the resulting computational cost of $\dE$ between $\Z^d$-modules that are supported on the finite grid $[\ell]^d$ is $O(\ell^{d(m+n)}\log \ell)$ in time.
  \end{remark}

\subsection{Stability of  (restricted) ZIBs}\label{sec:5.3}

In this section, we reinterpret Theorem \ref{thm:gristability} in terms of  (restricted) ZIBs. 

A path $\Gamma$ in a poset $\catP$ is called \textbf{rank-representing} if for every $\catP$-module $M$, $\rk(M)=\rk(M_\Gamma)$. A subset $\Ical\subset \Int(\catP)$ is called \textbf{ZIB-rank-representing} if  each $I\in \Ical$ admits at least one rank-representing path $\Gamma$ in $I$. 

Let $\Ical\subset \Int(\R^d)$ be closed under thickenings and ZIB-rank-representing. 
For example, when $d=2$, the collection $\Ical$ can be $\Int_{m,n}^\mathrm{cc}(\R^2)$ (cf. Equation (\ref{eq:intmn cc})).
Let
\[\ZZ_{\Ical}:=\{\Gamma: \mbox{$\Gamma$ is a rank-representing path for some $I\in \Ical$}\}.\]
For $\Gamma\in \ZZ_{\Ical}$ and $\eps\geq 0$, let $\Gamma^\eps$ be any path that is rank-representing of $I^\eps$. 
We remark that $\Gamma^\eps$ may not be unique but exists in $\ZZ_{\Ical}$ by the assumption on $\Ical$.
Let $M_{\ZZ_{\Ical}}$ be the ZIB of $M$ on $\ZZ_{\Ical}$, i.e. the map sending each $\Gamma\in \ZZ_{\Ical}$ to $\barc(M_\Gamma)$. 
Given another $\R^d$-module $N$, we define the erosion distance $\dE(M_{\ZZ_{\Ical}},N_{\ZZ_{\Ical}})$ as the infimum $\eps\geq 0$ for which, for any $\Gamma\in \ZZ_{\Ical}$, we have $\rank(M_{\Gamma^\eps})\leq \rank(N_{\Gamma})$ and  $\rank(M_{{\Gamma}^\eps})\leq \rank(N_{\Gamma})$. In light of Theorems \ref{thm:zigzag modules are interval-decomposable}, \ref{thm:rkequalsintervals} and Proposition \ref{prop:rank via zigzag}, the condition $\rank(M_{\Gamma^\eps})\leq \rank(N_{\Gamma})$ can be read as: 
\begin{equation*}
    (\mathrm{Multiplicity \ of \ the \ full \ bar \ in \ }\barc(M_{\Gamma^\epsilon}))\ \leq \ (\mathrm{Multiplicity \ of \ the \ full \ bar \ in \ }\barc(N_\Gamma)).
\end{equation*}

\begin{theorem}[Reinterpretation of Theorem \ref{thm:gristability}]\label{thm:ZIB stability} 
Let $\Ical \subset \Int(\R^d)$ be any collection of intervals closed under thickenings and ZIB-rank-representing. Then, for any $\R^d$-modules $M$ and $N$ 
\begin{equation}\label{eq:erosion stability}
    \dE(M_{\ZZ_{\Ical}},N_{\ZZ_{\Ical}})\leq \dI(M,N).
\end{equation}
\end{theorem}

\begin{remark}\label{rem:erosion stability generalization}
When $d=2$, the collection $\Ical=\Int_{m,n}^\mathrm{cc}(\R^2)$ for any $m,n\in \N$ is a salient example for which this theorem is applicable. When $d\geq 3$, one such an example is $\Ical=\seg(\R^d)$. However, identifying some other such collections for $d\geq 3$ does not seem  straightforward for  a reason  analogous to the one described in Remark \ref{rem:estimation generalization}.
\end{remark}

\section{Discussion}\label{section:discussion}

We have addressed the four driving questions  
\begin{enumerate}[leftmargin=*]
\item[] \hyperlink{q:1}{\textbf{Question 1.}} 
How to restrict, if possible, the domain of the generalized rank invariant without any loss of information? 
\item[] \hyperlink{q:2}{\textbf{Question 2.}} 
Under what conditions can we more compactly encode the GRI as a 'persistence diagram,' even when the indexing poset $\catP$ is not discrete?

\item[] \hyperlink{q:3}{\textbf{Question 3.}} What is the trade-off  between
computational efficiency and the discriminating power of the GRI as the amount of the restriction varies? 

\item[] \hyperlink{q:4}{\textbf{Question 4.}} What proxies exist for persistence diagrams in the multi-parameter setting that can be derived from the GRI?
\end{enumerate}
Below, we delve into a selection of future research directions.
\medskip
\begin{itemize}
    \item 
Given that in the case of $\R^2$ or $\Z^2$-modules the generalized rank invariant and  the ZIB estimate each other (Section \ref{sec:gri vs. zz}),
 it becomes natural to seek an efficient algorithm for computing or estimating the generalized persistence diagram of a finitely presentable $\R^2$-module by harnessing zigzag persistence update algorithms \cite{dey2021updating}. Such an algorithm could also potentially be useful for estimating other invariants that are closely related to the generalized rank invariant
\cite{amiot2024invariants,asashiba2023approximation,blanchette2021homological,botnan2021signed,chacholski2022effective}. 

\item  A related direction is to study, for the case when $d\geq 3$, how to estimate the restricted or full GRI of an $\R^d$ or $\Z^d$-module via  its ZIB over a set of zigzag paths of manageable size, and to utilize the results of this study for establishing a stability result for ZIB; see Remarks \ref{rem:estimation generalization} and \ref{rem:erosion stability generalization}. 

\item Theorem \ref{thm:ZIB stability} suggests the possibility of utilizing zigzag persistence for an efficient estimation of the interleaving distance.

\item For the $\Z^2$-module $M$ given in the proof of Theorem \ref{thm:tame does not imply Int-GRI invertible}, the fact that $\rk_M^\Int$ is not M\"obius invertible implies that there exists \emph{no} finite exact sequence of interval-decomposable persistence modules $\{M_i\}_{i=0}^n$ and morphisms  $\{f_i\}_{i=0}^n$ with
\[0\rightarrow M_n\stackrel{f_n}{\rightarrow}\cdots\rightarrow  M_1\stackrel{f_1}{\rightarrow} M_0 \stackrel{f_0}{\rightarrow} M\rightarrow 0\]
such that $\rk_{M_i}^{\Int}=\rk_{\ker f_i}^\Int +\rk_{\ker f_{i-1}}^\Int$ for each $i=0,
\ldots, n$. Indeed, the existence of such a sequence implies the rank decomposition $\rk_M^\Int=\rk_{\bigoplus_{i}M_{2i}}^{\Int}-\rk_{\bigoplus_{i}M_{2i-1}}^\Int$ and thus the M\"obius invertibility of $\rk_M^{\Int}$. More generally, Remark \ref{rem:higher dimension} clarifies that there are $\Z^d$-modules, $d> 2$, that do not admit such finite sequences. This observation might have interesting implications in the perspective of recent studies at the intersection of representation theory and multi-parameter persistence \cite{amiot2024invariants,asashiba2023approximation,blanchette2021homological,botnan2021signed,chacholski2022effective}.
\end{itemize}

\bibliographystyle{plain}
\bibliography{references.bib}

\appendix

\section{Stability of the generalized rank invariants over general posets}\label{sec:appendix for stability}

\ok{In this section, we extend the erosion distance between the GRIs \cite{patel2018generalized,puuska2020erosion} to the general setting of $\catP$-modules and generalize Theorem \ref{thm:gristability} to this general setting as well (Definition \ref{def:erosion distance1} and Theorem \ref{thm:gristabilityold}). }

\ok{
To compare $\catP$-modules, we consider the interleaving distance between $\catP$-modules, developed by Bubenik et al. \cite{bubenik2015metrics} and expanded upon by de Silva et al. \cite{de2017theory}. This interleaving distance is an extension of the classical interleaving distance between $\R^d$-modules \cite{chazal2009proximity,lesnick2015theory}.}

\ok{Let $\catP$ be a poset, viewed as a category. 
For two order-preserving maps $T_1,T_2:\catP\to\catP$, we write $T_1\leq T_2$ if for all $p\in \catP$, $T_1(p)\leq T_2(p)$. Let $I_\catP$ be the identity functor on $\catP$.} 

\begin{definition}\label{def:superlinear family}
A \textbf{translation} on $\catP$ is an endofunctor $T:\catP\rightarrow \catP$ together with a natural transformation $\eta:I_\catP\to T$
\ok{(i.e. $p\leq T(p)$ for all $p\in \catP$).} 
A \textbf{family of superlinear translations} $\Omega=(\Omega_\epsilon)_{\epsilon \geq 0}$, is a family of translations $\Omega_\epsilon$ on $\catP$, for $\epsilon\geq 0$, such that $\Omega_0 = I_\catP$, and for $\epsilon,\zeta \geq 0$, $\Omega_\epsilon \Omega_\zeta \leq \Omega_{\epsilon+\zeta}$. 
\end{definition}

Throughout the following, $\Omega$ will refer to a family of superlinear translations on $\catP$.
For $\Omega_\epsilon$ a superlinear translation on $\catP$ and $I\subset \catP$, we denote $\Omega_\epsilon(I):=\{\Omega_\epsilon(p)  :  p\in I\}$. For all $\epsilon\geq 0$, the translation $\Omega_\epsilon$ comes with a natural transformation $\eta_\epsilon:I_\catP\to \Omega_\epsilon$. 
For any $M:\catP\to \cvec$, this induces a natural transformation $M\eta_\epsilon:M\to M\Omega_\epsilon$. 
This is used to define:

\begin{definition}\label{def:interleaving distance}
Two $\catP$-modules $M$ and $N$ are $\Omega_\epsilon$\textbf{-interleaved} if there exist a pair of natural transformations $\varphi:M\to N\Omega_\epsilon$ and $\psi:N\to M\Omega_\epsilon$ such that the diagram:
\[\begin{tikzcd}
M \arrow[rr,"M\eta_\epsilon"]\arrow[ddrr,"\varphi",pos=0.8] & &  M\Omega_\epsilon \arrow[rr,"M\eta_\epsilon\Omega_\epsilon"]\arrow[ddrr,"\varphi\Omega_\epsilon",pos=0.8] & & M\Omega_\epsilon \Omega_\epsilon\\
\\
N \arrow[rr,"N\eta_\epsilon",swap]\arrow[uurr,"\psi", pos=0.2,crossing over] & & N\Omega_\epsilon \arrow[rr,"N\eta_\epsilon\Omega_\epsilon",swap] \arrow[uurr,"\psi\Omega_\epsilon", pos=0.2,crossing over] & & N\Omega_\epsilon \Omega_\epsilon
\end{tikzcd}\]
commutes. 
The pair $(\varphi,\psi)$ is said to be an $\Omega_\epsilon$\textbf{-interleaving}.

The \textbf{interleaving distance with respect to} $\Omega$ is
\[\dI^\Omega(M,N) := \inf\{\epsilon\geq 0  :  M,N \ \mbox{are $\Omega_\epsilon$-interleaved}\},\]
or $\dI^\Omega(M,N):=\infty$ if there is no $\Omega_\epsilon$-interleaving for any $\epsilon\geq 0$.
\end{definition}

For example, let $P = \Z^n$ with the standard product order, and let $\Omega$ be the family with $\Omega_\epsilon$ the translation by $(\lfloor \epsilon\rfloor ,\lfloor \epsilon \rfloor,\ldots,\lfloor \epsilon \rfloor)\in \Z^n$.
Then $\dI^\Omega$ is the classical notion of interleaving used in Section \ref{sec:stability}.

\ok{We now extend the definition of erosion distance in a suitable way.}

\begin{definition}\label{def:epsthickening}
Let $I\subset \catP$ be non-empty.
For $\epsilon> 0$, we call $I^\epsilon$ the $\epsilon$-\textbf{thickening} of $I$, defined as:
\[I^\epsilon := \{r\in \catP  :  \exists p,p'\in I \ \mathrm{with} \ p\leq \Omega_\epsilon(r) \mathrm{\ and} \ r\leq \Omega_\epsilon(p')\}.\]
\end{definition}

Clearly, $I\subset I^\eps$. For example, 
if $P=\R$, and $\Omega$ is the family with $\Omega_\epsilon$ the translation by $\epsilon$ for $\epsilon\geq 0$, then for an interval $I=[a,b]$, its $\epsilon$-thickening would be the interval $I^\epsilon = [a-\epsilon,b+\epsilon]$. 
Furthermore, the $\epsilon$-thickening of an interval is an interval:

\begin{proposition}\label{prop:eps thickening is interval}
Let $I\in \Int(\catP)$. 
Then, $I^\epsilon\in \Int(\catP)$ for all $\epsilon\geq 0$. 
\end{proposition}
\begin{proof}
Let $\epsilon\geq 0$.
We need to show that $I^\epsilon$ is non-empty, convex, and connected. 

Since $I(\neq\emptyset)\subset I^\eps$, we have that $I^\eps$ is non-empty. Suppose $a,b\in I^\epsilon$, and $a\leq c\leq b$. 
By definition, there exist $p_a,p'_a,p_b,p'_b\in I$ such that $p_a\leq \Omega_\epsilon(a)$, $p_b\leq \Omega_\epsilon(b)$, $a\leq \Omega_\epsilon(p'_a)$, and $b\leq \Omega_\epsilon(p'_b)$. 
Then, we have: $p_a\leq \Omega_\epsilon(a)\leq \Omega_\epsilon(c)$ and $c\leq b\leq \Omega_\epsilon(p'_b)$, and thus $c\in I^\eps$, so $I^\epsilon$ is convex.

We now establish connectivity. First note that for all $r\in I$, the point $\Omega_\epsilon(r)$ belongs to $I^\epsilon$, which can be seen by letting $p=p'=r$ in Definition \ref{def:epsthickening}. Suppose $p,p'\in I^\epsilon$.
Then, we can find $r_p,r_{p'}\in I$ with $p\leq \Omega_\epsilon(r_p)$ and $p'\leq \Omega_\epsilon(r_{p'})$. As $r_p,r_{p'}\in I$, there is a chain $r_p=a_0,a_1,\ldots,a_n=r_{p'}$ of sequentially comparable elements of $I$. Then $p\leq \Omega_\epsilon(r_p)\geq r_p = a_0,a_1,\ldots,a_n=r_{p'}\leq \Omega_\epsilon(r_{p'}) \geq p'$ gives a chain of sequentially comparable elements of $I^\epsilon$, so $I^\epsilon$ is connected.
\end{proof}

\begin{definition}\label{def:closed under Omega thickenings}
If $\mathcal{I}\subset \Int(\catP)$ and for all $I\in \mathcal{I}$ and $\epsilon\geq 0$, $I^\epsilon\in \mathcal{I}$, then we say $\mathcal{I}$ is \textbf{closed under} $\Omega$-\textbf{thickenings}.
\end{definition}

\begin{example}\label{ex:rectangle thickening}
If $P = \R^2$ with the usual product order, and $\Omega$ is the family with $\Omega_\epsilon$ the translation by $(\epsilon,\epsilon)$ for $\epsilon\geq 0$, then for a rectangle $I=[a,b]$, its $\epsilon$-thickening would be $I^\epsilon=[a-(\epsilon,\epsilon),b+(\epsilon,\epsilon)]$. This is still a rectangle, so the collection of rectangles in $\R^2$ is closed under $\Omega$-thickenings.
\end{example}
Now we define the erosion distance for $\catP$-modules:

\begin{definition}\label{def:erosion distance1}
Let $\mathcal{I}\subset \Int(\catP)$ be closed under $\Omega$-thickenings, and let $M,N$ be $\catP$-modules.  
We say there is an $\epsilon$-erosion between $\rank_M^\mathcal{I}$ and $\rank_N^\mathcal{I}$ if for all $I\in \mathcal{I}$, we have
\[\rank_M(I^\epsilon)\leq \rank_N(I) \hspace{5mm} \mathrm{and} \hspace{5mm} \rank_N(I^\epsilon)\leq \rank_M(I).\]
Define the \textbf{erosion distance} (with respect to $\Omega$) between $\rank_M^\mathcal{I}$ and $\rank_N^\mathcal{I}$ as:
\[d_E^\Omega(\rank_M^\mathcal{I},\rank_N^\mathcal{I}) := \inf\{\epsilon\geq 0  :  \exists  \ \mathrm{an \ } \epsilon-\mathrm{erosion \ between \ }\rank_M^\mathcal{I} \ \mathrm{and} \ \rank_N^\mathcal{I} \}\]
and $d_E^\Omega(\rank_M^\mathcal{I},\rank_N^\mathcal{I}):=\infty$ if  $\epsilon$-erosions do not exist for any $\epsilon \geq 0$.\end{definition}

\begin{proposition}\label{prop: erosion is pseudometric}
Fix a collection $\mathcal{I}\subset \Int(\catP)$ that is closed under $\Omega$-thickenings. 
Then, $d_E^\Omega$ is an extended pseudometric on the collection $\{\rank_M^\mathcal{I}  :  M:\catP\to \cvec\}$.
\end{proposition}
\begin{proof}
Since $\Omega_0=I_\catP$, 
it is immediate that $d_E^\Omega(\rank_M^\Ical,\rank_M^\Ical) = 0$. Symmetry is immediate from the definition. 

It remains to show the triangle inequality. Note that $\Omega_\epsilon\Omega_{\epsilon'}\leq \Omega_{\epsilon+\epsilon'}$. This implies $(I^\epsilon)^{\epsilon'}\subset I^{\epsilon+\epsilon'}$. From this, if there is an $\epsilon$-erosion between $\rank_M^\Ical$ and $\rank_N^\Ical$, and an $\epsilon'$-erosion between $\rank_N^\Ical$ and $\rank_L^\Ical$, then for all $I\in \mathcal{I}$:
\[\rank_M(I)\geq \rank_N(I^\epsilon)\geq \rank_L((I^\epsilon)^{\epsilon'})\geq \rank_L(I^{\epsilon+\epsilon'})\]
\[\rank_L(I)\geq \rank_N(I^{\epsilon'})\geq \rank_M((I^{\epsilon'})^\epsilon)\geq \rank_M(I^{\epsilon'+\epsilon}),\]
hence there is an $(\epsilon+\epsilon')$-erosion between $\rank_M^\Ical$ and $\rank_L^\Ical$, as desired.
\end{proof}

To ensure stability of the generalized rank invariant for $\catP$-modules, we require the family $\Omega$ to satisfy two additional properties:
\begin{definition}\phantomsection\label{def:stability added conditions}
\begin{enumerate}[label=(\roman*),leftmargin=*]
    \item (\textbf{surjectivity}) For all $\epsilon\geq 0$, $\Omega_\epsilon:\catP\to \catP$ must be surjective, and\label{item:general stability 1}
    \item (\textbf{order embedding}) for all $\epsilon\geq 0$ and $p,p'\in \catP$, $p\leq p'\iff \Omega_\epsilon(p)\leq \Omega_\epsilon(p')$. \label{item:general stability 2}
\end{enumerate}
We call a family of superlinear translations $\Omega$ a \textbf{surjective order embedding} if $\Omega$ satisfies properties \ref{item:general stability 1} and \ref{item:general stability 2}.
\end{definition}

\begin{remark}
    Note that property \ref{item:general stability 2} is a strengthening of the condition in Definition \ref{def:superlinear family}, which only enforces the forward direction of the if and only if statement.
\newline \indent   
    Further, note that property \ref{item:general stability 2} implies that $\Omega_\epsilon$ is injective for all $\epsilon\geq 0$.
    This implies that if $\catP$ is finite, properties \ref{item:general stability 1} and \ref{item:general stability 2} imply each other.
    As demonstrated in the examples in Remark \ref{rem:general stability assumptions} below, the two properties are independent of each other when $P$ is infinite.
\end{remark}

For example, if $\Omega$ is the shift by $(\lfloor \epsilon \rfloor,\lfloor \epsilon \rfloor,\ldots \lfloor \epsilon\rfloor)$ in $\Z^n$, or more generally the standard shift by $(\epsilon,\epsilon,\ldots,\epsilon)$ in $\R^d$, then both properties are satisfied.
In this case, we call the family $\Omega$ is a surjective order embedding.

\begin{theorem}\label{thm:gristabilityold}
Fix $\Omega$ be a family of superlinear translations on $\catP$ such that $\Omega$ is a surjective order embedding.
Let $\mathcal{I}\subset \Int(\catP)$ be a collection of intervals closed under $\Omega$-thickenings.
Then for any $\catP$-modules $M$ and $N$:
\begin{equation}\label{eq:stability}
    d_E^\Omega(\rank_M^\Ical,\rank_N^\Ical)\leq \dI^\Omega(M,N).
\end{equation}
\end{theorem}

We omit most of the details of the proof of Theorem \ref{thm:gristabilityold} as it follows the same steps as the proof of Theorem \ref{thm:gristability}, under adjustments to the general setting such as replacing $p+\epsilon$ with $\Omega_\epsilon(p)$.
The assumption that $\Omega$ is a surjective order embedding ensures that the map the map $\alpha':\varprojlim N\vert_{I^\epsilon}\to \varprojlim M\vert_I$ is well-defined, i.e. $\alpha'$ sends sections to sections.

To see that $\alpha':\varprojlim N\vert_{I^\epsilon}\to M\vert_I$ sends a section to a section, let $(\ell_p)_{p\in I^\epsilon}$ be a section in $\varprojlim N\vert_{I^\epsilon}$.
$\alpha'$ sends maps this section to the collection $(\alpha_p(\ell_p))_{\Omega_\epsilon(p)\in I}$.
Since $\Omega_\epsilon$ is surjective, we know that this image $(\alpha_p(\ell_p))$ contains an element for all $p\in I$.
Property \ref{item:general stability 2} alongside naturality of $\alpha$ implies that for all $p\leq p'\in I^\epsilon$ with $\Omega_\epsilon(p)\leq \Omega_\epsilon(p')\in I$, \[\varphi_M(\Omega_\epsilon(p), \Omega_\epsilon(p'))(\alpha_p(\ell_p)) = \alpha_{p'}(\ell_{p'}),\] demonstrating that $(\alpha_p(\ell_p))_{\Omega_\epsilon(p)\in I}$ is indeed a section.

\begin{remark}\label{rem:general stability assumptions}
    \ok{We show that if the family $\Omega$ of Theorem \ref{thm:gristabilityold} lacks either property \ref{item:general stability 1} or property \ref{item:general stability 2}, the inequality presented in Equation (\ref{eq:stability}) cannot be guaranteed.}

    \ok{To see the necessity of property \ref{item:general stability 1},} \ok{
    let $\catP=[0,\infty)\subset \R$ with the usual order.
    Let $\Omega$ be the family of superlinear translations $\Omega_\epsilon:\catP\to \catP$, where for $\epsilon\geq 0$ and $a\in \catP$, $\Omega_\epsilon(a) := a+\epsilon$.
    Clearly, $\Omega$ satisfies property \ref{item:general stability 2}, but not \ref{item:general stability 1}.}

    \ok{Let $M:=\kf_{\catP}$ and $N:=\kf_{(0,\infty)}$.
    We have $\dI^\Omega(M,N) = 0$, as there is an $\Omega_\epsilon$-interleaving for all $\epsilon>0$.
    However, we claim $\dE^\Omega(\rk^{\Int}_M,\rk^{\Int}_N)=\infty$.
    To see this, let $I=[0,1]\subset \catP$.
    Then $\rk^{\Int}_N(I)=0$ because $N(0) = 0$.
    However, for all $\epsilon\geq 0$, $\rk^{\Int}_M(I^\epsilon) = 1$, and so there is no $\epsilon\geq 0$ for which $\rk^{\Int}_M(I^\epsilon)\leq \rk^{\Int}_N(I)$, implying the claim.
    Thus, if $\Omega$ is not surjective, the inequality presented in Equation \ref{eq:stability} cannot be guaranteed.}

 \ok{To see the necessity of property \ref{item:general stability 2},}    
let $\catP$ be the $x$-axis in $\R^2$ union the \ok{negative part of} the $y$-axis, $\{0\}\times (-\infty,0]$.
Let the poset structure on $\catP$ be induced by the standard order on $\R^2$.
For $\epsilon\geq 0$, let $\Omega_\epsilon$ be defined as $\Omega_\epsilon((a, 0)):= (a+\epsilon, 0)$, and for $b<0$, $\Omega_\epsilon((0,b)):= \begin{cases} (0,b+\epsilon), & \epsilon \leq -b\\ (\epsilon + b, 0), & \epsilon > -b\end{cases}$.
It is straightforward to check that $\Omega:=(\Omega_\epsilon)_{\epsilon\geq 0}$ forms a family of superlinear translations on $\catP$.
Further, it is straightforward to verify that $\Omega_\epsilon$ is surjective for all $\epsilon\geq 0$, so property \ref{item:general stability 1} holds.
However, property \ref{item:general stability 2} fails.
For one example, $\Omega_3((-2,0)) = (1,0) \leq (2,0) = \Omega_3((0,-1))$, however $(-2,0)$ and $(0,-1)$ are incomparable.

Let $M$ and $N$ be $\catP$-modules defined as follows. 
$M$ is the the direct sum of interval modules $\kf_\catP\oplus \kf_{(0,0)}$.
Let $N$ consist of the same vector spaces as $M$ pointwisely, but with maps given by, for any $a<0$, $\varphi_N((a,0), (0,0)) = \begin{bmatrix} 1\\ 0\end{bmatrix}$, $\varphi_N((0,a), (0,0)) = \begin{bmatrix} 0\\ 1\end{bmatrix}$, and for any $b>0$, $\varphi_N((0,0), (b,0)) = \begin{bmatrix} 1 & 1\end{bmatrix}$.
All other admissible maps $\varphi_N(a, b):\kf\to \kf$ are the identity.
It is straightforward to check that $N$ is indecomposable, and not interval-decomposable as it has dimension 2 at $(0,0)$, so $M$ and $N$ are non-isomorphic.

We claim that $\dI^\Omega(M,N) = 0$.
Let $\epsilon > 0$.
We claim there exists interleaving maps $\psi_N:N\to M\Omega_\epsilon$ and $\psi_M:M\to N\Omega_\epsilon$.
Define $\psi_N$ and $\psi_M$ pointwise as follows:
\[\psi_N\vert_{N(a)} := \begin{cases}
    \varphi_M(a, \Omega_\epsilon(a)) & a\neq (0,0)\\
    \varphi_N(a, \Omega_\epsilon(a)) & a = (0,0)
\end{cases}\] 
\[\psi_M\vert_{M(a)} := \begin{cases}
    \varphi_N(a, \Omega_\epsilon(a)) & a\neq (0,0)\\
    \varphi_M(a, \Omega_\epsilon(a)) & a = (0,0)
\end{cases}\] 

It is straightforward to check that $\psi_N$ and $\psi_M$ satisfy the all commutative diagrams for the interleaving condition (Definition \ref{def:interleaving distance}) are satisfied.
Hence, there exists an $\epsilon$-interleaving between $M$ and $N$ for all $\epsilon>0$, and so $\dI^\Omega(M,N)=0$.

Lastly, we claim $\dE^\Omega(\rk^\Int_M,\rk^\Int_N) = \infty$.
For any $I\in \Int(\catP)$ such that $\{(0,0)\}\subsetneq I$, it is immediate that $\rk^\Int_M(I) = 1$, since $M$ has as a summand $\kf_\catP$.
Let $I$ be the interval which is the convex hull (within $\catP$) of the points $(-1,0)$ and $(0,-1)$.
We claim $\rk^\Int_N(I) = 0$. 
This follows immediately as $\varprojlim N\vert_I = 0$.
To see this limit is trivial, let $\ell:=(\ell_p)_{p\in I}$ be a section of $N\vert_I$.
Then $\exists \ a,b\in \kf$, with $\ell_{(-1,0)} = a$ and $\ell_{(0,-1)} = b$.
Then by the definition of a section (Convention \ref{con:limit and colimit}), we must have $\ell_{(0,0)} = \varphi_N((-1,0), (0,0))(a) = (a,0)$ and $\ell_{(0,0)} = \varphi_N((0,-1),(0,0))(b) = (0,b)$, which implies $a=b=0$, and thus the only section $\ell$ over $N\vert_I$ has $\ell_p = 0$ for all $p\in I$.

Hence, for all $\epsilon\geq 0$, $\rk^\Int_N(I) = 0< 1=\rk^\Int_M(I^\epsilon)$, and so $\dE^\Omega(\rk^\Int_M,\rk^\Int_N) = \infty$ as claimed.
Thus, if property \ref{item:general stability 2} does not hold, the inequality presented in Equation \ref{eq:stability} cannot be guaranteed.
\end{remark}

\section{Proofs from Section \ref{sec:gri vs. zz}}

\noindent\paragraph{Proof of Proposition \ref{prop:rank via zigzag}} To prove Proposition \ref{prop:rank via zigzag}, we need the following definition and lemmas (which are also used in \cite{dey2022computing}). Recall the construction of the (co)limit from Convention \ref{con:limit and colimit}.

Let $\catP$ be a poset and let $M$ be any $\catP$-module.
Let $p,p'\in \catP$ and let $v_{p}\in M_p$ and $v_{p'}\in M_{p'}$.
We write $v_p\sim v_{p'}$ if $p$ and $p'$ are comparable, and either $v_p$ is mapped to $v_{p'}$ via $\varphi_M(p,p')$ or $v_{p'}$ is mapped to $v_p$ via $\varphi_M(p',p)$.

\begin{definition}
Let $\Gamma:p_0,\ldots,p_k$ be a path in $\catP$. 
A $(k+1)$-tuple $\bv\in \bigoplus_{i=0}^k M_{p_i}$ is called a \textbf{section of $M$ along $\Gamma$} if $\bv_{p_i}\sim\bv_{p_{i+1}}$ for each $i$.
\end{definition} 
Note that $\bv$ is not necessarily a section of the restriction  $M\vert_{\{p_0,\ldots,p_k\}}$  of $M$ to the subposet $\{p_0,\ldots,p_k\}\subset I$ \cite[Example 21]{dey2022computing}. 
Furthermore, $\Gamma$ can contain multiple copies of the same point in $\catP$.

\begin{lemma}\label{lem:path characterization of zero in colim}
Let $p,p'\in \catP$. 
For any vectors $v_{p}\in M_{p}$ and $v_{p'}\in M_{p'}$, it holds that $[v_{p}]=[v_{p'}]$ as elements of\footnote{For simplicity, we write $[v_{p}]$
 and $[v_{p'}]$ instead of $[j_p(v_p)]$ and $[j_q(v_{p'})]$ respectively where $j_p:M_p\rightarrow \bigoplus_{r\in \catP}M_r$  and  $j_{p'}:M_{p'}\rightarrow \bigoplus_{r\in \catP}M_r$ are the canonical inclusion maps.} the colimit $\varinjlim M$ if and only if there exist a path $\Gamma:p=p_0,p_1,\ldots,p_n=p'$ in $\catP$ and a section $\bv$ of $M$ along $\Gamma$ such that $\bv_{p}=v_{p}$ and $\bv_{p'}=v_{p'}$. 
\end{lemma}

\begin{lemma}\label{lem:lower fence and upper fence isomorphism}
Let $I$ be a finite interval of $\Z^2$. 
Let $L:=\overmin(I)$ and $U:=\overmax(I)$.
Given any $I$-module $M$, we have
$\varprojlim M \cong \varprojlim M\vert_{L}$ and $\varinjlim M \cong \varinjlim M\vert_{U}$.
\end{lemma}

The isomorphism $\varprojlim M \cong \varprojlim M\vert_{L}$ in Lemma \ref{lem:lower fence and upper fence isomorphism} is given by the canonical \emph{section extension} $e:\varprojlim M\vert_{L} \rightarrow \varprojlim M$. Namely,
\begin{equation}\label{eq:section extension}
    e: (\bv_{p})_{p\in L}\mapsto (\bw_{p'})_{p'\in \catP },
\end{equation}
where for any $p'\in \catP$, the vector $\bw_{p'}$ is defined as $\varphi_M(p,p')(\bv_p)$ for \emph{any} $p\in L\cap p'^{\downarrow}$; the connectedness of $L\cap p'^{\downarrow}$ guarantees that $\bw_{p'}$ is well-defined. 
Also, if $p'\in L$, then $\bw_{p'}=\bv_{p'}$. The inverse $r:=e^{-1}$ is the canonical section restriction. 
The other isomorphism $\varinjlim M \cong \varinjlim M\vert_{U}$ in Lemma \ref{lem:lower fence and upper fence isomorphism} is given by the map $i:\varinjlim M\vert_U\rightarrow \varinjlim M$ defined by $[v_p]\mapsto [v_p]$ for any $p\in U$ and any $v_p\in M_p$; the fact that this map $i$ is well-defined follows from Lemma \ref{lem:path characterization of zero in colim}.

\begin{proof}[Proof of Proposition \ref{prop:rank via zigzag}]
Let $L$ and $U$ be as in Lemma \ref{lem:lower fence and upper fence isomorphism} above.
Let us define $\xi: \varprojlim M\vert_{L} \rightarrow \varinjlim M\vert_{U}$ by $i^{-1}\circ \psi_M\circ e$. By construction, the following diagram commutes
\begin{equation}\label{eq:diagram commutes}
\begin{tikzcd}
\varprojlim M\vert_{L} \arrow{r}{\xi} \arrow[d, "e", "\cong"']
&\varinjlim M\vert_{U} \arrow[d,"i","\cong"']\\
\varprojlim M \arrow{r}{\psi_M} &\varinjlim M,
\end{tikzcd}
\end{equation}
where $\psi_M$ is the canonical limit-to-colimit map of $M$. Hence we have that $\rank(\psi_M)=\rank(\xi)$. 
Now, it suffices to show:
\[\rank(\xi)=\rank(\psi_{M_{\Gamma}}:\varprojlim M_{\Gamma} \rightarrow \varinjlim M_{\Gamma}).\]
Let us recall the following: let $\alpha:V_1\rightarrow V_2$ and $\beta:V_2\rightarrow V_3$ be two linear maps. 
If $\alpha$ is surjective, then $\rank(\beta\circ \alpha)=\rank(\beta)$.
If $\beta$ is injective, then $\rank(\beta\circ \alpha)=\rank(\alpha)$.  
Therefore, it suffices to show that there exist a surjective linear map $f:\varprojlim M_{\Gamma}\rightarrow \varprojlim M\vert_L$ and an injective linear map $g:\varinjlim M\vert_U \rightarrow \varinjlim M_{\Gamma}$ such that $\psi_{M_{\Gamma}}=g\circ \xi \circ f$. 
We define $f$ as the canonical section restriction $(\bv_q)_{q\in \Gamma}\mapsto (\bv_{q})_{q\in L}$. We define $g$ as the canonical map, i.e. $[v_{q}]\mapsto [v_q]$ for any $q\in U$ and any $v_q\in M_q$. By Lemma \ref{lem:path characterization of zero in colim} and  by construction of $M_{\Gamma}$, the map $g$ is well-defined. 

We now show that $\psi_{M_{\Gamma}}=g\circ \xi \circ f$. Let $\bv:=(\bv_q)_{q\in \Gamma}\in \varprojlim M_{\Gamma}$. 
Then, by definition of $\psi_{M_{\Gamma}}$, the image of $\bv$ via $\psi_{M_{\Gamma}}$ is $[\bv_{q_0}]$ where $q_0\in U$ is defined as in Equation (\ref{eq:zigzag posets2}).
Also, we have \[\bv\stackrel{f}{\longmapsto} (\bv_{q})_{q\in L} \stackrel{\xi}{\longmapsto} [\bv_{q_0}] \big(\in \varinjlim M\vert_U) \stackrel{g}{\longmapsto} [\bv_{q_0}] (\in \varinjlim M_{\Gamma}\big),\] which proves the equality $\psi_{M_{\Gamma}}=g\circ \xi \circ f$.

We claim that $f$ is surjective. Let $r':\varprojlim M \rightarrow \varprojlim M_{\Gamma}$ be the canonical section restriction map $(\bv_q)_{q\in I}\mapsto (\bv_{q})_{q\in \Gamma}$. Then, the restriction $r:\varprojlim M \rightarrow \varprojlim M\vert_L$, can be seen as the composition of two restrictions $r=f\circ r'$.
 Since $r$ is the inverse of the isomorphism $e$ in Equation (\ref{eq:diagram commutes}), $r$ is surjective and thus so is $f$. 

Next we claim that $g$ is injective. Let $i':\varinjlim M_{\Gamma}\rightarrow \varinjlim M$ be defined by $[v_q]\mapsto [v_q]$ for any $q\in \Gamma$ and any $v_q\in M_q$.  By Lemma \ref{lem:path characterization of zero in colim} and  by construction of $M_{\Gamma}$, the map $i'$ is well-defined. Then, for the isomorphism $i$ in Equation (\ref{eq:diagram commutes}), we have $i=i'\circ g$. This implies that $g$ is injective.
\end{proof}
\end{document}